\documentclass{cornell}
\usepackage{amsmath}
\usepackage{amsthm}
\usepackage{amsfonts}
\usepackage{amssymb}
\usepackage{amscd}
\usepackage{graphics}
\usepackage{graphicx}
\usepackage{verbatim}

\theoremstyle{plain}
\newtheorem{lemma}{Lemma}[section]
\newtheorem{theorem}[lemma]{Theorem}

\newtheorem{proposition}[lemma]{Proposition}
\newtheorem{corollary}[lemma]{Corollary}
\newtheorem{corrolary}[lemma]{Corollary}
\theoremstyle{definition}
\newtheorem{definition}[lemma]{Definition}
\newtheorem{example}[lemma]{Example}
\newtheorem{notation}[lemma]{Notation}

\newcommand{\comp}[1]{\operatorname{comp} #1}
\newcommand{\group}[1]{\operatorname{group} #1}
\newcommand{\grad}{\nabla}
\newcommand{\bdry}{\partial}
\newcommand{\del}{\partial}
\newcommand{\Vol}{\operatorname{Vol}}
\newcommand{\Dom}{\operatorname{Dom}}
\newcommand{\Lip}{\operatorname{Lip}}
\newcommand{\diam}{\operatorname{diam}}
\newcommand{\sign}{\operatorname{sign}}

\newcommand{\trace}{\operatorname{Tr}}

\newcommand{\sgn}{\operatorname{sgn}}

\newcommand{\supp}{\operatorname{supp}}

\newcommand{\X}{\mathcal{X}}
\newcommand{\M}{\mathcal{M}}
\newcommand{\E}{\mathcal{E}}
\newcommand{\R}{\mbox{\bf R}}
\newcommand{\Z}{\mbox{\bf Z}           }

\newcommand{\ceil}[1]{\lceil #1 \rceil}
\newcommand{\floor}[1]{\lfloor #1 \rfloor}
\newcommand{\norm}[1]{\lVert #1 \rVert}
\newcommand{\abs}[1]{\lvert #1 \rvert}
\newcommand{\varl}{\ell}
\def\Xint#1{\mathchoice
    {\XXint\displaystyle\textstyle{#1}}%
    {\XXint\textstyle\scriptstyle{#1}}%
    {\XXint\scriptstyle\scriptscriptstyle{#1}}%
    {\XXint\scriptscriptstyle\scriptscriptstyle{#1}}%
    \!\int}
\def\XXint#1#2#3{{\setbox0=\hbox{$#1{#2#3}{\int}$}
     \vcenter{\hbox{$#2#3$}}\kern-.5\wd0}}
\def\dashint{\Xint-}
\newcommand{\aveint}{\dashint}
\newcommand{\coverlap}{C_{Over}}
\newcommand{\csup}{C_{sup}}
\newcommand{\cpon}{C_{P}}
\newcommand{\cover}{C_{N,S}}
\newcommand{\Size}{R}
\newcommand{\CompareXG}{C_{XG}}
\newcommand{\CXG}{C_0}
\newcommand{\F}{\mathcal{F}}
\newcommand{\Const}{C_g}
\newcommand{\Rz}{R_0}

\title{Heat Kernels on Euclidean Complexes}
\author{Melanie Anne Pivarski}

\begin{document}

\maketitle
\makecopyright
\begin{abstract}
In this thesis we describe a type of metric space called an Euclidean
polyhedral complex.  We define a Dirichlet form on it; this is used to
give a corresponding heat kernel.  We provide a uniform small time
Poincar\'{e} inequality for complexes with bounded geometry and use
this to determine uniform small time heat kernel bounds via a theorem
of Sturm.  We then consider such complexes with an underlying finitely
generated group structure.  We use techniques of Saloff-Coste and
Pittet to show a large time asymptotic equivalence for the heat kernel
on the complex and the heat kernel on the group.
\end{abstract}
\begin{biosketch}
Melanie Pivarski was born on August 13, 1977 somewhere in the
outskirts of Pittsburgh to Lynn and Thomas Pivarski.  She grew up with
her parents and two sisters, Kara and Janelle; her grandmother,
Caroline Matovick, lived a few blocks away.  They can all attest to
the fact that yes, Melanie has always talked with her hands.

She attended Colfax Elementary School where she greatly enjoyed
Ms. Kengor's math classes; partly, this was because she could spend
time doing logic puzzles.  She was also involved in Girl Scouts with
Chrissy.  She then attended Springdale High School where she
participated in many different activities, most notably art classes,
Drama Club with Ms. Frauenholz, and Academic Games.

From 1995 through 1999, she attended Carnegie Mellon University, where
she majored in mathematics and minored in computer science.  While
there, she took Prof. Mizel's freshman analysis course out of Apostle.
This class convinced her that she needed to study mathematics.  In her
analysis class she met Helena and Ruth who became her good friends and
study partners.  Computers are a part of the culture at CMU, and so
she found herself in a number of fun computer science classes.  Though
she spent much of her time on math and computers, she found time to
take some ballet classes with her friend Robert and some introductory
Polish classes at the University of Pittsburgh.  Quite significantly,
she met Jim McCann (now Jim McCann Pivarski) during freshman
orientation.  They began dating that fall and were married in June of
1999.

Melanie and Jim moved to Ithaca in the summer of 1999, where they
became graduate students in math and physics respectively.  Melanie
has greatly enjoyed her time in the math department; while there she
was involved in the teaching seminars, the outreach program Expanding
Your Horizons, the women in math potlucks, the 120A Cafe, and the
Gingerbread House contests, specializing in tower constructions.  She
also studied some math.

Outside the math department, Melanie found many things to do.  She's
taken ballet classes and sang in the choir at St. Catherine's for much
of her time here.  She's also been involved in the St Catherine's
young adult group, the PreCana team, and spent a few years helping out
at Loaves and Fishes.  Through all of this, she's met a number of
interesting people and had a variety of experiences.  She considers
herself to be more mature than she's ever been before, and she hopes
to continue growing and learning throughout her life.

In the fall of 2006, Melanie and Jim will move to College Station
where they will be employed as postdocs at Texas A \& M University.
\end{biosketch}
\begin{dedication}
To Friendship! And most especially to Jim!
\end{dedication}
\begin{acknowledgements}
The math department at Cornell is full of wonderful people.  The
community here is splendiferous.  I'm very grateful for the years
spent here, the math learned, and especially the friendships.

Most importantly, I'd like to thank my advisor, Prof. Saloff-Coste,
who gave me an interesting problem to work on.  He's exposed me to
loads of cool mathematics, and he has been essential in my learning
some of it.  He also has a nearly infinite amount of patience, which
comes in quite handy.

I'd also like to thank my committee members, Prof. Gross and
Prof. Strichartz, who also helped me to develop mathematically through
both courses and conversations.  Thanks go to Prof. Fulling as well,
who found a mistake in an early draft of this thesis. 

My mathematical siblings, David Revelle, Lee Gibson, Sharad Goel,
Guan-Yu Chen, Evgueni Klebanov, Pavel Gyra, and Jessica Zuniga, are
all great people.  They've been very encouraging and quite helpful to
me in clarifying my thoughts and definitions.  I've learned a lot in our
group meetings!

Thanks go to Todd Kemp and Treven Wall, who helped me with various
analysis bits, Jim Belk, who gave me a crash course in algebraic
topology one summer, Kristin Camenga who helped me with the geometric
definitions, and Mike Kozdron, who showed me to various latex
commands.  I'd also like to thank Josh Bowman, Jonathan Needleman,
Robyn Miller, Mia Minnes, and Brigitta Vermesi who, along with many of
the other folk mentioned above, helped me to organize my thoughts into
some kind of presentable form through various conversations.  Maria
Belk and Maria Terrell should also be thanked for their encouragement
and general good advice on how to be a graduate student.
\end{acknowledgements}

\tableofcontents
\figurelistpage
\symlist
\begin{itemize}
\item $X$ the complex
\item $X^{(k)}$ the $k$-skeleton of $X$
\item $\gamma$ a path in $X$
\item $L(\gamma)$ length of $\gamma$
\item $d_X(\cdot,\cdot)$ distance in $X$ induced by the Euclidean metric
\item $d_{X^{(i)}}(\cdot,\cdot)$, $d_{i}(\cdot,\cdot)$ distance in
$X^{(i)}$ induced by the Euclidean metric
\item $\mu$ measure on $X$
\item $E(\cdot,\cdot)$ energy form constructed via $\Gamma$-limit
\item $\E(\cdot,\cdot)$ energy form constructed via gradients
\item $\Delta$, $\Delta^k$  Laplacian for $X$,  $X^{(k)}$
\item $\Delta^{\Omega}$ Laplacian for $\Omega \subset X$ with
Dirichlet boundary condition
\item $\grad f$, $F$ gradient of $f$ 
\item $d\mu(x)$, $dx$ equivalent ways of writing the differential form
\item $\Lip(X)$ the set of Lipschitz functions on $X$
\item $W^{1,p}(X)$ the Sobolev space on $X$ of functions in $L^p(X)$ with first derivatives in $L^p$ 
\item $C_0^{\Lip}(X)$ the set of compactly supported Lipschitz functions on $X$
\item $\aveint_B$ average integral over the set $B$
\item $I_{\Omega}$ indicator function on $\Omega$ (1 if in $\Omega$, 0 if not)
\item $f_E$ average value of $f$ on the set $E$
\item $\bdry E$ the boundary of the set $E$
\item $S^{(k)}$ the $k$-sphere
\item $W_k$, $W_{j,k}$ wedges of a ball in $X$
\item $N(j)$ the list of indices of faces adjacent to $W_j$ including $j$
\item $M$ degree bound on $X$
\item $\varl$ lower bound on edge lengths of $X$
\item $\alpha$ smallest interior angle in $X$
\item $\kappa$ constant related to multiples of radii 
\item $R_0$ bound on radius defined in terms of $\kappa$
\item $C_{Weak}$ constant in weak Poincar\'{e} inequality
\item $\cpon$ constant for the Poincar\'{e} inequality
\item $\F$ collection of balls in the Whitney cover
\item $r_B$ radius of the ball $B$
\item $B_z$ the central ball in the Whitney cover
\item $\F(B)$ a string of balls that takes $B_z$ to $B$
\item $h_t(x,y)$, $h^k_t(x,y)$ heat kernel on $X$, $X^{(k)}$
\item $H_t$ heat semigroup on $X$
\item $\tau$ exit time for Brownian motion on a set
\item $X_t$ random variable for location of a Brownian motion in a subset of $X$
\item $h^{\Omega}_t(x,y)$ heat kernel on $\Omega \subset X$
\item $H_t^{\Omega}$ heat semigroup on $\Omega \subset X$
\item $p_t(x,y)$, $p_t^A(x,y)$ heat kernel on $G$, $A \subset G$
\item $K^n$, $K^n_A$ $n$ step transition operator on $G$, $A \subset G$
\item $\lambda_{\Omega}(i)$ ith eigenvalue for $H_1^{\Omega}$
\item $\beta_A(i)$ ith eigenvalue for $K^1_A$
\item $\trace{(K^n)}$ trace of $K^n$
\item $G$ group
\item $S$ generating set for $G$
\item $Y$ compact subset of $X$ with the property $X/Y=G$
\item $|g|$ word length of $g \in G$
\item $d_G(\cdot,\cdot)$ distance in $G$ with respect to word length
\item $d_Y(\cdot,\cdot)$ distance in $Y$ based on Euclidean paths in $Y$.
\item $\diam(Y)$ diameter of $Y$ with respect to $d_Y$
\item $Y^{(i)}$ the $i$ skeleton of $Y$; $X^{(i)} \cap Y$
\item $|A|$, $\#A$ number of elements in $A$
\item $\Vol$, $\Vol_X$, $\Vol_G$ volume (with respect to $X$, $G$)
\item $B_r$, $B_X(r)$, $B_G(r)$ ball of radius $r$ (in $X$, $G$)
\item $B_X(x,r)$, $B_G(g,r)$ ball in $X$ ($G$) with radius $r$
centered at $x$ ($g$)
\item $E(f,f)$ energy form; for $G$ this is $\frac{1}{|S|}\sum_{g \in
G}\sum_{s \in S} |f(g) -f(gs)|^2$, for $X$ this is $- \langle \grad f,
\grad f \rangle$
\item $|\grad f(x)|$ length of gradient; on $G$ this is
$\sqrt{\frac{1}{|S|}\sum_{s \in S} |f(x) -f(xs)|^2}$.
\item $||f||_{p,A}$ the $L^p$ norm restricted to a subset $A$; for $G$
this is $(\sum_{x\in A} |f(x)|^p)^{1/p}$, for $X$ this is $(\int_{A}
|f(x)|^p dx)^{1/p}$
\item $\{ \gamma_i\}_{i=1}^N$ centers of the balls of radius $\delta$
covering $Y$
\item $\{ g\gamma_i\}_{i=1..N; g \in G}$ centers of the balls of
radius $\delta$ covering $X$
\item $\cover$ maximum number of balls overlapping a point in $X$
\item $\group f$ a new function from $(G,N) \rightarrow R$ defined to
be $\dashint_{B_X(g\gamma_i,\delta)} f(x) dx$
\item $\CompareXG$, $\CXG$ constants used to compare metrics in $G$ and $X$
\item $C_H$ constant for the Harnack inequality
\item $\chi$,  $\chi_g$ a smooth function on $X$, $\chi$ translated by $g$
\item $\csup$ constant used to bound the support of $\chi$
\item $\coverlap$ constant bound on the maximum number of $\chi_g$
supported at any point in X
\item $\Const$ bound on the magnitude of the gradient of $\chi$
\item $\comp{f(x)}$ a new function from $X \rightarrow R$ defined by
$\sum_{g \in G} f(g) \chi_g(x)$
\item $C_{grad}$ constant bound comparing norms of gradients 
\item $\floor{x}$ the largest integer less than or equal to $x$
\item $\ceil{x}$ the smallest integer greater than or equal to $x$ 
\item $C_0(\Omega)$ continuous funcions which are compactly supported
in $\Omega$
\item $U(A)$ subset of $X$ depending on $A$
\end{itemize}
\pagebreak

\normalspacing
\pagestyle{cornell}
\pagenumbering{arabic}
\setcounter{page}{1}

\chapter{Setup for the Complexes}
\section{Introduction}
  We will study how local and global geometries affect heat kernels on a
set of metric spaces called Euclidean polyhedral complexes.

Euclidean complexes are formed by taking a collection of $n$
dimensional convex polytopes and joining them along $n-1$ dimensional
faces.  Within each polytope, we will have the same metric structure
as $R^n$.  When we join them, we will glue the faces of two polytopes
together so that points on one face are identified with points on the
other face, and the metrics on those faces are preserved.  We will
require that these structures have a countable number of polytopes,
are locally finite, and have lower bound on the interior angles and
edge lengths.  The complex formed by looking at $k$ dimensional faces
is called the $k$-skeleton.  For instance, the 0-skeleton is set of
vertices.  A 1-skeleton is a graph where the space includes both
vertices and points on the edges; sometimes this is called a metric
graph \cite{Kuc}.  Note that we can triangulate any convex polytope to
obtain a collection of simplices, and so this structure is essentially
equivalent to looking at a simplicial complex.

\begin{figure}[h]
\centering
\includegraphics[angle=0,width=1in]{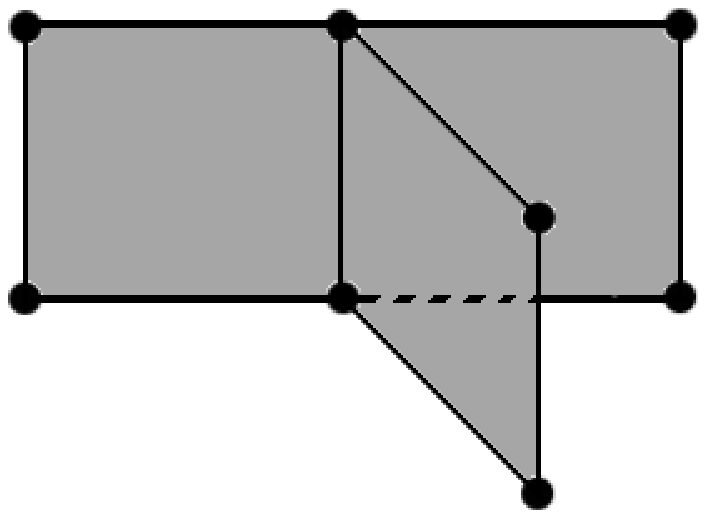}
\hspace{.5in}
\includegraphics[angle=0,width=1in]{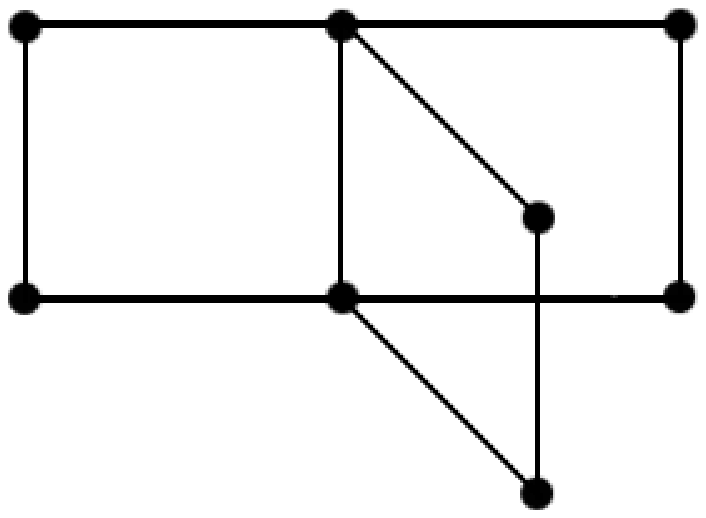}
\hspace{.5in}
\includegraphics[angle=0,width=1in]{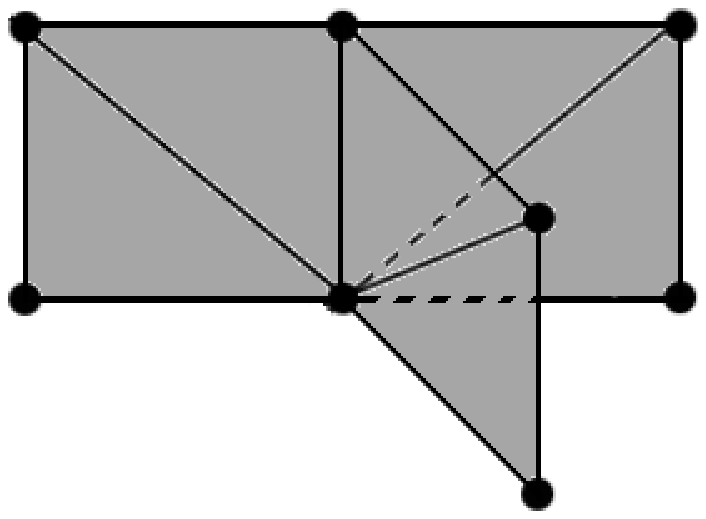}
\caption{Example of a 2 dimensional Euclidean Complex (left),
its 1-skeleton (center), one possible triangulation (right).}
\end{figure}
Let $h_t^k(x,y)$ be the heat kernel on the $k$-skeleton.  This is the
fundamental solution to the heat equation $\partial_t u - \Delta u =0$
on the $k$-skeleton. It can be used to describe the probability that
we travel from $x$ to $y$ in time $t$ when our movement is restricted
to the $k$-skeleton.  
\begin{theorem}
For $X$ a uniformly locally finite Euclidean complex of dimension $n$
whose interior angles and edge lengths are bounded below.  Fix $T \in
(0,\infty)$.  There exist $c, C\in (0,\infty)$ such that for any $x
\in X$ and $t < T$ we have:
\begin{eqnarray*}
\frac{c}{t^{k/2}} \le h_t^k(x,x) \le \frac{C}{t^{k/2}}.
\end{eqnarray*}
\end{theorem}

Note that this claims that the heat kernel on the $k$-skeleton behaves, up
to a constant that is independent of where in $X$ we are, like the heat
kernel on $R^k$ asymptotically when $t \rightarrow 0$.  The local behavior
reflects the local geometry and structure of our space. 

Theorems in Sturm \cite{SturmDiffHK} can be applied to Euclidean complexes to
show that on any compact subset of $X^{(k)}$, the heat kernel is
locally like the one on $R^k$, with constants that depend on the
choice of compact subset.  The essential difference in our
theorem is that the constants are uniform throughout the entire
complex.

An interesting example of these complexes comes from biology. In a
paper by Billera, Holmes, and Vogtmann \cite{BilleraHolmesVogtmann}
they describe a way of classifying distances between phylogenetic
trees, which are trees that describe evolution of species.  One can
form an Euclidean complex, where each of the faces corresponds to a
different tree, and one moves through the points in the face by
changing the edge lengths in the tree.  One can then consider
probability distributions on this space to determine likely genetic
ancestry.

Euclidean complexes are also examples of fractal ``blow-ups'', which
are infinite fractals that are locally nice but globally have a
structure with repetition.  See Kigami \cite{KigamiFrac} for a
description of these fractals.  In this setting, our small time
asymptotic estimates apply.  Note that these examples need not be
compact.  In \cite{BarlowKumagai}, Barlow and Kumagai studied the
small time asymptotic of heat kernels for compact self-similar
sets.

Another collection of examples can be found by considering metric
spaces, $X$, which are acted upon by a finitely generated group, $G$
of isometries.  When we take the space and mod out by that group, we
obtain a compact set $Y = X/G$. When $Y$ can be expressed as a finite
Euclidean complex, then $X$ is an Euclidean complex as well.  Note
that the $k$-skeleton of $Y$ will be the $(k$-skeleton of $X)/G$.  A
simple example of this is $X=\R^2$, $Y=$ the unit square, and
$G=\Z^2$.  A more interesting example occurs when $G$ is the free
group; there the space is globally hyperbolic, but locally Euclidean.
With this added group structure, we can describe the large time
behavior of the heat kernel.  We write the heat kernel on a group as
$p_t(\cdot,\cdot)$.

\begin{proposition}
Let $X$ be a locally finite countable Euclidean complex of dimension
$n$ and let $G$ be finitely generated group $G$.  If $X/G$ is a
complex comprised of a finite number of polytopes with Euclidean
metric, we have:
\begin{eqnarray*}
p_t(x,x) \simeq h_t(x,x) \mbox{ as } t \rightarrow \infty .
\end{eqnarray*}
\end{proposition}
 Our main result says that, up to a constant, the heat kernel will
behave the same asymptotically as $t \rightarrow \infty$ on both the
group and the complex.  By transitivity, it will behave the same
asymptotically regardless of which $k$-skeleton we consider.

This theorem relates to a paper of Pittet and Saloff-Coste
\cite{LSCP}.  They show that a manifold $M$ which has a finitely
generated group of isometries $G$ satisfies $sup_x h^M_t(x,x)
\simeq h^G_t(e,e)$ for large $t$.  

In chapter one, we describe our set-up.  We provide definitions for
the complex and skeletons and then define an energy form and a
Laplacian on them.  In chapter two, we will prove the initial theorem
by first showing that a series of Poincar\'{e} inequalities hold,
starting with one for balls where the radius of the ball depends on
the center and generalizing until the result is uniform in space.  In
chapter three, we apply these inequalities to a result of Sturm
\cite{Sturm} to yield a small time on diagonal heat kernel asymptotic
with a uniform constant.  We also provide off diagonal estimates with
constants that depend on $d(x,y)$, but not on where $x$ and $y$ are
located.  We give several examples of heat kernels.  In chapter four,
we consider complexes with underlying group structure. We describe how
to compare metrics on the complex and those on the underlying group,
as well as how to switch from a function on a group to one on a
complex and vice versa.  We then use the metric comparison as well as
our small time Poincar\'{e} inequality to compare norms of functions
on complexes and their group counterparts.  In chapter five, we
consider heat kernels on the group and the complex.  We split into two
cases; nonamenable groups, which have exponentially fast heat kernel
decay, and amenable groups.  For the amenable groups, we look at heat
kernels restricted to subsets of our space, and then take a F\o lner
sequence to limit to a bound on the heat kernels themselves.  In this
way, we prove the second theorem.
\section{Geometry of the Complexes}
  We will take our definitions of polytopes and polyhedral sets from
Gr\"{u}nbaum's Convex Polytopes \cite{Grunbaum}.
\begin{definition}
A polyhedron $K$ is a subset of $R^n$ formed by intersecting a finite
family of closed half spaces of $R^n$.  Note that this can be an
unbounded set, but it will be convex.
\end{definition}
\begin{definition}
A set $F$ is a face of $K$ if $F=\emptyset$, $F=K$, or if $F=H\cap K$
where $H$ is a supporting hyperplane of $K$.  $H$ is a supporting
hyperplane of $K$ if $H\cap K \ne \emptyset$ and $H$ does not cut $K$
into two pieces.
\end{definition} 
\begin{definition}
A point $x\in K$ is an extreme point of a set $K$ if the only $y,z \in
K$ which are solutions to $x = \lambda y + (1-\lambda) z$ for some
$\lambda \in (0,1)$ are $x=y=z$.  That is, $x$ cannot be expressed as
a convex combination of points in $K - \{x\}$.  Note that the extreme
points of $K$ are faces for $K$.
\end{definition}
\begin{definition}
A polytope is a compact convex subset of $R^n$ which has a finite set
of extreme points.  This is equivalent to saying it is a bounded
polyhedron. 
\end{definition}
\begin{definition}
A polyhedral complex $X$ is the union of a collection, $\X$, of convex
polyhedra which are joined along lower dimensional faces.  By this we
mean that for any two distinct polyhedra $P_1, P_2 \in \X$, 
\begin{itemize}
\item $P_1 \cap P_2$ is a polyhedron whose dimension satisfies \\
$\dim(P_1 \cap P_2) < \max(\dim(P_1),\dim(P_2))$ and 
\item $P_1 \cap P_2$ is a face of both $P_1$ and $P_2$.  We allow this 
face to be the empty set.
\end{itemize}  
\end{definition}
We do not have a specific embedding for the complex, $X$; however, we
require each polyhedra to have a metric which is consistent with that
of its faces.

Note that this definition implied $P_1\cap P_2$ is a connected set.
This rules out expressing a circle as two edges whose ends are joined,
but it allows us to write it as a triangle of three edges.  This is
not very restrictive, as we can triangulate the polyhedra in order to
form a complex which avoids the overlap.

  Simplicial complexes are an example of a polyhedral complex; the
difference here is that we allow greater numbers of sides.  Note that
we allow infinite polyhedra, not just finite polytopes.
\begin{definition}
Define a $p$-skeleton, $X^{(p)}$, for $0\le p\le \dim X$ to be the union
of all faces of dimension $p$ or smaller.  Note that this is also a
polyhedral complex.
\end{definition}
\begin{definition}
A maximal polyhedron is a polyhedron that is not a proper face of any
other polyhedron.  The set of maximal polyhedra of $X$ is denoted
$\X_{MAX}$.  We say $X$ is dimensionally homogeneous if all of its
maximal polyhedra have dimension $n$.  Note that in combinatorics
literature this is called pure.
\end{definition}
\begin{definition}
$X$ is locally (n-1)-chainable if for every connected open set $U
\subset X$, $U-X^{(n-2)}$ is also connected.  For a dimensionally
homogeneous complex $X$ this is equivalent to the property that any
two $n$ dimensional polyhedra that share a lower dimensional face can
be joined by a chain of contiguous $(n-1)$ or $n$ dimensional
polyhedra containing that face.
\end{definition}
\begin{definition}
We call $X$ admissible if it is both dimensionally homogeneous
and in some triangulation $X$ is locally (n-1)-chainable. 
\end{definition}
\begin{figure}[h]
\centering
       \includegraphics[angle=0,width=1in]{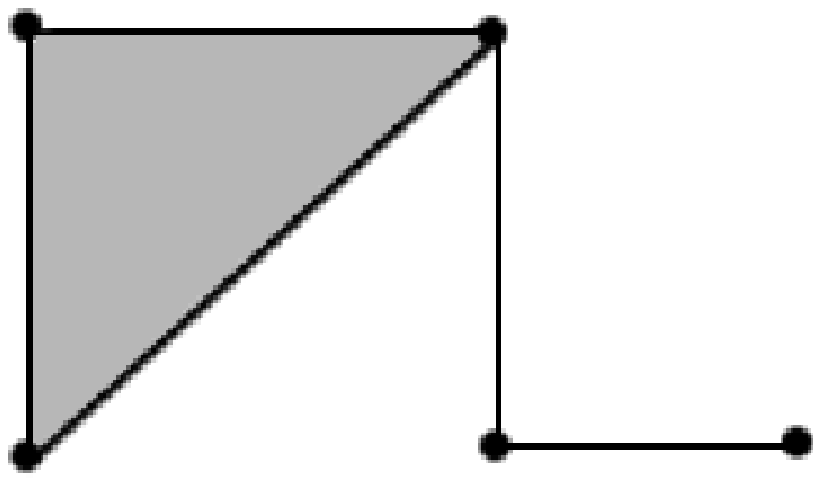}
\hspace{.5in}
       \includegraphics[angle=0,width=1in]{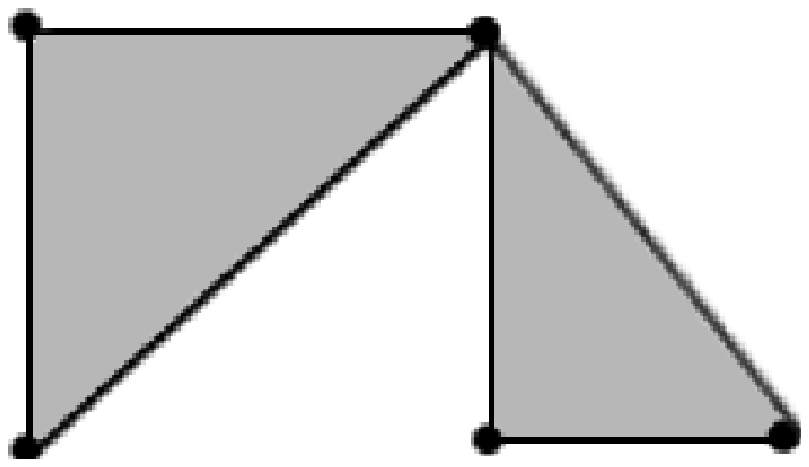}
\hspace{.5in}
       \includegraphics[angle=0,width=1in]{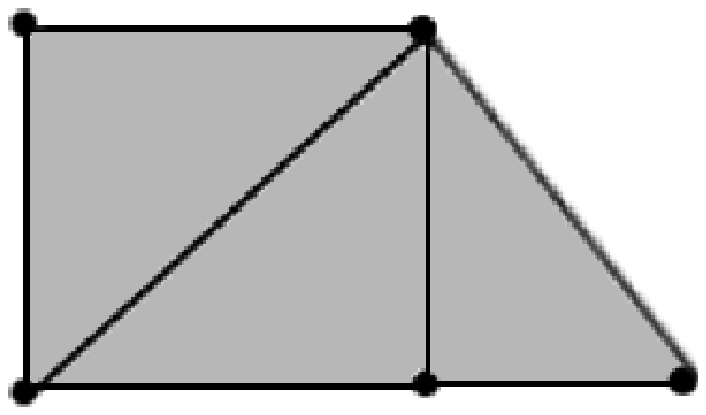}
\caption{Examples of a complex which is not dimensionally homogeneous (left), 
one which is not 1-chainable (center), and one which is admissible (right).}
\end{figure}
We will be working with connected admissible complexes, and for our
purposes, we'd like to consider polyhedra that have an Euclidean
metric. Let $X$ be an n-dimensional complex.  When two polyhedra share
a face, we require these metrics to coincide.  For points $x$ and $y$
in different polyhedra, we define the distance as follows. 
\begin{definition}
 Consider the set of paths connecting $x$ to $y$ which consist of a
finite number of line segments.  We can label each of these by the
points it crosses in the $(n-1)$ skeleton. Set $\gamma
=\{x=x_0,x_1,x_2,..,x_k=y\}$ where $x_i \in$ $(n-1)$-skeleton for
$i=1..k-1$, and $x_i$, $x_{i+1}$ are both in the closure of the same
maximal polyhedron.  Then set $L(\gamma) = \sum_{i=1}^k
d(x_{i-1},x_i)$.  We define the distance between points in different
polyhedra to be $d(x,y) = \inf_{\gamma} L(\gamma)$.
\end{definition}

Essentially, we are splitting the path into pieces, and letting
the lengths of those pieces inside of the simplices be the standard
lengths in $R^n$.  Since our complex is created using closed
polyhedra, if the geometry of the polyhedra is bounded, the inf will
be realized.  This will give us a length space; ie, one in which
distances are realized by geodesics in the space.  Discussions of
length spaces and other metric measure spaces can be found in Heinonen
\cite{Hein} and Burago, Burago, and Ivanov \cite{BuragoI}.

\begin{definition}
Let $X=\cup_i P_i$, where the $P_i$ are the maximal polyhedra. We will
set the measure of $A$, a Borel subset of $X$, to be $\mu(A) =\sum_i \mu_i(A
\cap P_i)$ where $\mu_i$ is the Lebesgue measure on $P_i$.  
\end{definition}
Notice that the measure within the interior of maximal polyhedra is
the same as Lebesgue measure on $R^n$.  This means that locally we
will have all of the structure of $R^n$; in particular, we will have
volume doubling for balls contained in the interior of the maximal
polyhedra.  Since our complex is locally finite, volume doubling will
hold locally for all points in the complex.

\begin{definition}
An admissible polyhedral complex, $X$, equipped with distance,
$d(\cdot,\cdot)$ and measure $\mu$ is called an Euclidean polyhedral complex.  
\end{definition}
For brevity, we will often call this an Euclidean complex.
A book which describes these complexes is Harmonic Maps Between
Riemannian Polyhedra \cite{EellF}.  In it, the authors define these
structures with a Riemannian metric and provide analytic results on
both the complexes and functions whose domain and range are both
complexes.

\section{Analysis on the Complexes}
\subsection{The Dirichlet Form}
Now that we've defined the space geometrically, we will define a
Dirichlet form whose core consists of compactly supported Lipschitz
functions.

\begin{definition}
A function $f$ on a metric space $X$ is called $L$-Lipschitz
(alternately, Lipschitz) if there exists a constant $L \ge 0$ so that
$d(f(x),f(y)) \le L d_X(x,y)$ for all $x$ and $y$ in $X$.  The space
of Lipschitz functions is denoted $\Lip(X)$.  The space of compactly
supported Lipschitz functions is denoted $C_0^{\Lip}(X)$. 
\end{definition}
Note that Lipschitz functions are continuous.  By theorem 4 in section
5.8 of \cite{Evans}, for each $B_{\epsilon}(x) \subset X-X^{(n-1)}$
and $f \in C_0^{\Lip}\left(X\right)$, $f$ restricted to
$B_{\epsilon}(x)$ is in the Sobolev space
$W^{1,\infty}\left(B_{\epsilon}(x)\right)$.  This tells us that $f$
has a gradient almost everywhere in $X-X^{(n-1)}$.  Since
$\mu(X^{(n-1)})=0$, $f$ has a gradient for almost every $x$ in $X$.

We would like an energy form that acts like $E(u,v) =\int_X \langle
\grad u,\grad v \rangle d\mu$ with domain $F$ to define our operator
$\Delta$ with domain $\Dom(\Delta)$. We can define this in a very
general manner which does not depend on the local Euclidean structure
by following a paper of Sturm \cite{SturmDif}.  We can also define it
in a more straightforward manner which uses the geometry of $X$.  We
do both, and then show that they coincide.

Sturm assumes that the space $(X,d)$ is a locally compact separable
metric space, $\mu$ is a Radon measure on $X$, and that $\mu(U)>0$ for
every nonempty open set $U \subset X$.  These assumptions hold both in
our space, $X$, and on the skeletons, $X^{(k)}$.  We begin by
approximating $E$ with a form $E^r$ defined to be:
\begin{eqnarray*}
E^r(u,v) := \int_X \int_{B(x,r)-\{x\}}
 \frac{(u(x)-u(y))(v(x)-v(y))}{d^2(x,y)} 
\frac{2 N d\mu(y)d\mu(x)}{\mu(B_r(x)) + \mu(B_r(y))}
\end{eqnarray*} 
for $u,v \in \Lip(X)$ where N is the local dimension.  Note that whenever $x$
  is in a region locally like $R^n$, we have
\begin{eqnarray*}
\mathop{\lim}_{r\rightarrow 0} \frac{N}{\mu(B_r(x))} \int_{B(x,r)-\{x\}}
 \frac{(u(x)-u(y))^2}{d^2(x,y)} d\mu(y) = |\grad u(x)|^2 ,
\end{eqnarray*}
 and so this form looks very similar to $E(u,u)=\int_X|\grad u|^2 dx$.

This form with domain $C_0^{\Lip}(X)$ is closable and symmetric on
$L^2(X)$, and its closure has core $C_0^{\Lip}(X)$.  See Lemma 3.1 in
\cite{SturmDif}. One can take limits of these operators in the
following way.  The $\Gamma$-limit of the $E^{r_n}$ is defined to be
the limit that occurs when the following lim sup and lim inf are equal
for all $u\in L^2(X,m)$.  See Dal Maso\cite{DalMaso} for a thorough
introduction.
\begin{align*}
\Gamma-\mathop{\lim \sup}_{n \rightarrow \infty} E^{r_n}(u,u) 
&:= \mathop{\lim}_{\alpha \rightarrow 0} 
\mathop{\lim \sup}_{n \rightarrow \infty} 
\mathop{\mathop{\inf}_{v \in L^2(X)}}_{||u-v||\le \alpha} E^{r_n}(v,v)
\\
\Gamma-\mathop{\lim \inf}_{n \rightarrow \infty} E^{r_n}(u,u) 
&:= \mathop{\lim}_{\alpha \rightarrow 0} 
\mathop{\lim \inf}_{n \rightarrow \infty} 
\mathop{\mathop{\inf}_{v \in L^2(X)}}_{||u-v||\le \alpha} E^{r_n}(v,v).
\end{align*}
For any sequence $\{E^{r_n}\}$ of these operators with $r_n
\rightarrow 0$ , there is a subsequence $\{r_{n'}\}$ so that the
$\Gamma$-limit of $E^{r_n'}$ exists by Lemma 4.4 in
\cite{SturmDif}.  These lemmas are put together into a theorem
(5.5 in \cite{SturmDif}) that tells us that this limit, $E^0$,
with domain $C_0^{\Lip}(X)$ is a closable and symmetric form, and its
closure, $(E,F)$, is a strongly local regular Dirichlet form on
$L^2(X,m)$ with core $C_0^{\Lip}(X)$.

Alternately, we can define the energy form using the structure of the
space.  We set $\E(\cdot,\cdot)$ to the following for $f \in
C_0^{\Lip}(X)$:
\begin{eqnarray*}
\E(f,f) = \sum_{X_M \in \M} \int_{X_M} |\grad f|^2 d\mu(x).
\end{eqnarray*}
\begin{lemma}
$\E(\cdot,\cdot)$ is a closable form.  That is, for any sequence \\ $\{ f_n
\}_{n=1}^{\infty} \subset C_0^{\Lip}\left(X\right)$ that converges to
0 in $L^2(X)$ and is Cauchy in $|| \cdot ||_2 + \E(\cdot,\cdot)$ we have
$\lim_{n \rightarrow \infty} \E(f_n,f_n) =0$.
\end{lemma}
\begin{proof}
To show this, we will first look at what happens on one fixed
polyhedron, and then look at what happens on a complex.  Let $X_M$ be
a maximal polyhedron.  Since $\{ f_n \}_{n=1}^{\infty}$ is Cauchy in
the norm, we have
\begin{eqnarray*}
\lim_{m,n \rightarrow \infty} 
\left( \int_{X_M} (f_n -f_m)^2 d\mu \right)^{\frac{1}{2}} +
\left( \int_{X_M} (\grad f_n - \grad f_m)^2 d\mu \right)^{\frac{1}{2}}
= 0.
\end{eqnarray*}
  This gives us two functions, $f$ and $F$ which are the 
limits of $f_n$ and $\grad f_n$ respectively.  We have $f=0$ by
assumption.  We need to show that $F=0$.  For almost every $x,y \in X_M$
and line $\gamma_{x \sim y}$ in $X_M$ we have 
\begin{eqnarray*}
\int_{\gamma_{x \sim y}} \grad f_n
d\mu = f_n(y) -f_n(x).
\end{eqnarray*}
Then we can take the limit as $n$ goes to infinity to get
\begin{eqnarray*}
\lim_{n \rightarrow \infty}
\int_{\gamma_{x \sim y}} \grad f_n d\mu =\lim_{n \rightarrow \infty} f_n(y)
-f_n(x) =0.
\end{eqnarray*}
  This gives us $\lim_{n \rightarrow \infty} \grad f_n(x) =0$
 for almost every $x \in X_M$.  

Since the choice of $X_M$ was arbitrary, this shows $\lim_{n
 \rightarrow \infty} \grad f_n(x) =0$ for almost every $x \in X$.

Showing $L^2$ convergence is a bit trickier, as we need to show that
 we can interchange the limit with the sum over the maximal polyhedra.
 We can do this for $|\grad f_n - \grad f_m|$ by Fatou's Lemma.
\begin{eqnarray*}
\lim_{n \rightarrow \infty} \sum_{X_M \in \M} \int_{X_M} |\grad f_n|^2 d\mu
&=& \lim_{n \rightarrow \infty} \sum_{X_M \in \M} \int_{X_M} 
      |\grad f_n -\lim_{m \rightarrow \infty} \grad f_m|^2 d\mu
\\ &=& \lim_{n \rightarrow \infty} \sum_{X_M \in \M} \int_{X_M} 
    \lim_{m \rightarrow \infty}  |\grad f_n - \grad f_m|^2 d\mu
\\ &\le& \lim_{n \rightarrow \infty}  
    \lim_{m \rightarrow \infty}  \sum_{X_M \in \M} \int_{X_M} 
     |\grad f_n - \grad f_m|^2 d\mu
\\ &=&0.
\end{eqnarray*}

This tells us that the form is closable.
\end{proof}
We will show that the two energy forms are the same.  To do this, we
show that they are the same on the core $C_0^{\Lip}\left(X\right)$;
this gives equality on the domain.  
\begin{lemma}
Each function $f \in C_0^{\Lip}\left(X\right)$ satisfies
$E(f,f)=\E(f,f)$.
\end{lemma}
\begin{proof}
We can write $X$ as $(X-X^{(n-1)}) \cup X^{(n-1)}$; this is a
collection of maximal polyhedra and a set of measure 0.  The interior
of the maximal polyhedra is a Riemannian manifold without boundary.
$X$ is also a locally compact length space, and so it satisfies the
conditions of example 4G in \cite{SturmDiffHK}.  This implies it has the
strong measure contraction property with an exceptional set.
Corollary 5.7 in \cite{SturmDiffHK} says that this then has $E(f,f)
=\E(f,f)$ for each $f \in C_0^{\Lip}(X)$.  The equality is shown by
approximating the forms using an increasing sequence of open subsets
which limit to $X-X^{(n-1)}$.  As $C_0^{\Lip}(X)$ is a core for both
$E$ and $\E$, the Dirichlet forms are the same.
\end{proof}

We will explain more clearly where the domain of this operator lies.
The domain is the closure of $C_0^{\Lip}(X)$ in the $W^{1,2}(X)$ norm.
This domain is a subset of the set of functions which are in $W^{1,2}$
of the interiors of the maximal polyhedra.
\begin{lemma}
For $\E$, 
\begin{align*}
\overline{C_0^{\Lip}(X)} \subset \overline{C(X) \cap
\left(\cup_{X_M \in \X_{MAX}}\bigoplus W^{1,2}(X_M^o) \right)}
\end{align*}
 where the closure is taken with respect to the $W^{1,2}$ norm,
 $||\cdot||_2 + \E(\cdot,\cdot)$.  $X_M^o$ denotes the interior of
 $X_M$.
\end{lemma}
\begin{proof}
First note that $C_0^{\Lip}(X) \subset C(X)$.  For any $f
\in C_0^{\Lip}(X)$, we have the compact subset $Y= \supp(f)$.  Then $f$ 
restricted to $X_M^o$ will be in $W^{1,2}(X_M^o)$, since $||\grad
f||_{2,X_M^o} \le ||\grad f||_{\infty,X_M} \mu(X_M^o \cap \supp(f))$.
This tells us 
\begin{eqnarray*}
f \in \cup_{X_M \in \X_{MAX}}\bigoplus W^{1,2}(X_M^o).
\end{eqnarray*}
We now have a containment without the closures:
\begin{align*}
C_0^{\Lip}(X) \subset 
C(X) \cap \left( \cup_{X_M \in \X_{MAX}}\bigoplus W^{1,2}(X_M^o) \right).
\end{align*}
As we then close both sides with respect to the same norm, we have:
\begin{align*}
\overline{C_0^{\Lip}(X)} \subset 
\overline{C(X) \cap \left( \cup_{X_M \in \X_{MAX}}\bigoplus W^{1,2}(X_M^o) 
                    \right)}.
\end{align*}
\end{proof}

\subsection{The Laplacian}
The Dirichlet form uniquely determines a positive self-adjoint
operator \\ $\{ \Delta, \Dom(\Delta) \}$ on $L^2$ where
$F=\Dom(\Delta^{\frac{1}{2}})$ and $E(u,v) = (u,\Delta v)$ for all
$u\in F$ and $v \in \Dom(\Delta)$.  This is done by defining a
collection of quadratic forms, \\ $E_{\alpha}(u,v):=E(u,v) + \alpha(u,v)$
for $\alpha >0$.  Then, by the Reisz representation theorem, there
will be an operator $G_{\alpha}$ so that $E_{\alpha}(G_{\alpha}u,v) =
(u,v)$ for any u,v in $\Dom(E)$.  The set of these operators forms a
$C_0$ resolvent.  One can look at inverses, $G_{\alpha}^{-1}$ on the
image of $G_{\alpha}$.  We can then consider $\Delta =
G_{\alpha}^{-1}- \alpha$ on the space $G_{\alpha}(\Dom(E))$.  One can
show that this definition is independent of $\alpha$.  The domain of
the operator is $\Dom(\Delta)=G_{\alpha}(\Dom(E))$.  It's difficult to
explicitly state exactly which functions are in $\Dom(\Delta)$, but
the domain is dense in $L^2(X)$.  See Fukushima et al \cite{FOT} for
the full argument; a fine summary of this is done in Todd Kemp's
lecture notes \cite{Kemp}.

Note that this set-up will work on each of the skeletons, and so we
can use it to define a different Laplacian on each of them.  When we
define the $E^r$ on a k-skeleton, $X^{(k)}$, we'll set $N=k$,
integrate over $X^{(k)}$, and let m be a k-dimensional measure.  This
technique will define $\Delta_k$ on a dense subset of $L^2(X^{(k)})$.

In the one dimensional case, this Laplacian gives us a structure
called a quantum graph.  Here, the functions in the domain of the
Laplacian should be continuous and the inward pointing derivatives
should sum to zero at each vertex.  This is known as a Kirchoff
condition.  A nice introduction to these graphs and their spectra as
well as a wide variety of references to the literature on them can be
found in Kuchment \cite{Kuc}.

In the two dimensional case, this operator is related to results in a
paper of Brin and Kifer \cite{BrinK} which constructs Brownian motion
on two dimensional Euclidean complexes.  Bouziane \cite{Bouz}
constructs and proves the existence of Brownian motion on admissible
Reimannian complexes of any dimension.  It would be interesting to
determine whether these constructions define the same operator;
however, that is not our focus.

\chapter{Local Poincar\'{e} Inequalities on X}
  In this chapter we will show that a uniform local Poincar\'{e}
inequality holds for a certain class of admissible complexes.  Local
Poincar\'{e} inequalities have appeared in \cite{White} and
\cite{EellF} for finite complexes or for compact subsets of complexes.
In White's article \cite{White}, a global Poincar\'{e} inequality was
shown for Lipschitz functions on an admissible complex made up of a
finite number of polyhedra.  The constant in this proof was linear in
the number of polyhedra involved, and so it does not extend to an
infinite complex.  A uniform weak local inequality for Lipschitz
functions was also shown on this finite complex.  This too differs
from our inequality in its dependence on a finite complex.

In Eells and Fuglede's book \cite{EellF}, they show that for any
relatively compact subset of an admissible complex, a local
Poincar\'{e} inequality will hold for locally Lipschitz functions with
a constant that depends on the particular choice of compact subset.
The larger complex itself can be infinite, but the constant in the
inequality depends on our particular choice of compact subset.

We will show the following for $f \in \Lip(X)$, under some assumptions
on the geometry of $X$:
\begin{eqnarray*}
\norm{f-f_B}_{p,B} \le p P_0 r \norm{\grad f}_{p,B}
\end{eqnarray*}
where $f_B$ is the average of $f$ over $B$, $B=B(z,r)$, $r<R_0$.
$R_0$ and $P_0$ are constants depending on the space, $X$.

Our result shows that a uniform local Poincar\'{e} inequality will
hold for Lipschitz continuous functions on any ball of radius less than
$R_0$, where $R_0$ is fixed and depends only on the complex itself, not on
the specific choice of ball.  We require our complex to be admissible.
Our complex can be infinite, but we bound below the angles of the
polyhedra and the distance between two vertices.  We also bound above
the number of polyhedra that join at a vertex.  In both White and
Eells and Fuglede these assumptions hold because their sets are either
finite or relatively compact.

Connections between Poincar\'{e} inequalities and other analytic
inequalities can be found in Sobolev met Poincar\'{e} by Haj{\l}asz
and Koskela \cite{HajKosk}.

\section{Weak Poincar\'{e} Inequalities}
We would like to prove a local Poincar\'{e} inequality for an
admissible Euclidean polytopal complex.  If we look at a convex
subset of Euclidean space, this is a well known statement.  We will show it
first in a convex space, and then we will generalize it to our locally
nonconvex space.

A note on our notation: often we will abbreviate $d\mu(x)$ by $dx$.
Similarly, we will write the average integral of $f$ over a set $A$ by
$\aveint_A f dx$.

\begin{lemma}\label{ConvexSimplify}
Let $\Omega$ be a connected convex set with Euclidean distance and
structure and $\Omega_1, \Omega_2$ be convex subsets of $\Omega$.  For
$f \in \Lip(\Omega) \cap L^1(\Omega)$, the following
holds:
\begin{eqnarray*}
\int_{\Omega_2}  
           \int_{\Omega_1} \abs{f(z)-f(y)}dz dy 
\le 2^{n-1} \frac{ \diam(\Omega)}{n} (\mu(\Omega_1)+\mu(\Omega_2)) 
  \int_{\Omega} \abs{\grad f(y)} dy.
\end{eqnarray*}
\end{lemma}
\begin{proof}
The type of argument used here can be found in Aspects of Sobolev-Type
Inequalities \cite{LSC}.

Let $\gamma$ be a path from $z$ to $y$. The definition of a
 gradient gives us:
\begin{eqnarray*}
\abs{f(z)-f(y)} \le \int_{\gamma}\abs{ \grad f(s)} ds .
\end{eqnarray*}
Note that if we are in a 1-dimensional space, a convex subset is a
line.  The desired inequality follows from expanding $\gamma$ to
$\Omega$, and then noting that integrating over $x$ and $y$ has the
effect of multiplying the right hand side by
\\ $\mu(\Omega_1)\mu(\Omega_2) \le \diam(\Omega)(\mu(\Omega_1)
+\mu(\Omega_2))$.

Because $z$ and $y$ are in the same convex region $\Omega$ with an
Euclidean distance, we can let the path $\gamma$ be a straight line:
\begin{eqnarray*}
\abs{f(z)-f(y)} \le 
\int_0^{\abs{y-z}} \abs{\grad f\left(z + \rho \frac{y-z}{|y-z|}\right)} d\rho .
\end{eqnarray*}
We integrate this over $z \in \Omega_1,y \in \Omega_2$.  To get a nice
bound, we will use a trick from Korevaar and Schoen \cite{KorevaarS}.
We split the path into two halves.  For each half, we switch into and
out of polar coordinates in a way that avoids integrating
$\frac{1}{s}$ near $s=0$.  This allows us to have a bound which
depends on the volumes of $\Omega_1$ and $\Omega_2$ rather than
$\Omega$.

First, we consider the half of the path which is closer to $y \in
\Omega_2$.  $I_{\Omega}(\cdot)$ is the indicator function for $\Omega$.
\begin{eqnarray*}
\int_{\Omega_1} \int_{\Omega_2} \int_{\frac{|y-z|}{2}}^{|y-z|} 
  \abs{\grad f(z + \rho\frac{y-z}{|y-z|})}I_{\Omega}(z +
  \rho\frac{y-z}{|y-z|}) d\rho dy dz .
\end{eqnarray*}
We change of variable so that $y-z = s \theta$.  That is, $|y-z| = s$
and $\frac{y-z}{|y-z|} = \theta$. Note that $\diam(\Omega)$ is an
upper bound on the distance between $y$ and $z$.
\begin{eqnarray*}
...= \int_{\Omega_1}\int_{S^{n-1}} \int_0^{\diam(\Omega)} 
  \int_{s/2}^{s} 
             \abs{\grad f(z + \rho\theta)}I_{\Omega}(z + \rho\theta)
                 s^{n-1}d\rho ds d\theta dz.
\end{eqnarray*}
We switch the order of integration.  Now, $\rho$ will be between $0$
and $\diam(\Omega)$ and $s$ will be between $\rho$ and
$\min(2\rho,\diam(\Omega))$.  This allows us to integrate with respect
to $s$.
\begin{eqnarray*}
... &=& \int_{\Omega_1}  \int_{S^{n-1}}  \int_0^{\diam(\Omega)} 
  \int_{\rho}^{ \min(2 \rho,\diam(\Omega))}
             \abs{\grad f(z + \rho\theta)}I_{\Omega}(z + \rho\theta)
                 s^{n-1} ds d\rho d\theta dz \\
&=& \int_{\Omega_1} \int_{S^{n-1}}  \int_0^{\diam(\Omega)} 
             \abs{\grad f(z + \rho\theta)}I_{\Omega}(z + \rho\theta)
              \frac{(\min(2 \rho,\diam(\Omega)))^n -\rho^n}{n}d\rho d\theta dz.
\end{eqnarray*}
Now we reverse the change of variables to set $y = z + \rho \theta$.
Since our integral includes an indicator function at $z + \rho
\theta$, we have $y \in \Omega$.
\begin{eqnarray*}
\int_{\Omega_1} \int_{\Omega} 
  \abs{\grad f(y)}
    \frac{(\min( 2 |y-z|,\diam(\Omega)))^n - |y-z|^n}{n |y-z|^{n-1}} dy dz.
\end{eqnarray*}
Let's consider the possible values of $\frac{(\min( 2
|y-z|,\diam(\Omega)))^n - |y-z|^n}{n|y-z|^{n-1}}$.

If $|y-z| < \frac{\diam(\Omega)}{2}$, then $\min( 2
|y-z|,\diam(\Omega)) = 2 |y-z|$.  This gives us:
\begin{eqnarray*} 
\frac{(\min( 2 |y-z|,\diam(\Omega)))^n - |y-z|^n}{n|y-z|^{n-1}}
&=& \frac{ 2^n |y-z|^n - |y-z|^n}{n|y-z|^{n-1}} \\
&=& \frac{2^n-1}{n} |y-z| \\
&\le& \frac{\diam(\Omega)(2^n-1)}{2n}.
\end{eqnarray*}
Otherwise, if $|y-z| \ge \frac{\diam(\Omega)}{2}$, then 
$\min( 2 |y-z|,\diam(\Omega)) =\diam(\Omega) $.  This gives us:
\begin{eqnarray*} 
\frac{(\min( 2 |y-z|,\diam(\Omega)))^n - |y-z|^n}{n|y-z|^{n-1}}
&=& \frac{ \diam(\Omega)^n - |y-z|^n}{n|y-z|^{n-1}} \\
&\le& 2^{n-1} \frac{ \diam(\Omega)^n - |y-z|^n}{n \diam(\Omega)^{n-1}} \\
&\le& 2^{n-1} \frac{ \diam(\Omega)^n}{n \diam(\Omega)^{n-1}}\\
&=& 2^{n-1} \frac{ \diam(\Omega)}{n}.
\end{eqnarray*}
Both cases are dominated by $ 2^{n-1} \frac{ \diam(\Omega)}{n}$.  We
place this into the original integral:
\begin{equation*}
\begin{split}
\int_{\Omega_1} \int_{\Omega} 
  \abs{\grad f(y)}
   & \frac{(\min( 2 |y-z|,\diam(\Omega)))^n - |y-z|^n}{n |y-z|^{n-1}} dy dz
\\ &\le \int_{\Omega_1} \int_{\Omega} 
  \abs{\grad f(y)} 2^{n-1} \frac{ \diam(\Omega)}{n} dy dz
\\ &= 2^{n-1} \frac{ \diam(\Omega)}{n} \mu(\Omega_1) 
  \int_{\Omega} \abs{\grad f(y)} dy.
\end{split}
\end{equation*}
This is an upper bound for 
\begin{eqnarray*}
\int_{\Omega_1} \int_{\Omega_2} \int_{\frac{|y-z|}{2}}^{|y-z|} 
  \abs{\grad f(z + \rho\frac{y-z}{|y-z|})}I_{\Omega}(z + \rho\frac{y-z}{|y-z|})
  d\rho dy dz .
\end{eqnarray*}
We can apply the same argument to the half of the geodesic closest to $z
\in \Omega_1$, after first substituting $\rho'=|y-z| -\rho$:
\begin{equation*}
\begin{split}
\int_{\Omega_1} \int_{\Omega_2} \int_0^{\frac{|y-z|}{2}} 
 & \abs{\grad f(z + \rho\frac{y-z}{|y-z|})}
   I_{\Omega}(z + \rho\frac{y-z}{|y-z|})
   d\rho dy dz
\\ &= \int_{\Omega_2} \int_{\Omega_1} \int_{\frac{|y-z|}{2}}^{|y-z|} 
  \abs{\grad f(y + \rho'\frac{z-y}{|z-y|})} I_{\Omega}(y +
      \rho'\frac{z-y}{|z-y|}) d\rho dz dy
\\ &\le  2^{n-1} \frac{ \diam(\Omega)}{n} \mu(\Omega_2) 
  \int_{\Omega} \abs{\grad f(y)} dy.
\end{split}
\end{equation*}
Combining these with the original inequality, we have 
\begin{eqnarray*}
\int_{\Omega_2}  
           \int_{\Omega_1} \abs{f(z)-f(y)}dz dy 
\le 2^{n-1} \frac{ \diam(\Omega)}{n} (\mu(\Omega_1)+\mu(\Omega_2)) 
  \int_{\Omega} \abs{\grad f(y)} dy.
\end{eqnarray*}
\end{proof}
\begin{notation}
Let $X$ be an admissible Euclidean polytopal complex of dimension $n$.
\end{notation}
\begin{definition}
Let $B$ be a ball of radius $r$ whose center is on a $D$-dimensional
face with the property that $B$ intersects no other $D$-dimensional
faces. We define wedges $W_k$ of $B$ to be the closures of each of the
connected components of \\ $B-X^{(n-1)}$.
\end{definition}
Note that for any point $z$ in $X$, a ball $B(z,r)$ satisfying the
above criteria exists: for each dimension $D$, we can take any point $z \in
X^D-X^{(D-1)}$ and any $r < d(z,X^{(D-1)})$ and create $B= B(z,r)
\subset X$.  Then $B$ is a ball of radius $r$ whose center is on a
$D$-dimensional face, and $B$ intersects no other $D$-dimensional
faces.  In essence, the wedges, $W_k$, are formed when the $(n-1)$
skeleton slices the ball $B$ into pieces.  This construction tells us
that each $W_k$ has diameter at most $2r$, as each of the points in
$W_k$ is within distance $r$ of $z$, and $z$ is included in $W_k$.
\begin{example}\label{SkelWithWedges}
\begin{figure}[h]
\centering
       \includegraphics[angle=0,width=1in]{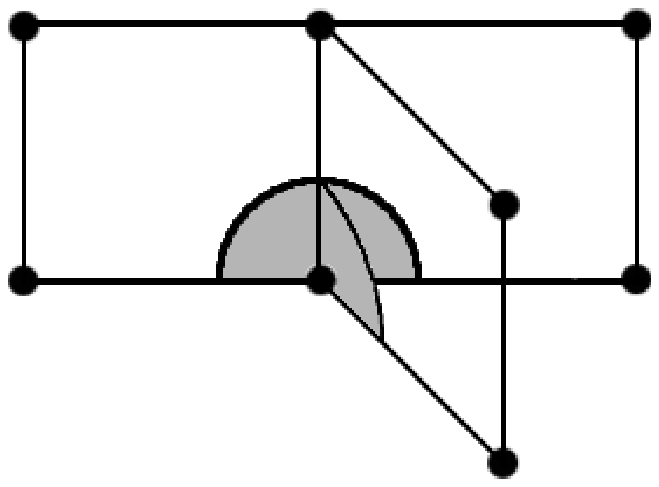}
\hspace{.5in}
       \includegraphics[angle=0,width=1in]{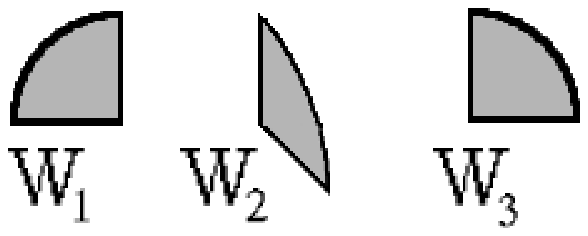}
\caption{Complex with shaded ball B (left); the three wedges for B (right).}
\end{figure}
In figure \ref{SkelWithWedges} we have an example of a 2 dimensional complex with a shaded ball
centered at a vertex.  This ball has three wedges; one for each of the
two dimensional faces that share the vertex.  Note that each wedge is
a fraction of a sphere.
\end{example}
\begin{definition}
We say that $X$ has degree bounded by $M$ if $M$ is the maximal
number of edges in $X^{(1)}$ that can share a vertex in $X^{(0)}$.
\end{definition}
\begin{definition}
We say that $X$ has edge lengths bounded below by $\varl$ if
\begin{equation*} 
0 < \varl \le \inf_{v,w \in X^{(0)}} d(v,w).
\end{equation*}
\end{definition}
Note that having degree bounded by $M$ implies that the maximum number
of $k$ dimensional faces that can share a lower dimensional face is
also $M$.  This tells us that sufficiently small balls will be split
into at most $M$ wedges. Note that if $X$ has degree bounded by $M$,
then $X^{(k)}$ will as well.  When $X$ has degree bounded by $M$,
volume doubling will occur locally with a uniform constant.  In
particular, when the edge lengths are bounded below by $\varl$ the
strong statement:
\begin{eqnarray*}
\mu(B(x,c r)) \le M c^N \mu(B(x,r))
\end{eqnarray*}
will hold whenever $c r \le \varl$. For balls in $X$, $N$ will equal
$n$, the dimension of $X$.  If we restrict to balls in $X^{(k)}$, then
this holds with $N=k$.

To show a local Poincar\'{e} inequality on $X$, we will split the
balls, which are not necessarily convex, up into smaller overlapping
pieces which are.  We will do this using the wedges.  We can use a
chaining argument in order to move through $B$ from one of the $W_k$
to another.  Note that this uses the fact that our space $X$ is
admissible.  We will say $W_k$ and $W_j$ are adjacent if they share an
$n-1$ dimensional face, and let $N(j)$ be the list of indices of faces
adjacent to $W_j$ including $j$.  In order to create paths which we
can integrate over, we need an overlapping region between the adjacent
faces. For $k \in N(j)$, let $W_{k,j}=W_{j,k}$ be the largest subset
of $W_k \cup W_j$ which has the property that $W_k \cup W_{k,j}$ and
$W_j \cup W_{k,j}$ are both convex.  Then, for each $x$ in $W_{k,j}$
there is a way of describing the rays between $x$ and $W_k$ in a
distance preserving manner as one would have in $R^n$.  This
will justify our use of the $\rho$ in the calculation below.
\begin{example}
\begin{figure}[h]\label{NonConvexWedges}
\centering
       \includegraphics[angle=0,width=1in]{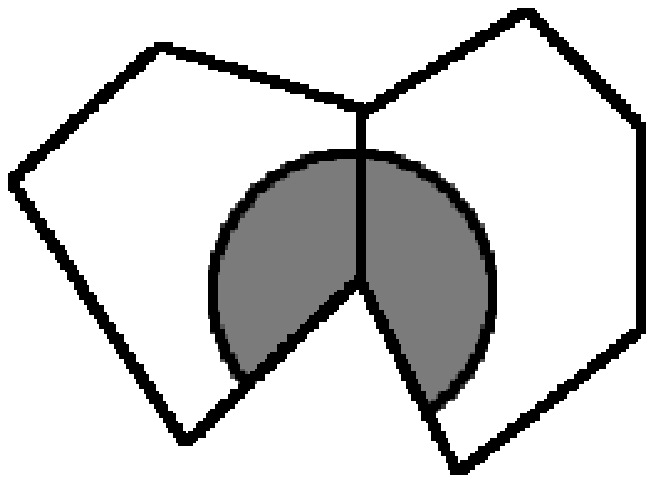}
\hspace{.5in}
       \includegraphics[angle=0,width=1in]{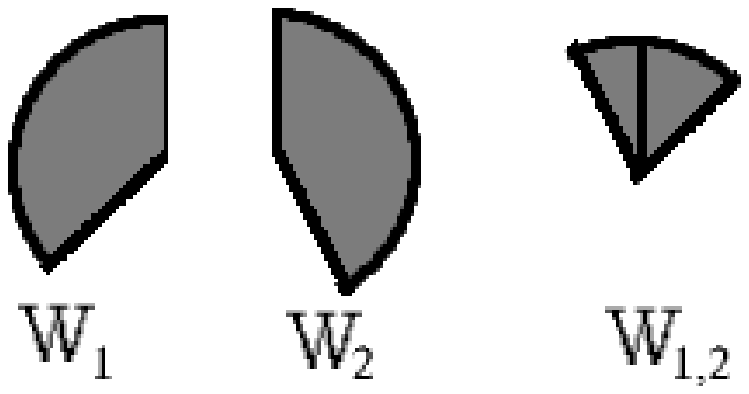}
\caption{Complex with shaded ball B (left); the two wedges for B and a region which overlaps both of them(right).}
\end{figure}
In figure \ref{NonConvexWedges} we have a complex and ball with two
adjacent wedges. The union of the wedges, $W_1$ and $W_2$, is not
convex, so we form the region $W_{1,2}$.  In this example, both $W_1
\cup W_{1,2}$ and $W_2 \cup W_{1,2}$ are half circles.
\end{example}
\begin{theorem}\label{PoincareP1}
Let $X$ be an admissible Euclidean polytopal complex of dimension $n$
with degree bounded by $M$.  For each $z \in X$ there exists $r>0$ so
that for $B=B(z,r)$ and its corresponding wedges, $W_{i,j}$, the
following holds for \\ $f \in \Lip(X) \cap L^1(B)$:
\begin{eqnarray*}
||f - f_B||_{1,B} 
\le  2M \max_{k,j \in N(k)} 
     \left( \frac{\mu(B)}{\mu(W_k)} +2 \right)
          \frac{2^n r (\mu(W_k) + \mu(W_{j,k}))}{n\mu(W_{j,k})} 
             ||\grad f||_{1,B} .
\end{eqnarray*}
\end{theorem}
\begin{proof}
For a given $z\in X$ let $D$ be the dimension such that $z \in
X^{(D)}-X^{(D-1)}$.  Pick $r < d(z,X^{(D-1)})$.  Let $B=B(z,r)$. For $x$
in $B$ we have:
\begin{eqnarray*}
|f(x)-f_B| 
&=& 
 \abs{\int_B \frac{1}{\mu(B)} f(x)dy -\int_B \frac{1}{\mu(B)} f(y)dy} \\
&\le& \frac{1}{\mu(B)} \int_B |f(x) - f(y)|dy.
\end{eqnarray*}
We would like to apply Lemma \ref{ConvexSimplify} to this; however,
$B$ is not necessarily convex.  We will construct a path from $x$ to
$y$ using a finite number of straight lines, where each of the line
segments is contained in a convex region.  For simplicity, we will
consider $x \in W_i$ and $y \in W_k$.  It is quite possible that these
two wedges are not contained in a convex subset of $B$.  We need to
use the fact that our space is locally $(n-1)$-chainable by looking at
a chain in $B-\{z\}$ starting at $W_i$ and ending at $W_k$.  The
pieces of the chain will move us from a point in $W_j$ into a
connecting point in $W_{j,l}$, and then from that connecting point in
$W_{j,l}$ into a point in $W_l$.

Formulated more precisely, there is a sequence of indices,
$\sigma(1)=i,...\sigma(l)=k$ that corresponds to this chain of $W$'s, so
that for each $j$, $W_{\sigma(j)}$ and $W_{\sigma(j+1)}$ are adjacent,
and none of the indices repeat.  We can take points in these regions;
$z_1 \in W_{\sigma(1)}$, $z_2 \in W_{\sigma(1),\sigma(2)}$,
... $z_{2j-1} \in W_{\sigma(j)}$ and $z_{2j} \in
W_{\sigma(j),\sigma(j+1)}$.  Note that each pair in this sequence is
located in a convex region-- either $W_{\sigma(j)} \cup
W_{\sigma(j),\sigma(j+1)}$ or $W_{\sigma(j+1)} \cup
W_{\sigma(j),\sigma(j+1)}$.  The line segments between these
points will define our path $\gamma$ from $x$ to $y$.
\begin{eqnarray*}
|f(x)-f(y)|
 &=& |f(x)-f(z_1) + f(z_1) -...+f(z_{2l}) -f(y)|
\\ &\le& |f(x)-f(z_1)| + |f(z_1) -f(z_2)|+...+|f(z_{2l}) -f(y)|
\\ &=& |f(x)-f(z_1)| + \sum_{j=1}^{l-1} 
\left( |f(z_{2j})-f(z_{2j-1})|+|f(z_{2j})-f(z_{2j+1})| \right) \\  
 & & \mbox{} + |f(z_{2l}) -f(y)|.
\end{eqnarray*}
Since it didn't matter which $z$'s we chose, as long as they were in the
proper sets, we can average the pieces over all of the possible $z$'s.
\begin{eqnarray*}
|f(x)-f(y)|
&\le &
\lefteqn{\aveint_{W_{i,\sigma(1)}}
   |f(x)-f(z_1)| dz_1}
 \\  & & + 
\sum_{j=1}^{l-1} 
\left(
   \aveint_{W_{\sigma(j)}} \aveint_{W_{\sigma(j),\sigma(j+1)}}
    |f(z_{2j})-f(z_{2j-1})| dz_{2j} dz_{2j-1} \right.
\\ & &+ \left.
  \aveint_{W_{\sigma(j+1)}} \aveint_{W_{\sigma(j),\sigma(j+1)}}
|f(z_{2j})-f(z_{2j+1})|
     dz_{2j} dz_{2j+1}
\right) \\  & &+ \aveint_{W_{\sigma(l),k}}
   |f(z_{2l})-f(y)| dz_{2l} .
\end{eqnarray*}
We will not want to keep track of the exact path between every pair of
regions, although in specific examples one may want to do that in
order to achieve a tighter bound.  Rather, it is useful simply
integrate over all pairs of neighboring wedges, as this will include
everything in our path.  
\begin{eqnarray*}
|f(x)-f(y)|
&\le & \lefteqn{\sum_{l \in N(i)} \aveint_{W_{i,l}}
   |f(x)-f(z)| dz}
 \\  & &+ 
\sum_{j} \sum_{l \ne i, l \in N(j)} 
  \aveint_{W_l} \aveint_{W_{j,l}}
 |f(z)-f(w)| dz dw
 \\  & &+ \sum_{j \in N(k)} \aveint_{W_{j,k}}
     |f(z)-f(y)| dz.
\end{eqnarray*}
This new inequality will hold for $x$ and $y$ in any pair of $W_i$ and
$W_k$ with $k \ne i$.  If we expand our notation so that $W_{i,i}=W_i$
, then this will hold when $x$ and $y$ are in the same set $W_k=W_i$.
To integrate over all $y \in B$, we can split the integral into two
parts; one where $x$ and $y$ are both in $W_i$, and then add it to the
second where $y$ is in one of the $W_k \ne W_i$.  Similarly, we can
integrate over $x$ in $W_i$ and then sum over $i$.
\begin{equation*}
\begin{split}
\frac{1}{\mu(B)} \int_B \int_B & |f(x) - f(y)|dy dx\\
\le & \lefteqn{\frac{1}{\mu(B)} 
\left(
\sum_{i,k}  \sum_{l \in N(i)} \int_{W_i} \int_{W_k}
   \aveint_{W_{i,l}} 
             |f(x)-f(z)|dz dy dx \right.} \\
& +\sum_{i,k,j} \sum_{l \in N(j)} \int_{W_i}
   \int_{W_k}  
   \aveint_{W_{j,l}} \aveint_{W_l} 
             |f(z)-f(w)|dw dz dy dx \\
& + \left. \sum_{i,k} \sum_{j \in N(k)}\int_{W_i} \int_{W_k} 
    \aveint_{W_{j,k}} 
             |f(z)-f(y)|dz dy dx
\right) \\
= & \lefteqn{ 
\sum_i   \sum_{l \in N(i)} \int_{W_i} 
   \aveint_{W_{i,l}} 
             |f(x)-f(z)|dz dx}\\
&+\mu(B) \sum_{j} \sum_{l \in N(j)} 
   \aveint_{W_{j,l}} \aveint_{W_l} 
             |f(z)-f(w)|dw dz \\
&+ \sum_{k} \sum_{j \in N(k)} \int_{W_k} 
    \aveint_{W_{j,k}} 
             |f(z)-f(y)|dz dy.
\end{split}
\end{equation*}
We can combine this into one double sum by setting $x=w$ and $y=w$ as
well as reindexing so that $i=j$ and $l=k$.
\begin{eqnarray*}
... \le
 \sum_{k} \sum_{j \in N(k)} 
   \left( \frac{\mu(B)}{\mu(W_k)} +2 \right) 
  \int_{W_k} \aveint_{W_{j,k}} 
             |f(z)-f(w)| dz dw .
\end{eqnarray*}
Applying lemma \ref{ConvexSimplify} with 
$\Omega=W_k \cup W_{j,k}$, $\Omega_1=W_{j,k}$, $\Omega_2=W_k$, 
and $\diam(\Omega) \le 2r$ to each of the pieces we find:
\begin{eqnarray*}
... \le
 \sum_{k} \sum_{j \in N(k)} 
   \left( \frac{\mu(B)}{\mu(W_k)} +2 \right)
   \frac{2^n r (\mu(W_k) + \mu(W_{j,k}))}{n \mu(W_{j,k})} 
  \int_{W_k \cup W_{j,k}} 
             |\grad f(y)| dy.
\end{eqnarray*}
Note that points in the sets $W_k \cup W_{j,k}$ are counted at most
$2M$ times, since each of the $W_k$ has at most $M$ neighbors.  This
allows us to combine the sums to find:
\begin{equation*}
\begin{split}
\frac{1}{\mu(B)} \int_B \int_B &|f(x) - f(y)|dy dx \\
&\le
 2M \max_{k,j \in N(k)} 
   \left( \frac{\mu(B)}{\mu(W_k)} +2 \right)
\frac{2^n r (\mu(W_k) + \mu(W_{j,k}))}{n \mu(W_{j,k})}
  \int_{B} |\grad f(y)| dy .
\end{split}
\end{equation*}
This is the desired result.
\end{proof}
Now that we have the inequality when $p=1$, we can use a trick to
extend it to other values of $p$.
\begin{lemma}\label{PoincareP}
If for any $f \in \Lip(X)$ we have:
\begin{eqnarray*}
||f-f_B||_{1,B} \le C r||\grad f||_{1,B}
\end{eqnarray*}
 for $B=B(z,r)$ then 
\begin{eqnarray*}
\inf_{c \in (-\infty,\infty)} ||f - c||_{p,B} 
&\le& p C r||\grad f||_{p,B} \text{ and} \\
||f-f_B||_{p,B} &\le& 2p C r||\grad f||_{p,B}.
\end{eqnarray*}
 holds for $1 \le p < \infty$.
\end{lemma}
\begin{proof}
Let $g(x)=|f(x)-c_f|^p \sign(f(x)-c_f)$.  Note that $g$ is in
 $\Lip(X)$. Then if $\grad f$ is the gradient of $f$, we have that $p
 |f(x)-c_f|^{p-1} |\grad f(x)|$ is the length of the gradient of $g$.

Pick a value of $c_f$ so that $g_B=\int_B g(x) dx =0$.  (One will
exist; we consider $g_B$ as a function of $c_f$ and apply the
intermediate value theorem.)

Applying our assumption to $g$, we have:
\begin{eqnarray*}
\int_B |g(x)-0| dx &\le& 
C \int_B |\grad g(x)| dx  \\
&=& C \int_B  p |f(x)-c_f|^{p-1} |\grad f(x)| dx .
\end{eqnarray*}
Now we use H\"{o}lder to find:
\begin{eqnarray*}
\int_B |f(x)-c_f|^{p-1} |\grad f(x)| dx 
\le \left(\int_B (|f(x)-c_f|^{p-1})^q dx\right)^{1/q} 
    \left(\int_B |\grad f(x)|^p dx\right)^{1/p} .
\end{eqnarray*}
Since $\frac{1}{p} +\frac{1}{q} =1$, we have $(p-1)q=p$. Combining
this with the above inequality gives us:
\begin{eqnarray*}
\left(\int_B  |f(x)-c_f|^p dx\right)^{1/p} 
\le p 
C r \left(\int_B |\grad f(x)|^p dx\right)^{1/p} .
\end{eqnarray*}

When we take an infimum, we find that 
\begin{eqnarray*}
\inf_c \left(\int_B  |f(x)-c|^p dx\right)^{1/p} 
\le \left(\int_B  |f(x)-c_f|^p dx\right)^{1/p}.
\end{eqnarray*}
When we combine these, we have
\begin{eqnarray*}
\inf_{c} ||f - c||_{p,B} 
\le p C r||\grad f||_{p,B}. 
\end{eqnarray*}

In the case where $p=2$, it is easy to compute the infimum exactly.
Consider 
\begin{eqnarray*}
h(c)&=& \int_B |f(x)-c|^2 dx \\
    &=& \int_B f(x)^2 dx -2c \int_B f(x) dx +c^2 \mu(B) .
\end{eqnarray*}
This is a parabola whose minimum occurs at $c=\frac{1}{\mu(B)} \int_B
f(x) dx = f_B$.  Its minimum is the same as that of $\sqrt{h(c)}$, and
so this gives us
\begin{eqnarray*}
||f-f_B||_{2,B} \le 2 C r||\grad f||_{2,B}.
\end{eqnarray*}

When $p \ne 2$, we can use Jensen's inequality to get the average.  We
do this by noticing:
\begin{eqnarray*}
\dashint |f_B -c|^p dx &=& |\dashint_B (f-c)dx |^p \\
& \le& \dashint_B |f-c|^p dx. 
\end{eqnarray*}
This tells us that 
\begin{eqnarray*}
||f -f_B||_{p,B}
&\le& \inf_c ||f -c||_{p,B} + ||f_B -c||_{p,B} \\
&\le& 2 \inf_c ||f -c||_{p,B} \\
&\le& 2 p C r||\grad f||_{p,B}. 
\end{eqnarray*}
\end{proof}
\begin{definition}
We say that $X$ has solid angle bound $\alpha$ if for each \\ $z \in
X^{(D)}-X^{(D-1)}$ and $r <d(z,X^{(D-1)})$ the wedges of the ball $B(z,r)$
satisfy
\begin{eqnarray*}
\alpha \le \frac{\mu(W_k)}{\mu(r^n S^{(n-1)})} \le 1.
\end{eqnarray*}
\end{definition}
Note that the right hand side of the inequality reflects the fact that
each of the $W_k$ is a subset of an Euclidean ball.

If we have a uniform bound on the solid angles formed, then the
constant in Theorem \ref{PoincareP1} will simplify.

\begin{corrolary}\label{VolPPoincare}
Suppose $X$ is an admissible n-dimensional Euclidean polytopal complex
with solid angle bound $\alpha$, and $f\in \Lip(X)$.  For each $z \in
X$ there exists $r>0$ so that for $B=B(z,r)$ we have
\begin{eqnarray*}
||f - f_B||_{p,B} 
\le C_{X} p r ||\grad f||_{p,B}
\end{eqnarray*}
where the constant $C_{X} = \frac{2^{3n+3} M^2}{\alpha n}$ depends
only on the space $X$.
\end{corrolary}
\begin{proof}
We need to bound $ \max_{k,j \in N(k)} \left( \frac{\mu(B)}{\mu(W_k)}
+2 \right) \frac{\mu(W_k) + \mu(W_{j,k})}{\mu(W_{j,k})}$ from Theorem
\ref{PoincareP1}.  Since we will want to bound the $\mu(W_{j,k})$, we
need a way to compare its size to the volume of the other $W_j$.  We
can subdivide the space initially by cutting each piece in half in
each of the $n$ dimensions, so that there are at most $M'= 2^n M$
pieces.  When the $W'_k$ and $W'_j$ are adjacent, this tells us that
$W'_{j,k}$ has a volume which is larger than
$\min(\mu(W'_j),\mu(W'_k))$.  Thus $\frac{\mu(W_k) +
\mu(W_{j,k})}{\mu(W_{j,k})} \le 2$.

To bound $\frac{\mu(B)}{\mu(W_k)}$ we will need the solid angle bound.
Combining the solid angle bound with the factor of $2^{-n}$ decrease
in wedge size gives us the modified inequality:
\begin{eqnarray*}
\mu(W'_k) \le 2^{-n} \mu(r^n S^{(n-1)})  \le \frac{\mu(W'_k)}{\alpha}.
\end{eqnarray*}
Summing the left hand side of the inequality over $k$ tells us that
\begin{eqnarray*}
\mu(B) \le M 2^n 2^{-n} \mu(r^n S^{(n-1)}). 
\end{eqnarray*}
If we multiply the right hand side of the inequality by $M2^n$, we have
\begin{eqnarray*}
M \mu(r^n S^{(n-1)})  \le \frac{M 2^n \mu(W'_k)}{\alpha}.
\end{eqnarray*}
Combining these two inequalities, we find that:
\begin{eqnarray*}
\frac{\mu(B)}{\mu(W'_k)} \le \frac{M2^n}{\alpha}.
\end{eqnarray*}
We can substitute these into our constant to get:
\begin{eqnarray*}
2M' \max_{k,j \in N(k)} 
   \left( \frac{\mu(B)}{\mu(W'_k)} +2\right)
         \frac{2^n r(\mu(W'_k)+\mu(W'_{j,k}))}{n\mu(W'_{j,k})} 
\le 2 M2^n 
   \left(\frac{M2^n}{\alpha} +2\right)
 \frac{2^{n+1}r}{n} .
\end{eqnarray*}
This combined with theorem \ref{PoincareP1} and lemma
\ref{PoincareP} gives us that:
\begin{eqnarray*}
||f - f_B||_{p,B} 
\le \frac{2^{3n+3} M^2}{\alpha n}  p r ||\grad f||_{p,B}.
\end{eqnarray*}
\end{proof}

We would like to extend these theorems so that the radius is not dependent
on the center of the ball.  To do so, we will first show a weaker
Poincar\'{e} inequality, and then we will extend it via a Whitney type
covering to a stronger version.
\begin{theorem}\label{weakp1}
Suppose $X$ is an admissible n-dimensional Euclidean polytopal complex
with solid angle bound $\alpha$ and edge lengths bounded below by
$\varl$, and \\ $f\in \Lip(X)$.  When $\kappa =
6\left(\frac{2}{\sqrt{2(1-\cos(\alpha))}}+1\right)^n$, the following
inequality holds for each $z\in X$ and each $0\le r \le R_0$ where
$R_0 :=
\frac{\inf_{v,w\in X^{0}}d(v,w)}{
      6\left(\frac{2}{\sqrt{2(1-\cos(\alpha))}}+1\right)^n }
      =\frac{\varl}{\kappa}$.
\begin{eqnarray*}
\int_{B(z,r)} |f(x)-f_{B(z,r)}| dx 
              \le r C_{Weak}
                  \int_{B(z,\kappa r)} |\grad f(x)|dx
\end{eqnarray*}
where $C_{Weak}= \frac{2^{3n+3} M^3 \kappa^{N+1}}{\alpha n}$.
\end{theorem}
\begin{proof}
Let $z \in X$ and $r\le \frac{\inf_{v,w\in X^{0}}
d(v,w)}{6(\frac{2}{\sqrt{2(1-\cos(\alpha))}}+1)^n }$ be given.  If
$d(z,X^{(n-1)}) > r$, then the result follows as a weaker version of
Corollary \ref{VolPPoincare}.  Otherwise, we will need to find a
point $v_k$ which has the property that it is on a $k$-skeleton, and
there are no other faces in the $k$-skeleton that are intersected by
$B(v,d(v,z)+r)$.  We will do this by descending down the skeletons.

If there is a point within $r$ of $z$ with this property, we will use
it.                   

If not, set $r_0=3r$.  Then there is a $k$ such that the lowest
dimensional skeleton that is intersected by $B(z,r_0)$ is $X^{(k)}$,
and $X^{(k)}$ is intersected by $B(z,r_0)$ at at least two points
$v_k$ and $w_k$ on two different faces.  If these faces did not
intersect, then they would be at least distance $\inf_{v,w\in X^{(0)}}
d(v,w)$ from one another.  This would imply that $\inf_{v,w\in
X^{(0)}} d(v,w) \le 2r_0 = 6r$, a contradiction.  Thus those two faces
intersect in a smaller $j$-dimensional face.  Call $v_j$ the point on
the $j$-dimensional face which minimizes
$\min(d(v_j,v_k),d(v_j,w_k))$.  These three points form a triangle
with angle $v_k v_j w_k \ge \alpha$, where $\alpha$ is the smallest
interior angle in $X$.  Note that this angle is bounded by the
assumption $\alpha \le \frac{\mu(W_k)}{\mu(r^n S^{n-1})}$. The
triangle that would maximize the minimum distance to this new point,
$\min(d(v_j,v_k),d(v_j,w_k))$ is an isosceles one with angle
$v_k v_j w_k = \alpha$.  The law of cosines tells us that for the
isosceles triangle, $d(v_k,w_k)^2 = 2 d(v_k,v_j)^2(1-\cos(\alpha))$,
and so for a general triangle, $d(v_k,v_j) \le
\frac{d(v_k,w_k)}{\sqrt{2(1-\cos(\alpha))}} \le
\frac{2r_0}{\sqrt{2(1-\cos(\alpha))}}$.  

If this $v_j$ works, we're done.  Otherwise, we will have at least two
points, $v_j$ and $w_j$ within
$\left(\frac{2}{\sqrt{2(1-\cos(\alpha))}}+1\right)r_0$ of $z$.  We
will repeat the process by taking new $r$'s of the form $r_{i+1}=
\left(\frac{2}{\sqrt{2(1-\cos(\alpha))}}+1\right)r_i$ until 
we have a point which works.  Note that each time we repeat it, we
find a point on a lower dimensional skeleton.  The worst case scenario
will have us repeat this $n$ times until we're left with at least one
point on $X^{(0)}$.  The largest radius that we could require is $R=
\left(\frac{2}{\sqrt{2(1-\cos(\alpha))}}+1\right)^n 3 r$.  Using this 
$R$, we can show that $B(v,R)$ does not intersect two vertices.  The
condition $r \le \frac{\inf_{v,w\in X^{0}}
d(v,w)}{6\left(\frac{2}{\sqrt{2(1-\cos(\alpha))}}+1\right)^n }$ tells
us that $R \le \frac{1}{2}\inf_{v,w\in X^{0}}d(v,w)$.  As two vertices
cannot be closer than the closest pair, $B(v,R)$ contains at
most one vertex.

This construction gives us a center, $v$, on a $k$-dimensional face,
and a radius, $R \le
\left(\frac{2}{\sqrt{2(1-\cos(\alpha))}}+1\right)^n 3 r $ so that
$B(v,R)$ intersects only the $k$-dimensional face that $v$ is on.
This allows us first to recenter our ball around $v$ and then to apply
Corollary \ref{VolPPoincare} to $f$ on $B(v,R)$.  Then, as $\kappa =
6\left(\frac{2}{\sqrt{2(1-\cos(\alpha))}}+1\right)^n$, we find $B(v,R)
\subset B(z,\kappa r)$.
\begin{eqnarray*}
\int_{B(z,r)} |f(x)-f_{B(z,r)}| dx 
 &=&   \frac{1}{\mu(B(z,r))}\int_{B(z,r)} \int_{B(z,r)}|f(x)-f(y)| dx dy \\
 &\le& \frac{1}{\mu(B(z,r))}\int_{B(v,R)} \int_{B(v,R)}|f(x)-f(y)| dx dy \\
 &\le& \frac{\mu(B(v,R))}{\mu(B(z,r))} \int_{B(v,R)} |f(x)-f_{B(v,R)}| dx \\
 &\le& \frac{\mu(B(v,R))}{\mu(B(z,r))} \frac{2^{3n+3} M^2}{\alpha n} 
R \int_{B(v,R)} |\grad f(x)|dx \\
 &\le& \frac{\mu(B(z,\kappa r))}{\mu(B(z,r))}\frac{2^{3n+3} M^2}{\alpha n} 
\kappa r \int_{B(z,\kappa r)} |\grad f(x)|dx\\
 &\le& M \kappa^N \frac{2^{3n+3} M^2}{\alpha n} \kappa r 
\int_{B(z,\kappa r)} |\grad f(x)|dx .
\end{eqnarray*}
\end{proof}

\section{Whitney Covers}
We would like to strengthen the weak version of the Poincar\'{e}
inequality. We'll do this by using a Whitney type covering of the
ball, $E=B(z,r)$.  Once we have this cover, we can use a chaining
argument to allow us to replace \\
$\kappa=6\left(\frac{2}{\sqrt{2(1-\cos(\alpha))}}+1\right)^n$ with 1.

Given our set we will consider a collection $F$ of balls such that
\begin{description}
\item(1) $B \in \F$ are disjoint. 
\item(2) If we expand the balls to ones with twice the radius, we cover all of
$E$. $\cup_{B\in \F} 2B =E$.
\item(3) For any $B \in F$, its radius is $r_B = 10^{-3}\kappa^{-1}
d(B,\bdry E)$.  This implies $10^{3}\kappa B \subset E$. Note that
this also tells us that the distance from the center of $B$ to the
boundary of $E$ is $(10^{-3}\kappa^{-1} +1)d(B, \bdry E) = (10^3\kappa
+1)r_B$.
\item(4) $\sup _{x \in E} |\{B\in \F | x \in 36 \kappa B \}| \le K$.
\end{description}
Note that the constant $\kappa$ depends on $X$ but not on the specific
choice of $E$. We will show in the following lemma that $K$ is
independent of $E$ as well.
\begin{lemma}
Property 4 is satisfied for $K \le C_{vol}^{\log_2(8(10^3\kappa
+1))}$.  When $E$ is a ball in $X$ which intersects only one vertex,
we have $K \le M (8(1+10^{3}\kappa))^N$.
\end{lemma}
\begin{proof}
Let a point $x\in E$ be given. Let $B(y,r_y) \in \F$ be a ball centered at $y$
with the property that $x \in 36\kappa B(y,r_y)$.  Then:
\begin{eqnarray*}
d(x,y) &\le& 36 \kappa r_y = 36 \kappa 10^{-3}\kappa^{-1} d(y, \bdry E) \\
&\le& .036 (d(x,y) +d(x,\bdry E)).
\end{eqnarray*}
When we solve for $d(x,y)$, we have:
\begin{eqnarray*}
d(x,y) \le 
     \frac{.036}{1-.036}d(x,\bdry E) .
\end{eqnarray*}
The triangle inequality tells us that 
\begin{eqnarray*}
d(x, \bdry E) -d(x,y) \le d(y,\bdry E ) \le d(y,x) + d(x, \bdry E).
\end{eqnarray*}
We use this to bound $d(y,\bdry E )$.
\begin{eqnarray*}
\frac{1-.072}{1-.036}d(x,\bdry E)
\le d(y,\bdry E) 
\le \frac{1}{1-.036}d(x,\bdry E).
\end{eqnarray*}
Because $r_y=(10^3\kappa +1)^{-1}d(y,\bdry E)$, this tells us that the
radius $r_y$ is bounded by
\begin{eqnarray*}
(1+10^{3}\kappa)^{-1}
\frac{1-.072}{1-.036}d(x,\bdry E)
\le r_y 
\le \frac{(1+10^{3}\kappa)^{-1}}{1-.036}d(x,\bdry E).
\end{eqnarray*}

These inequalities hold for any $B(y,r_y) \in \F$ with $x \in 36\kappa
B(y,r_y)$.  Each of the balls, $B_y$ will be contained in
$B_x:=B(x,r_1d(x,\bdry E))$ where $r_1=\frac{(1+10^{3}\kappa)^{-1} +
.036}{1-.036 }$.  The $B_y$ have radius at least $r_2d(x,\bdry E)$
where $r_2=(1+10^{3}\kappa)^{-1}\frac{1-.072}{1-.036}$.  We also know
that the $B_y$ are disjoint.  This tells us that:
\begin{eqnarray*}
|\{B\in \F | x \in 36 \kappa B \}|  \min_{B_y \in \F\cap B_x} \mu(B_y) \le 
\sum_{B_y \in \F \cap B_x} \mu(B_y) \le \mu(B_x).
\end{eqnarray*}
We can use volume doubling to compare the size of $\min_{B_y \in \F\cap B_x} \mu(B_y)$ and
$\mu(B_x)$.  Note that $r_1 < 2$ and
$\frac{1}{2(1+10^{3}\kappa)}<r_2$.  This tells us that \\ $B(x,r_1 d(x,
\bdry E)) \subset B(y, 8(1+10^3 \kappa)r_2 d(x,\bdry E))$.
\begin{eqnarray*}
\mu(B_x) 
\le C_{vol}^{\log_2(8(1+10^{3}\kappa))} \mu(B_y).
\end{eqnarray*}
Combining these inequalities and taking the supremum over $x \in E$ gives us:
\begin{eqnarray*}
 \sup _{x \in E} | \{B\in \F | x \in 36 \kappa B \} | \le
 C_{vol}^{\log_2(8(1+10^{3}\kappa))}.
\end{eqnarray*}
Note that this only depends on $\kappa$ and $C_{vol}$.
When $E$ intersects only one vertex we have: 
\begin{eqnarray*}
\mu(B_x) 
\le M 2^{2N} (2(1+10^{3}\kappa))^N
\mu(B_y).
\end{eqnarray*}
This gives us a more refined estimate.  Here $N$ is the dimension of $X$:
\begin{eqnarray*}
 \sup _{x \in E} | \{B\in \F | x \in 36 \kappa B \} |
\le M (8(1+10^{3}\kappa))^N .
\end{eqnarray*}
\end{proof}

We will first describe properties of this collection, and then we will
use them to show a Poincar\'{e} inequality.  This is a modified
version of the argument found in \cite{LSC}.  We can also use this
technique to take a Poincar\'{e} inequality on a small ball and extend
it to one on a larger ball whenever we have volume doubling.  This
increases the constant involved, so it cannot be done indefinitely,
but given a fixed radius we will be able to have inequalities that
hold up to balls of that size.

We begin with a bit of notation.  Let $B_z \in \F$ be a ball such that
$z\in 2B_z$.  Note that there may be more than one; we will pick one
arbitrarily.  As $z$ is the center of $E$, we will call $B_z$ the
central ball.  For a ball, $B$, call the center $x_B$, and fix
$\gamma_B$, a distance minimizing curve from $z$ to $x_B$.
\begin{lemma}\label{536}
For any $B \in \F$ we have
\begin{eqnarray*}
d(\gamma_B,\bdry E) \ge \frac{1}{2} d(B,\bdry E) = \frac{1}{2}\kappa 10^3 r_B.
\end{eqnarray*}
If $B'\in \F$ has the property $2B' \cap \gamma_B \ne \emptyset$, then
$r_{B'}\ge \frac{1}{4}r_B$.
\end{lemma}
\begin{proof}
This first claim will follow from multiple applications of the
triangle inequality.  Let $\alpha$ be the point in $\gamma_B$ which is
closest to the boundary: \\ $d(\gamma_B,\bdry E) = d(\alpha, \bdry E)$.
Then we can bound $r_E$:
\begin{eqnarray*}
d(z,\alpha) + d(\alpha,\bdry E) \ge d(z, \bdry E) = r_E
\end{eqnarray*}
and we can bound the distance from $B$ to $\bdry E$:
\begin{eqnarray*}
d(x_B,\alpha) + d(\alpha,\bdry E) \ge d(x_B,\bdry E) \ge d(B,\bdry E).
\end{eqnarray*}
Summing them, we find that:
\begin{eqnarray*}
d(z,\alpha) + d(x_B,\alpha) + 2d(\alpha,\bdry E)\ge  r_E +d(B,\bdry E). 
\end{eqnarray*}
As $z$ is on $\gamma_B$ and $\alpha$ minimizes the distance to the
boundary, we have \\ $d(z,\alpha) + d(\alpha,x_B)=d(z,x_B) \le r_E$.
Putting this into the inequality, we find:
\begin{eqnarray*}
r_{E} + 2d(\alpha,\bdry E) &\ge&  r_{E}+d(B,\bdry E) \\
d(\alpha,\bdry E) &\ge&  \frac{1}{2}d(B,\bdry E).
\end{eqnarray*}
The second part follows from the fact:
\begin{eqnarray*}
\frac{1}{2}d(B,\bdry E) \le d(\gamma_B,\bdry E) 
\le d(\gamma_B \cap 2B',\bdry E).
\end{eqnarray*}
If $\alpha'$ is the point in $\gamma_B \cap 2\bar{B'}$ and $\beta'$ is
the point in $\bar{B'}$ that realizes the distance to the boundary
$\bdry E$, then we have
\begin{eqnarray*}
d(\gamma_B \cap 2B',\bdry E) = d(\alpha',\bdry E) 
 &\le& d(\alpha',x_{B'}) + d(x_{B'},\beta') + d(B',\bdry E) \\
 &\le& 2r_{B'} + r_{B'} + d(B',\bdry E) .
\end{eqnarray*}
Combining this with fact (3), we have:
\begin{eqnarray*}
\frac{1}{2} 10^3\kappa r_{B} &\le& 3r_{B'} +10^3\kappa r_{B'} \\
\frac{1}{4}r_{B} &\le& \frac{10^3\kappa}{2(3+10^3\kappa)} r_{B} \le r_{B'}.
\end{eqnarray*}
\end{proof}
For each ball $B$ in $\F$, we would like to define a string of balls,
$\F(B)$, that takes $B_z$ to $B$.  Set $B_0=B_z$.  Then, for the first
point on $\gamma_B$ that is not contained in $2B_i$, take a ball
$B_{i+1}$ in $F$ such that that point is contained in $2B_{i+1}$.  As
$B_{i+1}$ is open, this guarantees that $2B_i\cap 2B_{i+1} \ne
\emptyset$.  We continue in this manner until $2B_{\varl -1}\cap 2B \ne
\emptyset$.  Then we set $B_{\varl} =B$.  We label $\F(B)=
\cup_{i=0}^{\varl} B_i$.  Note that due to volume doubling, the chain
will be finite.  This chain will allow us to move from the central
ball to any other ball in the cover of $E$.  It is useful because
neighboring balls are of comparable radii and volume.
\begin{lemma} \label{537}
For any $B\in \F$ and any $B_i,B_{i+1} \in \F(B)$ we can compare the
radii where $r_j=r_{B_j}$ in the following manner:
\begin{eqnarray*}
(1 + 10^{-2}\kappa^{-1})^{-1}r_i \le r_{i+1} \le (1 + 10^{-2}\kappa^{-1}) r_i.
\end{eqnarray*}
We also have $B_{i+1}\subset 6B_{i}$ and $B_{i}\subset 6B_{i+1}$, and
so 
\begin{eqnarray*}
\mu(6B_i \cap 6B_{i+1}) \ge \max\{\mu(B_i),\mu(B_{i+1})\}.
\end{eqnarray*}
\end{lemma}
\begin{proof}
Let $x_i$ and $x_{i+1}$ be the centers of $B_i$ and $B_{i+1}$
respectively.  By our
construction, $2B_i \cap 2B_{i+1} \ne \emptyset$.  This tells us that
$d(x_i,x_{i+1}) \le 2r_i + 2r_{i+1}$.
\begin{eqnarray*}
d(x_{i+1},\bdry E) &\le& d(x_{i+1},x_i) + d(x_i,\bdry E) \\ r_{i+1} +
d(B_{i+1},\bdry E) &\le& (2r_i + 2r_{i+1}) + (r_i + d(B_i,\bdry E)) \\
r_{i+1} +10^3 \kappa r_{i+1} &\le& 2r_{i+1} + 3r_i + 10^3 \kappa r_{i}
\\ r_{i+1} &\le& \frac{3+10^3 \kappa}{10^3 \kappa-1} r_i =
\left(1+\frac{4}{10^3 \kappa-1}\right) r_i .
\end{eqnarray*}
This tells us that $r_{i+1} \le(1+10^{-2}\kappa^{-1}) r_i$.  By a symmetric
argument, we also get the lower bound.

To show set inclusions, we use the fact that 
\begin{eqnarray*}
d(x_i,x_{i+1}) \le 2r_i + 2r_{i+1} \le 2r_i +2(1+10^{-2}\kappa^{-1}) r_i .
\end{eqnarray*}
Since any point in $B_{i+1}$ is within distance $r_{i+1}$ of
$x_{i+1}$, the triangle inequality tells us that it is within distance
$2r_i +2(1+10^{-2}\kappa^{-1}) r_i + r_{i+1} \le 6r_i$ of $x_i$.  This
gives us the inclusion $B_{i+1}\subset 6B_{i}$.  The reverse holds by
a symmetric argument, and so $\mu(6B_i \cap 6B_{i+1}) \ge
\max\{\mu(B_i),\mu(B_{i+1})\}$ follows.
\end{proof}
\begin{lemma}\label{538}
Given $B\in \F$ and $A \in \F(B)$ we have $B \subset (10^3\kappa +9)A$.
\end{lemma}
\begin{proof}
Let $x_B$ and $x_A$ be the centers of $B$ and $A$ respectively, and
let $\alpha$ be a point in $2A\cap \gamma_B$.  By \ref{536}, we know
that $4 r_A \ge r_B$.  Note that $\alpha$ occurs on the distance
minimizing curve $\gamma_B$ between $x_B$ which is the center of $B$
and $z$, the center of the large ball $E$.  This tells us that
$d(z,x_B) = d(z,\alpha) +d(\alpha,x_B)$.  Then
\begin{eqnarray*}
d(x_A,x_B) &\le& d(x_A,\alpha) + d(\alpha,x_B) \\
&=& d(x_A,\alpha) + (d(z,x_B) - d(z,\alpha)) \\
&\le& 2r_A +d(z,\bdry E) - d(z,\alpha) \\
&\le& 2r_A +(d(z,\alpha) +d(\alpha,\bdry E)) - d(z,\alpha) \\
&\le& 2r_A +d(\alpha, x_A) + d(x_A,\bdry E) \\
&\le& 4r_A + (r_A + 10^3\kappa r_A) = (10^3\kappa +5)r_A .
\end{eqnarray*}
Because all points in $B$ are within $r_B \le 4 r_A$ of $x_B$, they
will be within \\ $(10^3\kappa +5)r_A + 4r_A = (10^3\kappa +9)r_A$ of $x_A$.
Thus, $B \subset (10^3\kappa +9)A$ holds.
\end{proof}
We now have a number of lemmas that describe the geometry of the
covering.  We can use these to get an extension of our Poincar\'{e}
inequality.  We will do this by using this chain of balls to get a
chain of inequalities.  Our first step is to compare the average of
neighboring balls in the chain.
\begin{lemma}\label{539}
For $B_i$ and $B_{i+1}$ neighboring balls in a chain $\F(B)$, we have
\begin{eqnarray*}
|f_{6B_i} - f_{6B_{i+1}}| 
 \le C_{Weak} 18 \kappa \frac{r_i}{\mu(B_i)} 
\int_{36\kappa B_i} |\grad f(x)| d\mu(x)
\end{eqnarray*}
whenever $f$ satisfies 
\begin{eqnarray*}
||f-f_{6B_i}||_{1,6B_i} \le C_{Weak} \kappa 6 r_i ||\grad f||_{1,\kappa 6B_i}
\end{eqnarray*}
 for all $B_i \in \F(B)$.
\end{lemma}
\begin{proof}
We can write:
\begin{equation*}
\begin{split}
\mu(6B_i \cap 6B_{i+1}) & |f_{6B_i} - f_{6B_{i+1}}|  \\
& = \int_{6B_i \cap 6B_{i+1}} |f_{6B_i} - f_{6B_{i+1}}| d\mu(x) \\
& \le \int_{6B_i \cap 6B_{i+1}} |f(x)-f_{6B_i}|+|f(x)- f_{6B_{i+1}}| d\mu(x) \\
& \le \int_{6B_i} |f(x)-f_{6B_i}| d\mu(x) +\int_{6B_{i+1}}|f(x) - f_{6B_{i+1}}| 
 d\mu(x) \\
& \le C_{Weak} 6 \kappa \left(r_i \int_{6\kappa B_i}|\grad f(x)| d\mu(x) 
                + r_{i+1} \int_{6\kappa B_{i+1}}|\grad f(x)| d\mu(x) \right)\\
& \le C_{Weak} 6 \kappa \left(r_i \int_{6\kappa B_i}|\grad f(x)| d\mu(x) 
                + 2 r_{i} \int_{36\kappa B_{i}}|\grad f(x)| d\mu(x) \right)\\
& \le C_{Weak} 18 \kappa r_i \int_{36\kappa B_{i}}|\grad f(x)| d\mu(x) .
\end{split}
\end{equation*}
This string of inequalities holds by the triangle inequality, set
inclusion, the weak Poincar\'{e} inequality (Theorem \ref{weakp1}),
and the comparisons in Lemma \ref{537}.
 
By Lemma \ref{537} we know that $\mu(B_i) \le \mu(6B_i \cap
6B_{i+1})$, and so we can rewrite this to get:
\begin{eqnarray*}
\mu(B_i)|f_{6B_i} - f_{6B_{i+1}}|
  \le 18 \kappa r_i C_{Weak} \int_{36\kappa B_i} |\grad f(x)| d\mu(x).
\end{eqnarray*}
\end{proof}
Recall that we have shown $C_{Weak}=\frac{2^{3n+3} M^3
\kappa^{N+1}}{\alpha n}$, as in Theorem \ref{weakp1} for our small
balls containing only one vertex.

We are now in a position to prove our main theorem.
\begin{theorem}\label{UnifPoincare1General}
Let $E$ be set whose subsets satisfy volume doubling with constant
$C_{vol}$.  Suppose $\F$ is a Whitney type cover of $E$ and that $f$
satisfies
\begin{eqnarray*}
||f-f_{6B_i}||_{1,6B_i} \le C_{Weak} \kappa 6 r_i ||\grad f||_{1,\kappa 6B_i}
\end{eqnarray*}
 for all $B_i \in \F$.  Then
\begin{eqnarray*}
\int_E |f(x)-f_E| d\mu(x) \le P_0 r \int_E |\grad f(x)| d\mu(x)
\end{eqnarray*}
holds where 
 $ P_0=\left(1+3 C_{vol}^{1+\log_2(10^3 \kappa +9)}\right)
     K C_{Weak} 12 \kappa 10^{-3} $.
\end{theorem}
\begin{proof}
We want to bound $|f-f_E|$.  In order to do this, we will split this
quantity into two essentially similar pieces, $|f-f_{6B_z}|$ and
$|f_E-f_{6B_z}|$.  At the end of the proof, we will show that
$|f_E-f_{6B_z}|$ can be bounded by $|f-f_{6B_z}|$.  Because of this,
we only need consider $|f-f_{6B_z}|$.  We will take $f$ minus its
average on the central ball and put it into a form where we can take
advantage of the covering.  This will involve splitting this further
into chains of sufficiently small balls, and then applying the weak
Poincar\'{e} inequality to them.  After a bit of work, this will give
us the desired inequality.
 
First, we will use the fact that $\cup_{B \in \F} 2B$ covers all of $E$
to split the integral up into pieces.
\begin{eqnarray*}
\int_E |f(x)-f_{6B_z}| d\mu(x) 
&\le& \sum_{B \in \F} \int_{2B} |f(x)-f_{6B_z}| d\mu(x) \\
&\le& \sum_{B \in \F} \int_{2B} |f(x)-f_{6B}| d\mu(x)
    + \int_{2B} |f_{6B}-f_{6B_z}| d\mu(x) .
\end{eqnarray*}

The first piece can be bounded nicely using the weak Poincar\'{e}
inequality (Theorem \ref{weakp1}).
\begin{eqnarray*}
\sum_{B \in F} \int_{2B} |f(x)-f_{6B}| d\mu(x) 
  &\le& \sum_{B \in F} \int_{6B} |f(x)-f_{6B}| d\mu(x) \\ &\le& \sum_{B
  \in F} C_{Weak} 6 \kappa r_B \int_{6 \kappa B} |\grad f(x)| d\mu(x)\\
  &\le& K C_{Weak} 6 \kappa 10^{-3} r_{E} \int_{E} |\grad f(x)| d\mu(x) .
\end{eqnarray*}

The last part of the inequality follows from the fact that $6 \kappa B
\subset E$ (by Lemma \ref{536}), and at most $K$ balls in $6 \kappa \F$ 
overlap any given point in $E$.

The second piece can be rewritten as:
\begin{eqnarray*}
\sum_{B \in \F}\int_{2B} |f_{6B}-f_{6B_z}| d\mu(x) 
&=&\sum_{B \in \F}\mu(2B) |f_{6B}-f_{6B_z}|\\
&\le& \sum_{B \in \F}C_{vol} \mu(B) |f_{6B}-f_{6B_z}|.
\end{eqnarray*}

Now let us consider what happens when we fix $B$.  We have a chain,
$\F(B)$, connecting $B$ to the central ball; we can use this and Lemma
\ref{539} to find:
\begin{eqnarray*}
|f_{6B}-f_{6B_z}| &\le& \sum_{i=0}^{\varl -1} |f_{6B_i}-f_{6B_{i+1}}| \\
&\le& \sum_{i=0}^{\varl -1} C_{Weak} 18 \kappa \frac{r_i}{\mu(B_i)} 
    \int_{36\kappa B_i} |\grad f(x)| d\mu(x) \\
&=& \sum_{A \in \F(B)} C_{Weak} 18 \kappa \frac{r_A}{\mu(A)} 
    \int_{36\kappa A} |\grad f(x)| d\mu(x).
\end{eqnarray*}

By lemma \ref{538} we know that $B \subset (10^3 \kappa +9) A$ for any
$A\in \F(B)$, and so we have $\chi_B = \chi_B \chi_{(10^3 \kappa +9) A}$.
Multiplying the previous inequality by this, summing over the $B$, and
then integrating over $E$ gives us:
\begin{equation*}
\begin{split}
\int_E &\sum_{B \in \F} |f_{6B}-f_{6B_z}|\chi_B(y) d\mu(y)
\\ &\le \int_E \sum_{B \in \F}\sum_{A \in \F(B)} 
      C_{Weak} 18 \kappa \frac{r_A}{\mu(A)} 
        \int_{36\kappa A} |\grad f(x)| d\mu(x) \chi_B(y) \chi_{(10^3 \kappa +9) A}(y) d\mu(y).
\end{split}
\end{equation*}
Since the $B$ are disjoint, we have $\sum_{B \in \F} \chi_B(y) \le 1$.
This allows us to simplify the right hand side.  We can then integrate.
\begin{eqnarray*}
... &\le& \int_E \sum_{A \in \F} 
      C_{Weak} 18 \kappa \frac{r_A}{\mu(A)} 
        \int_{36\kappa A} |\grad f(x)| d\mu(x) \chi_{(10^3 \kappa +9) A}(y) d\mu(y) \\
&=& \sum_{A \in \F} 
      C_{Weak} 18 \kappa \frac{r_A\mu((10^3 \kappa +9)A)}{\mu(A)} 
        \int_{36\kappa A} |\grad f(x)| d\mu(x).
\end{eqnarray*}
Volume doubling gives us:
\begin{eqnarray*}
&\le& \sum_{A \in \F} 
      C_{Weak} 18 \kappa r_{A}  C_{vol}^{\log_2(10^3 \kappa +9)}  
       \int_{36\kappa A} |\grad f(x)| d\mu(x).       
\end{eqnarray*}
We then use the bound from (4) to see:
\begin{eqnarray*}
&\le& C_{Weak} 18 \kappa 10^{-3} r_{E}
    C_{vol}^{\log_2(10^3 \kappa +9)} K \int_{E} |\grad f(x)| d\mu(x).
\end{eqnarray*}
Putting all of this together and factoring, our original inequality becomes:
\begin{equation*}
\begin{split}
\int_E |f(x)-f_{6B_z}| &d\mu(x) \\ &\le 
 \left(1+3C_{vol}^{1+\log_2(10^3 \kappa +9)} \right)
   K C_{Weak} 6 \kappa 10^{-3}
   r_{E} \int_{E} |\grad f(x)| d\mu(x) .
\end{split}
\end{equation*}

Let $\frac{1}{2}P_0=\left(1+3 C_{vol}^{1+\log_2(10^3 \kappa +9)}\right)
   K C_{Weak} 6 \kappa 10^{-3}$.  Then we can rewrite the inequality as:
\begin{eqnarray*}
\int_E |f(x)-f_{6B_z}| d\mu(x) 
  \le \frac{1}{2}P_0 r_{E} \int_{E} |\grad f(x)| d\mu(x) .
 \end{eqnarray*}

All that remains is to switch from $f_{6B_z}$ to the average on the
entire set, $f_E$.
\begin{eqnarray*}
\int_E |f_E-f_{6B_z}| d\mu(x) &=& \mu(E) |f_E-f_{6B_z}|\\
&=& \mu(E) \abs{\frac{1}{\mu(E)} \int_E f(x) -f_{6B_z} d\mu(x) }\\
&\le& \int_E |f(x) -f_{6B_z}| d\mu(x) \\
&\le& \frac{1}{2}P_0 r_{E} \int_{E} |\grad f(x)| d\mu(x).
\end{eqnarray*}

Thus, the Poincar\'{e} inequality holds on the ball $E=B(z,r)$.
\begin{eqnarray*}
\int_E |f(x)-f_E| d\mu(x) 
  &\le& \int_E |f(x)-f_{6B_z}| d\mu(x) +\int_E |f_{6B_z}-f_{E}| d\mu(x) \\
  &\le& P_0 r_{E} \int_{E} |\grad f(x)| d\mu(x) .
\end{eqnarray*}
\end{proof}
\begin{corollary}\label{UnifPoincare1}
Let $X$ be an admissible n-dimensional Euclidean complex with degree
bounded above by $M$, solid angle bounded by $\alpha$, and edge
lengths bounded below by $\varl$.  Let $E=B(z,r)$ where
$r<R_0:=\frac{\varl}{\kappa}$.  Then
\begin{eqnarray*}
\int_E |f(x)-f_E| d\mu(x) \le P_0 r \int_E |\grad f(x)| d\mu(x)
\end{eqnarray*}
holds for $f \in \Lip(X) \cap L^1(E)$ where
$\kappa=6(\frac{2}{\sqrt{2(1-\cos(\alpha))}}+1)^n$ and \\ $P_0=(1+3
M^3 2^n (10^3 \kappa +9)^n)M (8(1+10^{3}\kappa))^n C_{Weak} 6 \kappa$.
\end{corollary}
\begin{proof}
Apply Theorem \ref{UnifPoincare1General} with $K=M(8(1+10^{3}\kappa))^n$ and
$C_{vol} = M2^n$.
\end{proof}
\begin{corollary}\label{UnifPoincareP}
Let $X$ be an admissible n-dimensional Euclidean complex with degree
bounded above by $M$, solid angle bounded by $\alpha$, and edge
lengths bounded below by $\varl$. For $f \in \Lip(X) \cap L^p(E)$ and
$r<R_0$ we have
\begin{eqnarray*}
\inf_c ||f-c||_{p,E} \le p P_0 r ||\grad f||_{p,E}
\end{eqnarray*}
where $E$ is a ball of radius $r$ and $1\le p < \infty$.
Note that this implies:
\begin{eqnarray*}
||f-f_E||_{p,E} \le 2 p P_0 r ||\grad f||_{p,E}.
\end{eqnarray*}
Here $P_0=(1+3 M^3 2^n (10^3 \kappa +9)^n)M (8(1+10^{3}\kappa))^n C_{Weak}
6 \kappa$, 
$C_{Weak}=\frac{2^{3n+3} M^3 \kappa^{n+1}}{\alpha n}$, and 
$\kappa=6(\frac{2}{\sqrt{2(1-\cos(\alpha))}}+1)^n$. 
\end{corollary}
\begin{proof}
Apply Lemma \ref{PoincareP} to Corollary \ref{UnifPoincare1}.
\end{proof}
\begin{corollary}\label{PoincareExtend1}
Assume $p=1$ Poincar\'{e} inequality $||f-f_B||_{1,B} \le P_0 r
||\grad f||_{1,B}$ holds for $f \in \Lip(X) \cap L^p(E)$ on balls $B=
B(x,r)$ with $r \le R$.  Assume volume doubling holds with constant
$C_{vol}$ for balls with radius less than $C_0 R$.  Then
\begin{eqnarray*}
||f-f_E||_{p,E} \le p 2 P_0 \left( 6 (1+3C_{vol}^{11}) C_{vol}^{13}
10^{-3} \right)^{\ceil{\log_{\frac{10^3}{6}}(C_0)}} r_E ||\grad f||_{p,E}
\end{eqnarray*}
also holds for balls $E$ with radius less than $C_0 R$ and $1 \le p < \infty$.
\end{corollary}
\begin{proof}
Note that if $E=B(x,\frac{1}{6} 10^3r)$, then the $p=1$ Poincar\'{e}
inequality holds for all balls in the Whitney cover dilated by a
factor of 6 with $\kappa =1$.  We can apply Theorem
\ref{UnifPoincare1General} which gives us
\begin{eqnarray*}
||f-f_E||_{1,E} \le 6 (1+3C_{vol}^{11}) C_{vol}^{13} 10^{-3} P_0 r_E
  ||\grad f||_{1,E}.
\end{eqnarray*}
In particular, we can repeat this to show that the $p=1$ Poincar\'{e}
inequality holds for balls up to radius $C_0 R$ with constant $( 6
(1+3C_{vol}^{11}) C_{vol}^{13} 10^{-3})^{\ceil{\log_{\frac{10^3}{6}}(C_0)}}
P_0$.
To get the $p$ Poincar\'{e} inequality, apply lemma \ref{PoincareP}.  
\end{proof}

Note that bounds on degree $M$, angles $\alpha$ and edge lengths
$\varl$ give us uniform local volume doubling on our complex, $X$.
For any fixed $R$ we can then apply Lemma \ref{PoincareP} to Corollary
\ref{PoincareExtend1} to get the standard $L^2$ Poincar\'{e}
inequality for balls of radius up to $R$.  In general, this cannot be
extended to $R = \infty$; note that the constant in the new
Poincar\'{e} inequality goes to infinity as $C_0$ goes to infinity.

\chapter{Small Time Heat Kernel Estimates for X}
  The heat kernel, $h_t(x,y)$, is the fundamental solution to the heat equation
\begin{eqnarray*}
\del_t u = \Delta u.
\end{eqnarray*}
Note that our formulation does not have factors of $-1$ or
$\frac{1}{2}$, which appear in some of the literature.  This type of
differential equation is parabolic; one way of obtaining information
about it is through parabolic Harnack inequalities.  Sturm \cite{Sturm}
shows that local volume doubling and Poincar\'{e} inequalities on a
subset of a complete metric space imply a local parabolic Harnack
inequality on that subset.  He then uses this to find Gaussian
estimates on the heat kernel.  The equivalence of the parabolic
Harnack inequality with Poincar\'{e} and volume doubling had
previously been done in the Riemannian manifold case by Grigor'yan
\cite{Grigor} and Saloff-Coste \cite{LSCNoteOn}.  

\section{Small time Heat Kernel Asymptotics}
We have shown a uniform local Poincar\'{e} inequality, and our complex
is both complete and locally satisfies volume doubling.  This tells us
that we've satisfied the hypotheses of the following theorem of
Sturm \cite{Sturm} which gives a lower bound on the diagonal.
\begin{theorem}[Sturm]\label{Sturm1}
Assume $Y$ is an open subset of a complete space $X$ that admits a
Poincar\'{e} inequality with constant $C_P$ and volume doubling with
constant $2^N$.  Then there exists a constant $C=C(C_P,N)$ such that
\begin{eqnarray*}
h_t(x,x) \ge \frac{1}{C \mu(B(x,\sqrt{t}))}
\end{eqnarray*}
for all $x \in Y$ and all $t$ such that $0<t<\rho^2(x,X-Y)$.
\end{theorem}
Here, $\rho$ refers to the intrinsic distance.  In our complex, this
will always satisfy $\rho(x,y) \ge d(x,y)$.  Also note that since we
have a uniform local Poincar\'{e} inequality and a uniform local
volume doubling constant we have the following corollary:
\begin{corrolary}\label{Corge}
Let $X$ be an admissible $n$-dimensional Euclidean complex with degree
bounded above by $M$, solid angle bounded by $\alpha$, and edge
lengths bounded below.  For any $R_0 > 0$ there is a corresponding
constant $C=C(X,R_0)$ so that
\begin{eqnarray*}
h_t(x,x) \ge \frac{1}{C M \mu(S^{(n-1)}) t^{n/2}}
\end{eqnarray*}
for all $x \in X$ and all $t$ such that $0<t<R_0^2$.
\end{corrolary}
\begin{proof}
For each $x\in X$ apply \ref{Sturm1} to $X$ with $Y=B(x,R_0)$.  The
distance compares easily: $\rho(x,X-Y) \ge d(x,X-Y) = R_0$.  Because
the constant $C$ in \ref{Sturm1} depends only on $C_P$ and $N$, we can
use the fact that our constants $C_P$ and $N$ do depend only on the
radius of our ball to obtain a universal constant, $C$.
\end{proof}
Sturm \cite{Sturm} also proves an upper bound for the heat kernel;
this bound is especially useful near the diagonal.
\begin{theorem}[Sturm]\label{Sturm2}
Assume  $Y$ is an open subset of a complete space $X$ that admits a
Poincar\'{e} inequality with constant $C_P$ and volume doubling with
constant $2^N$.  Then there exists a constant $C=C(C_P,N)$ such that
for every $x,y\in Y$
\begin{eqnarray*}
h_t(x,y) \le \frac{Ce^{-\frac{\rho^2(x,y)}{4t}}}{\sqrt{\mu(B(x,\sqrt{T})) \mu(B(y,\sqrt{T}))}} 
\left(1+\frac{\rho^2(x,y)}{t}\right)^{N/2} 
e^{-\lambda t}(1+\lambda t)^{1+N/2}
\end{eqnarray*}
where $\rho$ is the intrinsic distance,
$R=\inf(\rho(x,X-Y),\rho(y,X-Y))$, $T=\min(t,R^2)$ and $\lambda$ is
the bottom of the spectrum of the self-adjoint operator $-L$ on
$L^2(X,\mu)$.
\end{theorem}
In general, we can replace $\lambda$ with $0$, which increases the
value of the right hand side.  In our setting, we can simplify this a
bit more.
\begin{corrolary}\label{Corle}
Let $X$ be an admissible $n$-dimensional Euclidean complex with degree
bounded above by $M$, solid angle bounded by $\alpha$, and edge
lengths bounded below.  Then for any $R_0$ we there exists a constant
$C = C(X,R_0)$ so that for any $x \in X$ and $t>0$ we have:
\begin{eqnarray*}
h_t(x,x) \le \frac{C}{M \mu(S^{(n-1)}) (\min(t,R_0^2))^{n/2}} .
\end{eqnarray*}
\end{corrolary}
\begin{proof}
For each $x\in X$ apply \ref{Sturm2} to $X$ with $Y=B(x,R_0)$.  The
distance compares easily: $\rho(x,X/Y) \ge d(x,X/Y) = R_0$.  The
$d(x,x)$ terms drop out, as do the $\lambda$ terms.  Because the
constant $C$ in \ref{Sturm2} depends only on $C_P$ and $N$, we can use
the fact that our constants $C_P$ and $N$ do not depend on our
specific choice of ball to obtain a universal constant, $C$.
\end{proof}
\begin{corrolary}\label{offdiagonalHeat}
Let $X$ be an admissible $n$-dimensional Euclidean complex with degree
bounded above by $M$, solid angle bounded by $\alpha$, and edge
lengths bounded below.  For any $R_0$ there exists a $C=C(X,R_0)$ so that for
any $x,y \in X$ and $t>0$ we have:
\begin{eqnarray*}
h_t(x,y) \le \frac{C}{M \mu(S^{(n-1)}) (\min(t,R_0^2))^{n/2}}
e^{-\frac{d^2(x,y)}{4t}}\left(1+\frac{d^2(x,y)}{t}\right)^{N/2} .
\end{eqnarray*}
\end{corrolary}
\begin{proof}
For each $x,y\in X$ apply \ref{Sturm2} to $X$ with $Y=B(x,R_0) \cup
B(y,R_0)$.  The distance compares easily: $\rho(x,X/Y) \ge d(x,X/Y) =
R_0$.  Because the constant $C$ in \ref{Sturm2} depends only on $C_P$
and $N$, we can use the fact that our constants $C_P$ and $N$ do not
depend on our specific choice of ball to obtain a universal constant,
$C$.
\end{proof}
Note that we can rewrite this as a bound of the following form for
some constants $C,c$:
\begin{eqnarray*}
h_t(x,y) \le \frac{C}{(\min(t,R_0^2))^{n/2}} e^{-c\frac{d^2(x,y)}{t}}.
\end{eqnarray*}
\begin{corrolary} \label{ComplexIsJustLikeRn}
For an admissible $n$-dimensional Euclidean complex $X$ with degree
bounded above by $M$, solid angle bounded by $\alpha$, and edge
lengths bounded below, on $X^{(k)}$ we have
\begin{eqnarray*}
\frac{1}{C_k t^{k/2}} 
\le h_t^k(x,x) \le \frac{C_k}{t^{k/2}} .
\end{eqnarray*}
This holds for all $t<R_0^2$ and $x\in X^{(k)}$, where $C_k$ depends
on $R_0$, $\alpha$, $M$, $k$, and $\inf_{v,w \in X^{(0)}} d(v,w)$.  In
particular, we can take $C= \max_{k=1..n} C_k$ to have a uniform
constant for each $X^{(k)}$.
\end{corrolary}
\begin{proof}
Apply Corollaries \ref{Corge} and \ref{Corle} to $X^{(k)}$.  This
holds because $X^{(k)}$ is also an admissible complex satisfying the
same bounds as $X$.  $C_k$ varies slightly in each dimension due to
the effect of dimension on volume doubling, and hence Poincar\'{e}.
\end{proof}
 Off diagonal, the lower bound is more complicated.
\begin{theorem}[Sturm]\label{Sturm3}
Assume $Y$ is an open subset of a complete space $X$ that admits a
Poincar\'{e} inequality with constant $C_P$ and volume doubling with
constant $2^N$.  Then there exists a constant $C=C(C_P,N)$ such that
for every $x,y\in Y$ which are joined by a curve $\gamma$ of length
$\rho(x,y)$
\begin{eqnarray*}
h_t(x,y) \ge \frac{1}{C \mu(B(x,\sqrt{T})) } 
e^{-C\frac{\rho^2(x,y)}{t}} e^{-\frac{Ct}{R^2}}
\end{eqnarray*}
where $\rho$ is the intrinsic distance, $R=\inf_{0\le s \le
1}(\rho(\gamma(s),X-Y))$, and $T=\min(t,R^2)$.
\end{theorem}
In our setting, we can find a near-diagonal lower bound for any complex.
\begin{corrolary}\label{offdiagonalHeatlower}
Let $X$ be an admissible $n$-dimensional Euclidean complex with degree
bounded above by $M$, solid angle bounded by $\alpha$, and edge
lengths bounded below.  For any $R_0>0$ there exists a $C=C(X,R_0)$ so that for
any $x,y \in X$ with $d(x,y)<R_0$ and $t>0$ we have:
\begin{eqnarray*}
h_t(x,y) \ge \frac{1}{C \mu(B(x,\sqrt{\min(t,R_0^2)})) } 
e^{-C\frac{d^2(x,y)}{t}} e^{-\frac{Ct}{R_0^2}}.
\end{eqnarray*}
If $X$ is volume doubling and has a global Poincar\'{e} inequality, we
can set $R_0 =\infty$ to get:
\begin{eqnarray*}
h_t(x,y) \ge \frac{1}{C \mu(B(x,\sqrt{t}))} e^{-C\frac{d^2(x,y)}{t}}.
\end{eqnarray*}
\end{corrolary}
\begin{proof}
For each $x,y\in X$ apply \ref{Sturm3} to $X$ with $Y=B(x,2R_0)$.
Since we have a length space, the distance compares easily:
$\rho(\gamma(s),X/Y) \ge d(\gamma(s),X/Y) = R_0$.  Because the
constant $C$ in \ref{Sturm2} depends only on $C_P$ and $N$, we can use
the fact that our constants $C_P$ and $N$ do not depend on our
specific choice of ball to obtain a universal constant, $C$.
\end{proof}
\section{Examples}
\begin{example}
In $R^1$ the heat kernel, $h_t(x,y)$ is the density for the transition
probability of Brownian motion.  
\begin{eqnarray*}
h_t(x,y) = \frac{1}{\sqrt{4\pi t}} e^{-\frac{|x-y|^2}{4t}}.
\end{eqnarray*}
This is a normal density for $y$ with expectation $x$ and variance
$2t$.  In the probability literature, it is common for the heat
equation to be written with the time derivative multiplied by an extra
factor of $1/2$ so that the variance is $t$. See for example Feller
Volume 2 \cite{Feller2}.  

We can think of $R^1$ as an Euclidean complex.  This matches our on
diagonal bound exactly, but it is slightly nicer (by a factor of
$\sqrt{1 + d^2(x,y)/t}$ than our off diagonal bound.
\end{example}

\begin{example}
In $R^n$, the heat equation can be solved using either a scaling
argument or Fourier series.  Alternately, it can be thought of as an
n-dimensional version of Brownian motion.  In the PDE literature, the
heat kernel is also called the Gauss kernel or the fundamental
solution to the heat equation.  See Evans \cite{Evans} for a
derivation. 
\begin{eqnarray*}
h_t(x,y) = \frac{1}{(4\pi t)^{n/2}} e^{-\frac{|x-y|^2}{4t}}.
\end{eqnarray*}
Note that $R^n$ is also an Euclidean complex, and that this kernel is
consistent with our asymptotics.
\end{example}

\begin{example}
We can think of a circle of length $1$ as a complex consisting of
three edges of length $1/3$ joined in a triangle shape.  One can
calculate the heat kernel in terms of a sum using Fourier series; see
Dym and McKean \cite{Dym}.  The heat kernel here is
\begin{eqnarray*}
h_t(x,y) 
= \frac{1}{\sqrt{4\pi t}} \sum_{n=-\infty}^{\infty} e^{-\frac{|x-y-n|^2}{4t}}.
\end{eqnarray*}
When $x=y$, this simplifies to:
\begin{eqnarray*}
h_t(x,x) = \frac{1}{\sqrt{4\pi t}} 
           + \frac{2}{\sqrt{4\pi t}} \sum_{n=1}^{\infty} e^{-\frac{n^2}{4t}}.
\end{eqnarray*}
For small values of $t$, the dominant term is $\frac{1}{\sqrt{4\pi
t}}$.  This is the same behavior as our small time prediction.  Note
that once $t =1/4$ , the ball of radius $\sqrt{t}$ will be of size 1.
By this point in time, the asymptotic will cease to be useful.
\end{example}
  \begin{example}
We will look at the heat kernel on a star shaped graph, $X$, which has a central vertex
with $n$ edges attached to it.
\begin{figure}[h]
\centering
       \includegraphics[angle=0,width=2in]{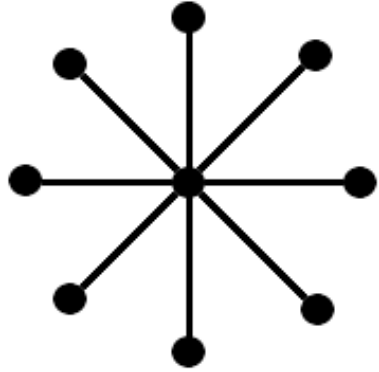}
\caption{Example of a star; here n=8.}
\end{figure}

 When we compute functions on $X$ for $x \in e(a,b)$, we will let $x$
represent $d(x,a)$. For example, $a$ would be 0, the midpoint would be
$\frac{1}{2}$, and $b$ would be 1.  All functions on $X$ are of the
form $f(x,j) = \sum_i f_i(x) I_{j=i} + f(0)I_{x=0}$ where $x \in
(0,1]$ and $j=1,2,...n$.  $f_i(x)$ represents the value of the
function along the ith leg, and $f(0)$ is the value at the center of
the star.  The measure on our star is $d\mu(x,j) = dx$.  When we
write the derivative $\frac{df}{dx}$ we will mean the usual derivative
with respect to Lebesgue measure.  $\frac{df}{d\mu}(0) = \sum_i
\frac{df_i}{dx}(0)$ for functions $f$ which are
differentiable for each $f_i$ on $(0,1]$.

We'll look at a domain where our function has zero derivative at the
boundary points and has zero derivative in the center.  \\ $\Dom(\Delta)
= \{f \in C(X): f_i \in C^1((0,1]) \text{ , and }\frac{df}{d\mu}(0)=0,
\frac{df_i}{d\mu}(1)=0 \}$; we will be using an $L^2$ norm on this
space.

The symmetric set of eigenfunctions are cosine on every leg:
\begin{eqnarray*}
\Phi_k(x) =\sqrt{\frac{2}{n}}\cos(k \pi x)
\mbox{ for }x \in e_i,\mbox{ } i=1..n,
\end{eqnarray*}
or, when $k=0$, they are a constant on every leg:
\begin{eqnarray*}
\Phi_0(x) =\frac{1}{\sqrt{n}} \mbox{ for }x \in e_i,\mbox{ } i=1..n.
\end{eqnarray*}

These functions have derivative zero at each vertex and are continuous
at the central vertex.  The coefficients are chosen so that they have
an $L^2$ norm of 1.

Note that if we look at the product of these for points $x$ and $y$ on the
star (regardless of which leg they occur on), we have
\begin{eqnarray*}
\Phi_k(x)\Phi_k(y) &=& \frac{2}{n}\cos(k \pi x)\cos(k \pi y) \\
         &=& \frac{1}{n}
              \left(\cos(k \pi (x-y)) + \cos(k \pi (x+y))\right).
\end{eqnarray*}
If we sum $e^{-\lambda^2 t}\Phi(x)\Phi(y)$ we find:
\begin{equation*}
\begin{split}
\frac{1}{n} +
\frac{1}{n}\sum_{k=1}^{\infty} e^{-(k\pi)^2t}
&   \left(\cos(k \pi (x-y)) + \cos(k \pi (x+y))\right) \\
&= \frac{1}{2n } \sum_{k=-\infty}^{\infty} e^{-(k\pi)^2t}
       \left(\cos(k \pi (x-y)) + \cos(k \pi (x+y))\right).
\end{split}
\end{equation*}

In order to simplify this, we will use Jacobi's identity (see Dym for
derivation):
\begin{eqnarray*}
\sum_{k=-\infty}^{\infty} e^{-\frac{(x-k)^2}{2s}} 
= \sqrt{2 \pi s}\sum_{k=-\infty}^{\infty}e^{-2 \pi^2 k^2 s}e^{2 \pi i k x}.
\end{eqnarray*}
Because this sums to a real number, we can rewrite it as:
\begin{eqnarray*}
\sum_{k=-\infty}^{\infty} e^{-\frac{(x-k)^2}{2s}} 
= \sqrt{2 \pi s}\sum_{k=-\infty}^{\infty}e^{-2 \pi^2 k^2 s} \cos(2\pi k x).
\end{eqnarray*}
This gives us:
\begin{equation*}
\begin{split}
\frac{1}{2n } \sum_{k=-\infty}^{\infty} e^{-(k\pi)^2t}
&       \left(\cos(k \pi (x-y)) + \cos(k \pi (x+y))\right) \\
&= \frac{1}{2n\sqrt{\pi t}} 
 \sum_{k=-\infty}^{\infty} \left( e^{-\frac{(x-y -2k)^2}{4t}} 
                               +  e^{-\frac{(x+y -2k)^2}{4t}}\right).
\end{split}
\end{equation*}

We will also have ones that form an $n-1$ dimensional basis on the
legs.  These are the ``odd'' eigenfunctions.  These will be either
sine or $0$ along the legs.  Since $\sin(0)=0$, they will be
continuous at the center.  We require the derivatives at the vertices
to be 0, and so the possible sine functions are $\sin
\left(\frac{(2k+1) \pi}{2} x\right)$.  Note that they must be
normalized according to an $L^2$ norm; this means that $\sum_{i=1}^n
\frac{1}{2} b_i^2 =1$, where the $b_i$ are the coefficients.  Since
they will need to have derivative zero at the central vertex, we will
need $\sum_{i=1}^n b_i =0$.  The combination of these two
restrictions, along with the fact that these eigenfunctions must be
orthogonal to one another, determine the coefficients.
 
For $i=1...\lfloor \frac{n}{2} \rfloor$ we have eigenfunctions of the form: 
\begin{eqnarray*}
\tilde{\Phi}_{k,i}(x) = 
      \left\{ \begin{array}{ll}             
            \sin \left(\frac{(2k+1) \pi}{2} x\right) & 
                  \text{ for } x \in e_{2i-1}  \\
            -\sin \left(\frac{(2k+1) \pi}{2} x\right) & 
                  \text{ for } x \in e_{2i}  \\
              0                            & \text{ otherwise.}
\end{array} \right. 
\end{eqnarray*}
These functions are trivially orthogonal to one another.  The factors
of $\pm 1$ give us derivative 0 at the center.

 For $i=1...\lfloor \frac{n}{2} \rfloor -1$ we have:
\begin{eqnarray*}
 \tilde{\Phi}_{k,\lfloor \frac{n}{2} \rfloor +i}(x) = 
      \left\{ \begin{array}{ll} 
     \frac{1}{\sqrt{i(i+1)}}
    \sin\left(\frac{(2k+1) \pi}{2} x\right) 
 & \text{ for } x \in e_j,  j=1..2i\\
      - \frac{i}{\sqrt{i(i+1)}}
        \sin\left(\frac{(2k+1) \pi}{2} x\right) 
 & \text{ for } x \in e_j, j=2i+1, 2i+2 \\
              0                            & \text{ otherwise.}
\end{array} \right. 
\end{eqnarray*}

In the case where there is an odd number of legs we have: 
\begin{eqnarray*}
 \tilde{\Phi}_{k,n-1}(x) = 
      \left\{ \begin{array}{ll} 
      \frac{\sqrt{2}}{\sqrt{n(n-1)}}
               \sin\left(\frac{(2k+1) \pi}{2} x\right) 
         & \text{ for } x \in e_j, j=1..n-1 \\
    - (n-1)\frac{\sqrt{2}}{\sqrt{n(n-1)}}
               \sin\left(\frac{(2k+1) \pi}{2} x\right) & \text{ for } x \in e_n
\end{array} \right. 
\end{eqnarray*}

Note that edges $2i$ and $2i-1$ have constants with the same sign,
which forces the functions to be orthogonal to the first set.  The
pattern of + and - allow them to be orthogonal to one another.  The
other factors guarantee that the derivative at the center is zero.

If we look at the product of the $\tilde{\Phi}_k$ for points $x$ and
$y$ on the star, we have
\begin{eqnarray*}
\tilde{\Phi}_k(x)\tilde{\Phi}_k(y) 
&=& c \sin\left(\frac{(2k+1) \pi}{2} x\right)
    \sin\left(\frac{(2k+1) \pi}{2} y\right) \\
&=& \frac{c}{2}\left(\cos\left((2k+1)\pi \frac{x-y}{2}\right) 
                   -\cos\left((2k+1) \pi\frac{x+y}{2}\right)\right).
\end{eqnarray*}

If we sum these ``odd'' eigenfunctions, multiplied by
$e^{-\frac{(2k+1)^2 \pi^2}{4}t}$, we have:
\begin{eqnarray*}
\sum_{k=0}^{\infty} e^{-\frac{(2k+1)^2 \pi^2}{4}t}\frac{c}{2}
      \left(\cos\left((2k+1)\pi \frac{x-y}{2}\right) 
           -\cos\left((2k+1) \pi\frac{x+y}{2}\right)\right).
\end{eqnarray*}
We can rewrite the $x-y$ terms as follows; a similar calculation will
work for the $x+y$ terms.  First we add in terms with $2k$.  We note
that cosine is an even function, and so we can extend this to negative
$k$.  We also have the zero term, since it will cancel between the two
sums.
\begin{equation*}
\begin{split}
\sum_{k=0}^{\infty} &e^{-((2k+1)\pi)^2\frac{t}{4}}\frac{c}{2}
      \cos\left((2k+1)\pi \frac{x-y}{2}\right) \\
+&\sum_{k=1}^{\infty} e^{-(2k\pi)^2\frac{t}{4}}\frac{c}{2}
       \cos\left(2k\pi \frac{x-y}{2}\right) 
  -\sum_{k=1}^{\infty} e^{-(2k\pi)^2\frac{t}{4}}\frac{c}{2}
        \cos\left(2k\pi \frac{x-y}{2}\right)  \\
&=\sum_{k=1}^{\infty} e^{-(k\pi)^2\frac{t}{4}}\frac{c}{2}
       \cos\left(k\pi \frac{x-y}{2}\right)  
  -\sum_{k=1}^{\infty} e^{-(k\pi)^2t}\frac{c}{2}
         \cos\left(k\pi(x-y)\right) \\
&=\sum_{k=-\infty}^{\infty} e^{-(k\pi)^2\frac{t}{4}}\frac{c}{4}
       \cos\left(k\pi \frac{x-y}{2}\right) 
  -\sum_{k=-\infty}^{\infty} e^{-(k\pi)^2t}\frac{c}{4}
         \cos\left(k\pi(x-y)\right). 
\end{split}
\end{equation*}

We will apply Jacobi's identity to each of the sums.  The first
sum yields:
\begin{eqnarray*}
\frac{c}{4}\sum_{k=-\infty}^{\infty} e^{-(k\pi)^2\frac{t}{4}}
\cos\left(2k\pi \frac{x-y}{4}\right) 
&=& \frac{c}{4\sqrt{2\pi\frac{t}{8}}}\sum_{k=-\infty}^{\infty}
e^{-\frac{(\frac{x-y}{4} -k)^2}{\frac{2t}{8}}} 
 \\
&=&\frac{c}{2\sqrt{\pi t}}\sum_{k=-\infty}^{\infty}
e^{-\frac{(x-y -4k)^2}{4t}} .
\end{eqnarray*}
The second sum gives us:
\begin{eqnarray*}
-\frac{c}{4}\sum_{k=-\infty}^{\infty} e^{-(k\pi)^2t}
 \cos\left(2k\pi \frac{x-y}{2}\right) 
=\frac{c}{4\sqrt{\pi t}}\sum_{k=-\infty}^{\infty}
-e^{-\frac{(x-y -2k)^2}{4t}}.
\end{eqnarray*}

Similarly, for the $x+y$ terms we have:
\begin{equation*}
\begin{split}
-\sum_{k=1}^{\infty} e^{-\frac{(2k+1)^2 \pi^2}{4}t} &\frac{c}{2}
     \cos\left((2k+1) \pi\frac{x+y}{2}\right) \\
&= 
-\frac{c}{2\sqrt{\pi t}}\sum_{k=-\infty}^{\infty}
 e^{-\frac{(x+y -4k)^2}{4t}} 
+ \frac{c}{4\sqrt{\pi t}}\sum_{k=-\infty}^{\infty}
e^{-\frac{(x+y -2k)^2}{4t}}.
\end{split}
\end{equation*}
To get the heat kernel, we sum the $e^{-\lambda^2 t}\Phi(x)\Phi(y)$:
\begin{eqnarray*}
h_t(x,y)&=&\frac{1}{2\sqrt{\pi t} n } 
 \sum_{k=-\infty}^{\infty}\left( e^{-\frac{(x-y -2k)^2}{4t}} 
 +  e^{-\frac{(x+y -2k)^2}{4t}}\right) \\
& &+ \frac{c}{2\sqrt{\pi t}}\sum_{k=-\infty}^{\infty}
\left(e^{-\frac{(x-y -4k)^2}{4t}} -e^{-\frac{(x+y -4k)^2}{4t}}\right) \\
& &+ \frac{c}{4\sqrt{\pi t}}\sum_{k=-\infty}^{\infty}
\left(e^{-\frac{(x+y -2k)^2}{4t}} -e^{-\frac{(x-y -2k)^2}{4t}}\right).
\end{eqnarray*}
Note that the value of $c$ depends on which edges $x$ and $y$ are on.
If they are on the same edge, $c=2\left(1-\frac{1}{n}\right)$.  If $x$
and $y$ are on different edges, then $c=-\frac{2}{n}$.

We are now in a position to see what happens on the star near $t=0$.
On the diagonal, the heat kernel will limit to infinity as $t$ goes to
zero with the following asymptotics:
\begin{eqnarray*}
h_t(0,0) & \approx& \frac{1}{n\sqrt{\pi t} }, \\
h_t(1,1) & \approx& \frac{1}{\sqrt{\pi t}}, \text{ and} \\
h_t(x,x) & \approx& \frac{1}{2\sqrt{\pi t}} \text{ for }x \ne 0,1.
\end{eqnarray*}

To determine what happens for $t$ near zero when we have two different
points, we need to consider the relative positions of $x$ and $y$.  We
know that the heat kernel will limit to zero as $t$ goes to zero.
Without loss of generality, we will look at when $d(x,0) < d(y,0)$.

When $x=0$ and $y \ne 1$, the $k=0$ terms will dominate.  They give
us:
\begin{eqnarray*}
h_t(0,y)  \approx
\frac{1}{n\sqrt{\pi t}} e^{-\frac{y^2}{4t}} .
\end{eqnarray*}

When $x=0$ and $y=1$, we will only have terms involving
$e^{-\frac{1}{4t}}$ contributing.  After cancellation, this gives us:
\begin{eqnarray*}
h_t(0,1)  \approx
\frac{2}{n\sqrt{\pi t}} e^{-\frac{1}{4t}}.
\end{eqnarray*}

When $x\ne 0$, the dominant terms will involve
$e^{-\frac{(x-y)^2}{4t}}$.  If the coefficient for those terms is
zero, we then have $e^{-\frac{(x+y)^2}{4t}}$ instead.  The relevant
terms are:
\begin{eqnarray*}
h_t(x,1) &\approx &
\left(\frac{2}{n}-c\right)\frac{1}{2\sqrt{\pi t}}
e^{-\frac{(x+1)^2}{4t}} 
+
\left(\frac{2}{n}+c\right)\frac{1}{2\sqrt{\pi t}}
e^{-\frac{(x-1)^2}{4t}}, \mbox{ and} \\
h_t(x,y) &\approx &
\left(\frac{2}{n}-c\right)\frac{1}{4\sqrt{\pi t}}
e^{-\frac{(x+y)^2}{4t}} 
+
\left(\frac{2}{n}+c\right)\frac{1}{4\sqrt{\pi t}}
e^{-\frac{(x-y)^2}{4t}} \mbox{ for }x \ne y.
\end{eqnarray*}
When $x$ and $y$ are on different edges, $c=-\frac{2}{n}$, and so all
of the terms with coefficient $\frac{2}{n}+c$ disappear.  
Note that the notation gives us $d(x,y)=x+y$:
\begin{eqnarray*}
h_t(x,1) &\approx&
  \frac{2}{n\sqrt{\pi t}} e^{-\frac{(x+1)^2}{4t}} 
  \text{ for }x \ne 1 \text{, and} \\
h_t(x,y) &\approx&
\frac{1}{n\sqrt{\pi t}} e^{-\frac{(x+y)^2}{4t}}  \text{ for }x \ne y, y \ne 1.
\end{eqnarray*}
When $x$ and $y$ are on the same edge, $\frac{2}{n} + c = \frac{2}{n}
+ 2\left(1-\frac{1}{n}\right) =2$.  This gives us the same asymtotic as in
$\R$.
\begin{eqnarray*}
h_t(x,1) &\approx &
\frac{1}{\sqrt{\pi t}} e^{-\frac{(x-1)^2}{4t}}, \mbox{ and} \\
h_t(x,y) &\approx &
\frac{1}{2\sqrt{\pi t}} e^{-\frac{(x-y)^2}{4t}} \mbox{ for }x \ne y.
\end{eqnarray*}
  \end{example}
  \begin{example}
Consider the graph consisting of two central vertices which are joined
by one edge with edges coming off of them.  We will calculate the
behaviour of the heat kernel between the two central vertices for
small times.

\begin{figure}[h]
\centering
       \includegraphics[angle=0,width=2in]{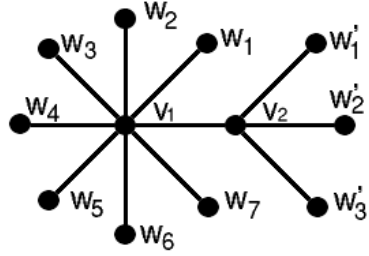}
\caption{Example of two joined stars with labels; here n=7 and m=3.}
\end{figure}

Consider two star-like graphs joined together by a central edge,
$e(v_1,v_2)$.  The star centered at $v_1$ connects to $n$ edges in
addition to the central edge.  These are labeled $e(v_1,w_i)$
$i=1..n$.  For the star centered at $v_2$, there are $m$ such edges
which are labeled $e(v_2,w_i')$ $i=1..m$.  This graph has three different
types of eigenfunctions.  When we compute these functions for $x \in
e(a,b)$, we will let $x$ represent $d(x,a)$. So $a$ would be 0, the
midpoint would be $\frac{1}{2}$, and $b$ would be 1.

The first kind correspond to eigenvalues of the form $\frac{(2k+1)^2
\pi^2}{4}$.  They form an $m-1+n-1$ dimensional space.  The
corresponding eigenfunctions are of the form $c(e(v,w))
\sin(\frac{(2k+1) \pi}{2} x)$ on the edges $e(v_1,w_i)$ and
$e(v_2,w_i')$ and are 0 on the central edge, $e(v_1,v_2)$.  Since each
eigenfunction is zero on the central edge, they will not contribute
when we compute $h_t(x,y)$ for $x \in e(v_1,v_2)$.

The second kind correspond to eigenvalues of the form $(k\pi)^2$.
When $k\ne 0$, these are $\pm \frac{\sqrt{2}}{\sqrt{m+n+1}}\cos(k\pi
x)$ along each edge in the graph, with sign chosen to preserve
continuity.  We write them:
\begin{eqnarray*}
\Phi_k(x) = \left\{ \begin{array}{l} 
       \frac{\sqrt{2}}{\sqrt{m+n+1}}\cos(k\pi x) 
        \mbox{ for }x \in e(v_1,v_2)\mbox{ or }e(v_1,w_i)\mbox{ }i=1..n \\
       \sgn(\cos(k \pi)) \frac{\sqrt{2}}{\sqrt{m+n+1}}\cos(k\pi x) 
        \mbox{ for }x \in e(v_2,w_i)\mbox{ }i=1..m. 
\end{array} \right. 
\end{eqnarray*}
The product of this function with itself at the vertices $v_1$ and $v_2$ is
\begin{eqnarray*}
\Phi_k(v_1)\Phi_k(v_2) = \frac{2}{m+n+1}\cos(k\pi).
\end{eqnarray*}
 When $k=0$, we have 
$\Phi_0(v_1)\Phi_0(v_2) = \frac{1}{m+n+1}$.

If we look at the sum of $\Phi_k(v_1)\Phi_k(v_2)e^{-k^2\pi^2 t}$,
we can write it as follows:
\begin{eqnarray*}
\sum_{k=0}^{\infty} \Phi_k(v_1)\Phi_k(v_2)e^{-k^2\pi^2 t} 
&=& \frac{1}{m+n+1} 
  +\sum_{k=1}^{\infty} \frac{2}{m+n+1}\cos(k\pi) e^{-k^2\pi^2 t}\\ 
&=& \frac{1}{m+n+1} \sum_{k=-\infty}^{\infty}\cos(k\pi) e^{-k^2\pi^2 t}.
\end{eqnarray*}

The third kind are more complicated.  They are of the form
\begin{eqnarray*}
f(x) = c_1(e)\left(\sin\left(\sqrt{\lambda}(1-x)\right) 
      + c_2(e)\sin\left(\sqrt{\lambda}x\right)\right)
\end{eqnarray*}
along each edge, where $c_1(e)$ and $c_2(e)$ will depend on the edge
$e$; see \cite{Kuc}.  For these to be in the domain, they must be
continuous and have zero derivative at each vertex.  We have
$2(m+n+1)$ constants $c_i(e)$ to determine, as well as the possible
values of $\lambda$.

To guarantee continuity, we need the function to have the same value
at $v_1$ regardless of which edge we're considering.  To get this, we
set 
\begin{align*}
c_1(e(v_1,w_i)) = c_1(e(v_1,v_2)).
\end{align*}
  Similarly, for continuity at
$v_2$, we need 
\begin{align*}
c_1(e(v_2,w_i')) = c_1(e(v_1,v_2))c_2(e(v_1,v_2)).
\end{align*}
This brings us down to $m+n+2$ different $c_i(e)$ that we need to
determine.

Now we need zero derivative at the vertices.  Note that for an edge $e$
\begin{eqnarray*}
f'(x) = 
\sqrt{\lambda} c_1(e)\left(-\cos\left(\sqrt{\lambda}(1-x)\right) 
+c_2(e)\cos\left(\sqrt{\lambda}x\right)\right).
\end{eqnarray*}

For edges of the form $e(v,w)$, we find that setting $c_2(e(v,w))
=\frac{1}{\cos\left(\sqrt{\lambda}\right)}$ will allow $f$ to satisfy
$f'(w)=0$.  For $f'(v_1)=0$, we need $c_2(v_1,v_2)
=\frac{(n+1)\cos^2\left(\sqrt{\lambda}\right)-n}
{\cos\left(\sqrt{\lambda}\right)}$.  We now have only two things that
can be used to determine our function; $c_1(e(v_1,v_2))$ and
$\lambda$.  Varying the value of $c_1(e(v_1,v_2))$ will multiply the
entire function by a constant.  We'll need to determine $\lambda$ in
order to have $f'(v_2)=0$.  For this, we need to solve
\begin{eqnarray*}
\sqrt{\lambda} c_1(e(v_1,v_2)) 
\left( 
 m\frac{(n+1)\cos^2\left(\sqrt{\lambda}\right)-n}
       {\cos\left(\sqrt{\lambda}\right)}
 \left(
  -\cos\left(\sqrt{\lambda}\right) +\frac{1}{\cos\left(\sqrt{\lambda}\right)} 
 \right) \right.& \\
+ \left. (-1)
 \left(
-1 +\frac{(n+1)\cos^2\left(\sqrt{\lambda}\right)-n}
         {\cos\left(\sqrt{\lambda}\right)}
        \cos\left(\sqrt{\lambda}\right) 
 \right) 
\right) &=0.
\end{eqnarray*}
The $(-1)$ comes from the fact that we want the sum of the inward
pointing derivatives along the edges to sum to zero.  We can simplify this to:
\begin{eqnarray*}
m
 \left(
 - (n+1)\cos^2\left(\sqrt{\lambda}\right) +2n+1 
 -\frac{n}{\cos^2\left(\sqrt{\lambda}\right)}
 \right) 
+ 
 \left(
n+1 - (n+1)\cos^2\left(\sqrt{\lambda}\right)
 \right)  \\
=0.
\end{eqnarray*}
This can be written as a quadratic equation in
$\cos^2\left(\sqrt{\lambda}\right)$; when solved, the only possibility
for $\cos^2\left(\sqrt{\lambda}\right)$ in $[0,1]$ is
$\cos^2\left(\sqrt{\lambda}\right) = \frac{mn}{(m+1)(n+1)}$.

Finally, we set $c_1(e(v_1,v_2))$ so that $||f||_2=1$.  This holds
when \\ $c_1(e(v_1,v_2))= \frac{\sqrt{m(m+1)}}{m+n+1}$.

Putting all of this information together, we find that the third type
of eigenfunction has the form:
\begin{eqnarray*}
\tilde{\Phi}_k(x) = 
\left\{ \begin{array}{l} 
\frac{\sqrt{m(m+1)}}{m+n+1}
\left(\sin\left(\sqrt{\lambda_k}(1-x)\right)+
  \frac{(n+1)\cos^2\left(\sqrt{\lambda_k}\right)-n}
       {\cos\left(\sqrt{\lambda_k}\right)} 
  \sin\left(\sqrt{\lambda_k}x\right)\right)
\\ \mbox{       for } x \in e(v_1,v_2) \\
\frac{\sqrt{m(m+1)}}{m+n+1}
\left(\sin\left(\sqrt{\lambda_k}(1-x)\right)+
\frac{1}{\cos\left(\sqrt{\lambda_k}\right)}
 \sin\left(\sqrt{\lambda_k}x\right)\right) 
\\ \mbox{       for } x \in e(v_1,w_i), i=1..n\\
  \frac{\sqrt{m(m+1)}((n+1)\cos^2\left(\sqrt{\lambda_k}\right)-n)}
       {\cos\left(\sqrt{\lambda_k}\right)(m+n+1)}
\left(\sin\left(\sqrt{\lambda_k}(1-x)\right)+
\frac{1}{\cos\left(\sqrt{\lambda_k}\right)}
 \sin\left(\sqrt{\lambda_k}x\right)\right) 
\\ \mbox{       for } x \in e(v_2,w_i'), i=1..m.
\end{array} \right. 
\end{eqnarray*}

These eigenfunctions correspond to eigenvalues $\lambda_k$, where \\
$\cos^2\left(\sqrt{\lambda_k}\right) = \frac{mn}{(m+1)(n+1)}$.  Let
$\sqrt{\lambda_0}$ be the square root of the eigenvalue in
$(0,\frac{\pi}{2})$.  All other $\sqrt{\lambda}$ are of the form
$\sqrt{\lambda_0} +k\pi$ and $-\sqrt{\lambda_0} +k\pi$, where $k$ is
a positive integer.  When $k$ is even, the cosine is positive; for $k$
odd, cosine is negative.

When we look at the product of $\tilde{\Phi}_k(v_1)$ and
$\tilde{\Phi}_k(v_2)$, this expression simplifies greatly.
\begin{eqnarray*}
\tilde{\Phi}_k(v_1)\tilde{\Phi}_k(v_2) 
=-\sgn\left(\cos\left(\sqrt{\lambda_k}\right)\right) \frac{\sqrt{mn}}{
\sqrt{(m+1)(n+1)}(m+n+1) } .
\end{eqnarray*}

We then know that the sum of
$\tilde{\Phi}_k(v_1)\tilde{\Phi}_k(v_2)e^{-\lambda_k^2 t}$ is
$\frac{\sqrt{mn}}{\sqrt{(m+1)(n+1)}(m+n+1)}$ multiplied by the
following:
\begin{equation*}
\begin{split}
\sum_{k=0}^{\infty} & -\sgn\left(\cos\left(\sqrt{\lambda_k}\right)\right)  
    e^{-(\lambda_k)^2 t} 
\\ &= 
\sum_{k=0}^{\infty} 
  e^{-((2k+1)\pi-\sqrt{\lambda_0})^2 t} 
-  e^{-((2k+2)\pi-\sqrt{\lambda_0})^2 t}  
-  e^{-(2k\pi +\sqrt{\lambda_0})^2 t}  
+  e^{-((2k+1)\pi+\sqrt{\lambda_0})^2 t}
\\ &= 
\sum_{k=0}^{\infty} 
  e^{-(\sqrt{\lambda_0}-(2k+1)\pi)^2 t} 
-  e^{-(\sqrt{\lambda_0}-(2k+2)\pi)^2 t}  
-  e^{-(\sqrt{\lambda_0}+2k\pi)^2 t}  
+  e^{-(\sqrt{\lambda_0}+(2k+1)\pi)^2 t}
\\ &=
\sum_{k=-\infty}^{\infty} 
e^{-(\sqrt{\lambda_0}-(2k+1)\pi)^2 t}
-e^{-(\sqrt{\lambda_0}-2k\pi)^2 t}.
\end{split}
\end{equation*}
We can sum $\Phi(v_1)\Phi(v_2)e^{-\lambda^2 t}$ over all of the
eigenfunctions to find an explicit expression for the heat kernel at
$(v_1,v_2)$.
\begin{equation*}
\begin{split}
 h_t(v_1,v_2) 
=& \frac{1}{m+n+1} \left( \sum_{k=-\infty}^{\infty} e^{- k^2 \pi^2 t} \cos(k \pi) \right. \\
 & + \left. \sqrt{\frac{mn}{(m+1)(n+1)}} 
\left(
\sum_{k=-\infty}^{\infty} e^{-(\sqrt{\lambda_0}-(2k+1)\pi)^2 t} 
-\sum_{k=-\infty}^{\infty} e^{-(\sqrt{\lambda_0} -2k\pi)^2 t}
\right) \right) .
\end{split}
\end{equation*}
We can rewrite this using Jacobi's identity (see Dym \cite{Dym}). The
identity is:
\begin{eqnarray*}
\sum_{k=-\infty}^{\infty} e^{-\frac{(x-k)^2}{2s}} 
= \sqrt{2 \pi s}\sum_{k=-\infty}^{\infty}e^{-2 \pi^2 k^2 s}e^{2 \pi i k x}. 
\end{eqnarray*}

Because the left hand side is a real number, we can rewrite the $e^{2
\pi i k x}$ on the right hand side to obtain:
\begin{eqnarray*}
\sum_{k=-\infty}^{\infty} e^{-\frac{(x-k)^2}{2s}} = \sqrt{2 \pi
s}\sum_{k=-\infty}^{\infty}e^{-2 \pi^2 k^2 s} \cos(2\pi k x). 
\end{eqnarray*}

We can apply this identity to the three series.  
For the first, we use the reverse of the equality with $x=\frac{1}{2}$
and $s=\frac{t}{2}$ in the right hand side.  This gives us:
\begin{eqnarray*}
\sum_{k=-\infty}^{\infty} e^{- k^2 \pi^2 t} \cos(k \pi)
=\frac{1}{\sqrt{\pi t}}
\sum_{k=-\infty}^{\infty} e^{-\frac{(1-2k)^2}{4t}}. 
\end{eqnarray*}

In the second, let $x=\frac{\sqrt{\lambda_0}+\pi}{2 \pi}$ and
$s=\frac{1}{8\pi^2 t}$.  This gives us \\ $ \frac{(x-k)^2}{2s} =
\frac{(\sqrt{\lambda_0}+\pi -2\pi k)^2}{(2\pi)^2 2 \frac{1}{8 \pi^2 t}}
= (\sqrt{\lambda_0} + (1-2k)\pi)^2 t$.  When we put that into the
identity, we have:
\begin{eqnarray*}
\sum_{k=-\infty}^{\infty} e^{-(\sqrt{\lambda_0} + (1-2k)\pi)^2 t} 
= \frac{1}{2\sqrt{\pi t}}\sum_{k=-\infty}^{\infty}e^{-\frac{k^2}{4t}} 
   \cos\left(k\left(\sqrt{\lambda_0}+\pi\right)\right). 
\end{eqnarray*}
 Note that $\sum_{k=-\infty}^{\infty}
e^{-(\sqrt{\lambda_0}-(2k+1)\pi)^2 t} =\sum_{k=-\infty}^{\infty}
e^{-(\sqrt{\lambda_0} + (1-2k)\pi)^2 t}$ by reindexing, and so this is
a way of rewriting our original sum.

For the third, we use $x= \frac{\sqrt{\lambda_0}}{2 \pi}$ and
$s=\frac{1}{8 \pi^2 t}$ which, by a similar computation, gives us:
\begin{eqnarray*}
\sum_{k=-\infty}^{\infty} e^{-(\sqrt{\lambda_0} -2k\pi)^2 t}
=\frac{1}{2\sqrt{\pi t}} \sum_{k=-\infty}^{\infty}e^{-\frac{k^2}{4t}} 
\cos\left(k \sqrt{\lambda_0}\right). 
\end{eqnarray*}

We can combine these three to rewrite $h_t(v_1,v_2)$ as follows:
\begin{equation*}
\begin{split}
h_t(v_1,v_2)=& \frac{1}{m+n+1} \frac{1}{\sqrt{\pi t}}\left(
\sum_{k=-\infty}^{\infty} e^{-\frac{(1-2k)^2}{4t}} \right. \\
 +& \left. \sqrt{\frac{mn }{(m+1)(n+1)}}
\frac{1}{2} 
\sum_{k=-\infty}^{\infty}e^{-\frac{k^2}{4t}} 
\left(
 \cos\left(k\sqrt{\lambda_0}+k \pi \right) 
 - \cos\left(k \sqrt{\lambda_0}\right)
\right) \right).
\end{split}
\end{equation*}
We can use the angle sum formula for cosines to see that
\begin{eqnarray*}
\cos\left(k\sqrt{\lambda_0}+k \pi \right) &=&
 \sin\left(k\sqrt{\lambda_0}\right)\sin\left(k\pi\right) 
+ \cos\left(k\sqrt{\lambda_0}\right)\cos\left(k\pi\right)\\
&=&\cos\left(k\sqrt{\lambda_0}\right)\cos\left(k\pi\right).
\end{eqnarray*}
As $\cos\left(k\pi\right) = 1$ for $k$ even and $-1$ for $k$ odd, the
term $\cos\left(k \sqrt{\lambda_0}+k \pi \right) -\cos\left(k
\sqrt{\lambda_0}\right)$ reduces to $0$ for even $k$ and $-2\cos\left(k
\sqrt{\lambda_0}\right)$ for odd $k$.  This allows us to write:
\begin{eqnarray*}
h_t(v_1,v_2) &=&  \frac{1}{m+n+1} \frac{1}{\sqrt{\pi t}} \left(
\sum_{k=-\infty}^{\infty} e^{-\frac{(2k+1)^2}{4t}} \right. \\
&&\left. + \sqrt{\frac{mn}{(m+1)(n+1)}} 
\frac{1}{2} 
\sum_{k=-\infty}^{\infty}e^{-\frac{(2k+1)^2}{4t}} 
(-2)\cos\left((2k+1) \sqrt{\lambda_0}\right) \right) 
 \\&=&  \frac{1}{m+n+1} \frac{1}{\sqrt{\pi t}} 
\sum_{k=-\infty}^{\infty} e^{-\frac{(2k+1)^2}{4t}}
\left(1-\frac{\sqrt{mn}\cos\left((2k+1)\sqrt{\lambda_0}\right)}
             {\sqrt{(m+1)(n+1)}}\right).
\end{eqnarray*}

When we consider the limit as $t$ approaches $0$, the $k=0$ and $k=-1$
terms will dominate.  By symmetry, these terms are equal.  This gives
us the approximation:
\begin{eqnarray*}
h_t(v_1,v_2) 
& \approx &
 \frac{1}{m+n+1} \frac{2}{\sqrt{\pi t}} 
e^{-\frac{1}{4t}}
\left(1-\frac{\sqrt{mn}\cos\left(\sqrt{\lambda_0}\right)}
             {\sqrt{(m+1)(n+1)}}\right) 
\\&=& \frac{1}{m+n+1} \frac{2}{\sqrt{\pi t}} 
 e^{-\frac{1}{4t}}
\left(1-\frac{\sqrt{mn}  \sqrt{ \frac{mn}{(m+1)(n+1)}}  }
             {\sqrt{(m+1)(n+1)}}\right)
\\&=&  \frac{1}{(m+1)(n+1)} \frac{2}{\sqrt{\pi t}} 
e^{-\frac{1}{4t}}.
\end{eqnarray*}

Note that when $m=n=0$ we have a line; this is the correct asymptotic
there.  Similarly, when $m=n=1$, we have a line of length 3.  This
works out there too.

We can also find the on-diagonal asymptotic at $v_1$, one of the
central vertices.  Here we have
\begin{eqnarray*}
\Phi_k^2(v_1)=\frac{2}{m+n+1}
\end{eqnarray*}
and
\begin{eqnarray*}
\tilde{\Phi}_k^2(v_1)
&=& \left(\frac{\sqrt{m(m+1)}}{\sqrt{m+n+1} } \right)
   \sin^2\left(\sqrt{\lambda_k}\right) \\
&=& \frac{m}{(m+n+1)(n+1)}.
\end{eqnarray*}
Using the same style of manipulations as before we find that $h_t(v_1,v_1)$ is
\begin{eqnarray*}
h_t(v_1,v_1) = \frac{1}{m+n+1} \sum_{k=-\infty}^{\infty} e^{-k^2 \pi^2 t}
+ \frac{m}{(m+n+1)(n+1)} \sum_{k=-\infty}^{\infty} e^{-(\lambda_0 -k \pi)^2 t}.
\end{eqnarray*}
We can use a Jacobi transform here with $x=0$, $t=2s$ in the first sum
and $x=\frac{\sqrt{\lambda_0}}{\pi}$ and $t=\frac{1}{2 \pi^2 s}$ in
the second to obtain:
\begin{eqnarray*}
h_t(v_1,v_1) = \frac{1}{(m+n+1)\sqrt{\pi t}} \sum_{k=-\infty}^{\infty}
\left(1 + \frac{m}{n+1} \cos(2k\sqrt{\lambda_0})\right)
e^{-\frac{k^2}{t}}.
\end{eqnarray*}
As t approaches zero, the $k\ne 0$ terms also approach zero.  The
dominant $k=0$ term is
\begin{eqnarray*}
\frac{1}{(m+n+1)\sqrt{\pi t}} \left(1 + \frac{m}{n+1} \right) 
= \frac{1}{(n+1)\sqrt{\pi t}}.
\end{eqnarray*}
This is the same asymptotic as we have on a single star whose center
connects to $n+1$ edges.  This is not surprising, since small values
of $t$ correspond to the local geometry of a space.
  \end{example}
  \begin{example}
 The heat kernel on an interval is messier than that on a line.  Dym
\cite{Dym} calculates it for 
$\frac{\partial}{\partial t} u -\frac{1}{2}\Delta u=0$ on the interval
$[0,1]$ to be:
\begin{eqnarray*}
p_t(x,y) 
= 1 + 2 \sum_{n=1}^{\infty} e^{-n^2 \pi^2 t/2} \cos(n \pi x) \cos(n \pi y) \\
= \frac{1}{\sqrt{2 \pi t}} 
 \sum_{n=-\infty}^{\infty} e^{-\frac{(x-y-2n)^2}{2t}} 
                         + e^{-\frac{(x+y-2n)^2}{2t}}.
\end{eqnarray*}
We would like to have the heat kernel for $\frac{\partial}{\partial t}
\tilde{u} -\Delta \tilde{u}=0$ on $[0,L]$.  We can find it by noting
that if $u(t,x)$ is a solution to $\frac{\partial}{\partial t} u
-\frac{1}{2}\Delta u=0$ on the interval $[0,1]$, then $\tilde{u}(t,x) = c
u(a t,b x)$ has derivatives
\begin{eqnarray*}
\left(\frac{\partial}{\partial t} \tilde{u}\right)(t,x) 
  = \left(c a \frac{\partial}{\partial t} u\right)(a t,b x) 
\end{eqnarray*}
and 
\begin{eqnarray*}
(\Delta \tilde{u})(t,x) = (c b^2 \Delta u)(a t, b x) = \left( 2 c b^2
\frac{\partial}{\partial t} u \right)(a t,b x).
\end{eqnarray*}
Then $(\frac{\partial}{\partial t} \tilde{u})(t,x) - (\Delta
\tilde{u})(t,x) = 0$ when $a=2b^2$.  We change the interval length by
setting $b=\frac{1}{L}$.  This gives us $\tilde{p}_t(x,y) = c
p_{\frac{2t}{L^2}}(\frac{x}{L},\frac{y}{L})$. The constant $c$ is used
to normalize so that $\int_0^L \tilde{p}_t(x,y) dx=1$ holds for each
$t>0$.  Since this integral is $\int_0^L c
p_{\frac{2t}{L^2}}(\frac{x}{L},\frac{y}{L}) dx =
\int_0^1 c p_{\frac{2t}{L^2}}(z,\frac{y}{L}) Ldz = cL$, 
we have $c=\frac{1}{L}$.

We can write this as
\begin{eqnarray*}
\tilde{p}_t(x,y) 
= \frac{1}{\sqrt{4 \pi t}} 
    \sum_{n=-\infty}^{\infty} e^{-\frac{(x-y-2nL)^2}{4t}} 
                            + e^{-\frac{(x+y-2nL)^2}{4t}}.
\end{eqnarray*}
This description allows one to look at terms with $n$ near 0 to find the
small time asymptotic.

Note that for $L=3$, $x=1$ and $y=2$ we have:
\begin{eqnarray*}
\tilde{p}_t(1,2) 
= \frac{1}{\sqrt{4 \pi t}} 
    \sum_{n=-\infty}^{\infty} e^{-\frac{(-1-6n)^2}{4t}} 
                            + e^{-\frac{(3-6n)^2}{4t}}.
\end{eqnarray*}
When $t$ is small, this behaves like the first $n=0$ term:
\begin{eqnarray*}
\tilde{p}_t(1,2) 
\approx \frac{1}{2\sqrt{\pi t}} e^{-\frac{1}{4t}} .
\end{eqnarray*}
This is consistent with the asymptotic for the two star case when
$m=n=1$.  Note also that when $L=1$, $x=0$, $y=1$ that this gives
$\frac{2}{\sqrt{\pi t}} e^{- \frac{1}{4t}}$ which is consistent with
the star with one edge.
   \end{example}
 \begin{example}
\begin{figure}[h]\label{picZ3}
\centering
\includegraphics[angle=0,width=1in]{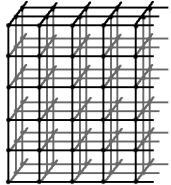}
\caption{A subset of the three dimensional grid.}
\end{figure}
Complexes can have an underlying group structure.  For example, $Z^3$,
the group consisting of triplets of integers, can be used to create a
3 dimensional complex by connecting a cube, $[0,1]^3$, to itself where
each pair of triples which differ by one by a line segment.  This grid
is the 1-skeleton, and the points in the group form a 0-skeleton.  We
can use the 1-skeleton to create a space that looks like a bunch of
empty boxes by filling in the faces formed by loops of four edges.
This is the 2-skeleton.  If we then fill in the boxes in the
2-skeleton, we'll have a 3-skeleton, which is $R^3$.  Locally, we've
shown that if $h_t$ is the heat kernel on the $k$-skeleton,
\begin{eqnarray*}
\frac{1}{{C} t^{k/2}}
\le h_t(x,x) \le \frac{C}{(\min(t,R_0^2))^{k/2}}
\end{eqnarray*}
for all $x \in X$ and all $t$ such that $0<t<R_0^2$.

We prove in chapter 5 that the heat kernel on each of these k-skeletons
globally behaves like the heat kernel on $R^3$.  
\end{example}

\begin{example}
There are also complexes with underlying group structure whose
geometry is not globally Euclidean.  Let $G$ be the free group on two
elements; this is a group of words formed by letters $a$ and $b$ and
their inverses $a^{-1}$ and $b^{-1}$ where the only cancellations are
$aa^{-1} =a^{-1}a =1$ and $bb^{-1}=b^{-1}b=1$, and $a$ and $b$ don't
commute.  Let $Y$ be the complex formed by three squares joined into
an L shape (see Fig. \ref{picN}).  We have a larger structure $X$
which has copies of $Y$ connected to each other via the group $G$.
That is, each copy of $Y$ will be connected to four other copies of
$Y$; the top of the L connects to the bottom edge of the lower left
square of the L, and the right of the L connects to the left edge of
the bottom square.

\begin{figure}[h]
\centering
\includegraphics[angle=0,width=1in]{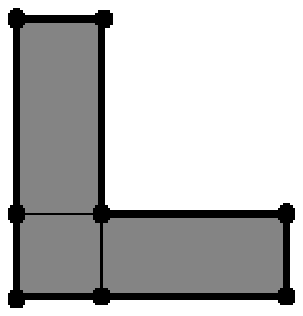}
\hspace{1in}
\includegraphics[angle=0,width=2in]{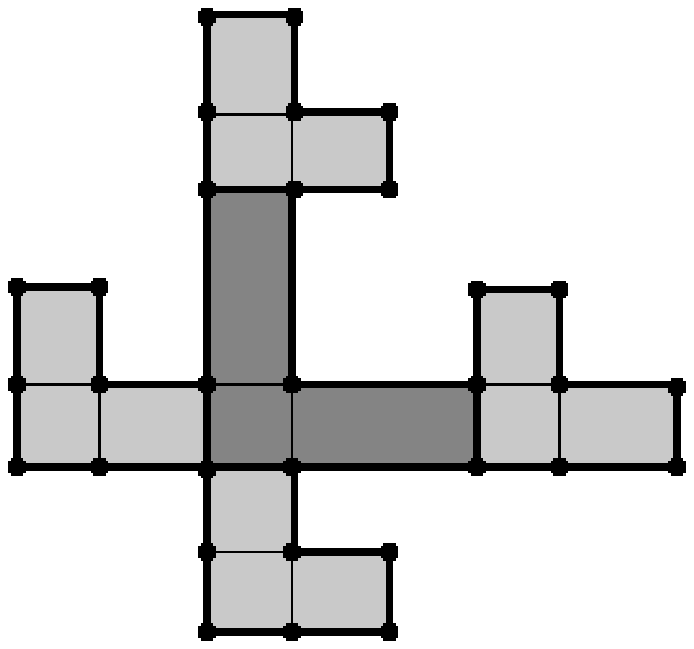}
\caption{Y (left); Y (darker shading) and its four surrounding copies
(lighter shading) (right) Each of the edges in these pictures should
be interpreted as having length 1.}\label{picN}
\end{figure}

  Note that we can not isometrically embed $X$ into $R^2$. The
stretching of the edges in Fig. \ref{picN} is to allow you to see
distinct edges and vertices.  This structure globally it acts like a
hyperbolic space, but locally it is Euclidean.

 For a small ball with $R<.5$, we have a two dimensional circle
(possibly missing a wedge) whose volume is $\approx \pi R^2$.  For $t
<.5$, corollary \ref{ComplexIsJustLikeRn} tells us:
\begin{eqnarray*}
h_t(x,x) \approx \frac{C}{(4\pi t)}.
\end{eqnarray*}

For a given copy of $Y$, there are four neighbors, each of which has
three additional neighbors.  For a ball of radius $R>1$, there will be
approximately $1 + 4 + 4(3) + \cdots + 4(3^R) \approx 2 (3^{R+1})$ copies of
$Y$.  This tells us that for large $R$, we have exponential volume
growth.  In particular, this group is nonamenable.  We define this in
chapter 5 and show that the large time behavior of the heat kernel is
\begin{eqnarray*}
\sup_{x \in X} h_t(x,x) \approx C_0 e^{-t/C_1}.
\end{eqnarray*}
\end{example}

\chapter{Setup for Groups}
A finite product of elements from a set $S$ is called a word.  If a
word is written $s_1s_2...s_k$, we say it has length $k$.
 
A finitely generated group is a group with a generating set, $S$,
where every element in the group can be written as a finite word using
elements of $S$.  Although for a given $g \in G$ it is computationally
difficult to determine which word is the smallest one representing
$g$, such a word (or words) exists.  If this word has length $k$, then
we write $|g|=k$.

We define the volume of a subset of $G$ to be the number of elements
of $G$ contained in that subset.  We write $|B_r|$ to denote the volume
of a ball of radius $r$, $B_r := \{g \in G : |g| \le r\}$.  For groups,
volume is translation invariant, and so we do not lose any generality
by having it centered at the identity.

For a function $f$ which maps elements of a group to the reals, we
define the Dirichlet form on $\varl^2(G)$ to be $E(f,f) = \frac{1}{|S|}\sum_{g
\in G}\sum_{s \in S} |f(g) -f(gs)|^2$.

Although we'd need to specify directions if we were to define a
gradient, we can define an object which behaves like the length of the
gradient of $f$ on $G$.  We write this as $|\grad f(x)|
=\sqrt{\frac{1}{|S|}\sum_{s \in S} |f(x) -f(xs)|^2}$.  Notationally,
this means that $E(f,f) = \sum_{g \in G} |\grad f(g)|^2$.

Discrete $L^p$ norms restricted to a subset $A \subset G$ are written
 as \\
$||f||_{p,A} = \left(\sum_{x\in A} |f(x)|^p \right)^{1/p}$.  When $A=G$, we
 will write $||f||_{p,G}$.
 
 One can show a Poincar\'{e} type inequality on a volume doubling
finitely generated group. The arguments used in this can be found in
\cite{CLsc1993}. 
 
\begin{lemma}\label{groupVDpoincare}
Let $G$ be a finitely generated group with generating set $S$.  For
any $f : G \rightarrow R$, the following inequality holds on balls
$B_r$:
\begin{eqnarray*}
\norm{f-f_{B_r}}_{1,B_r} 
  \le \frac{|B_{2r}|}{|B_r|} 2r \sqrt{|S|} \norm{\grad f}_{1,B_{3r}}.
\end{eqnarray*}
If the group is volume doubling, this is a weak Poincar\'{e}
inequality on balls for $p=1$:
\begin{eqnarray*}
\norm{f-f_{B_r}}_{1,B_r} 
  \le 2r C_{Doubling}\sqrt{|S|} \norm{\grad f}_{1,B_{3r}}.
\end{eqnarray*}
\end{lemma}
\begin{proof}
Let $G$ be a finitely generated group with a symmetric set of
generators, $S$.  Let $B_r$ be a ball of radius $r$; for brevity, we
will not explicitly write the center.  We can write the norm of $f$
minus its average as follows.
\begin{eqnarray*}
\norm{f-f_{B_r}}_{1,B_r}
&= &\sum_{x \in B_r} |f(x) - \frac{1}{|B_r|} \sum_{y \in B_r} f(y)| \\
&\le & \frac{1}{|B_r|}\sum_{x \in B_r}  \sum_{y \in B_r}| f(x)-f(y)|. 
\end{eqnarray*}
For each $y \in B_r$, there exists a $g\in G$ with $|g| \le 2r$ such
that $y =x g$. We make this substitution and sum over all $g\in G$
with $|g| \le 2r$.
\begin{eqnarray*}
 \frac{1}{|B_r|}\sum_{x \in B_r}  \sum_{y \in B_r}| f(x)-f(y)|
&\le &\frac{1}{|B_r|}\sum_{x \in B_r} \sum_{g: |g|\le 2r} | f(x)-f(x g)| \\
&=& \frac{1}{|B_r|}\sum_{g: |g|\le 2r}\sum_{x \in B_r} | f(x)-f(x g)|.
\end{eqnarray*}
We will begin with the innermost quantity, and then simplify the sums.
We can write $g = s_1..s_k$ as a reduced word with $k \le 2r$.  We
rewrite the difference of $f$ at $x$ and $xg$ by splitting the path
between them into pieces.
\begin{eqnarray*}
|f(x)-f(x g)|
 \le \sum_{i=1}^{|g|} |f(xs_1...s_{i-1})- f(xs_1...s_i)|.
\end{eqnarray*}
We fix $g$ and sum over all $x \in B_r$.
\begin{eqnarray*}
\sum_{x \in B_r} |f(x)-f(xg)|
& \le &\sum_{x \in B_r}  
     \sum_{i=1}^{|g|} |f(xs_1...s_{i-1})- f(xs_1...s_i)| \\
 &=& \sum_{i=1}^{|g|} \sum_{x \in B_r} 
     |f(xs_1...s_{i-1})- f(xs_1...s_i)|. 
\end{eqnarray*}
We can change variables by letting $z= xs_1...s_{i-1}$.  Note that $|x
s_1..s_{i-1}| < 3r$ as $i\le 2r$.  Then $z \in B_{3r}$.
\begin{eqnarray*}
\sum_{i=1}^{|g|} \sum_{x \in B_r} 
     |f(xs_1...s_{i-1})- f(xs_1...s_i)| 
 \le \sum_{i=1}^{|g|} \sum_{z \in B_{3r}} 
      |f(z)- f(zs_i)| .
\end{eqnarray*}
Since $s_i \in S$, we can sum over all $s \in S$ instead of the $s_i$
in $g$.  To do this, we must account for the multiplicity of the
$s_i$.  We could have at most $|g|$ copies of any generator; $|g| \le
2r$, and so we will multiply by $2r$.
\begin{eqnarray*}
\sum_{i=1}^{|g|} \sum_{z \in B_{3r}} 
      |f(z)- f(zs_i)| 
 \le 2r \sum_{s \in S} \sum_{z \in B_{3r}} 
      |f(z)- f(zs)|.
\end{eqnarray*}
Jensen's inequality allows us to rewrite this in terms of the gradient.
\begin{eqnarray*}
2r \sum_{s \in S} \sum_{z \in B_{3r}} 
      |f(z)- f(zs)|
& \le& 2r \sum_{z \in B_{3r}} 
           \sqrt{|S| \frac{1}{|S|} \sum_{s \in S} |f(z)- f(zs)|^2} \\
 & =&  2r \sum_{z \in B_{3r}} \sqrt{|S|} |\grad f(z)|.
\end{eqnarray*}
We'll use this calculation to get the desired inequality.  We now have
\begin{eqnarray*}
\sum_{x \in B_r} |f(x)-f(xg)|
\le 2r \sum_{z \in B_{3r}} \sqrt{|S|} |\grad f(z)|.
\end{eqnarray*}
Dividing by $|B_r|$ and summing over the $g$ gives us
\begin{eqnarray*}
\frac{1}{|B_r|}\sum_{g: |g|\le 2r}\sum_{x \in B_r} | f(x)-f(x g)| 
&\le& \frac{1}{|B_r|}\sum_{g: |g|\le 2r} 
    \sum_{z \in B_{3r}} 2r \sqrt{|S|} |\grad f(z)| \\
&=& \frac{|B_{2r}|}{|B_r|} 2r \sqrt{|S|}
    \sum_{z \in B_{3r}} |\grad f(z)|.
\end{eqnarray*}
This reduces to
\begin{eqnarray*}
\norm{f-f_{B_r}}_{1,B_r} 
  \le \frac{|B_{2r}|}{|B_r|} 2r \sqrt{|S|} \norm{\grad f}_{1,B_{3r}}.
\end{eqnarray*}
Note that in general, $\frac{|B_{2r}|}{|B_r|}$ will depend on the
radius, $r$.  If the group is volume doubling, this gives us a weak
Poincar\'{e} inequality.
\begin{eqnarray*}
\norm{f-f_{B_r}}_{1,B_r} 
  \le C_{Doubling} 2r \sqrt{|S|} \norm{\grad f}_{1,B_{3r}}.
\end{eqnarray*}
\end{proof}
\section{Comparing distances in X and G}
Let $X$ be a complex, and $G$ be a finitely generated group of
isomorphisms on the complex such that $X/G =Y$ is an admissible
complex consisting of a finite number of polytopes.

One example of this type of complex is a Cayley graph; this is the graph where
each vertex corresponds to a group element, and two vertices are
connected by an edge if they differ by an element of the generating
set.  In this case, $Y$ is the unit interval.

We would like to be able to compare functions defined on the group, $G$,
with functions defined on the complex, $X$.  To do this, we will look at
ways to transfer a function defined on $G$ to a function defined on $X$
that roughly preserves the norm of both the function and its energy
form.  We seek to do the reverse as well.  We will use a technique that originated with Kanai \cite{Kanai} and additionally was used by Coulhon and Saloff-Coste \cite{CoulhonLsc}.

We also want a way of changing from real valued functions which take
values in $X$ to ones that take values in $G$.  We will do this by
taking a copy of $Y$ and splitting it into many smaller pieces.  Given
$\delta \le \diam(Y)$, we can find a finite covering of $Y$ by balls
of radius $\delta$ such that balls of radius $\delta /2$ are disjoint
in $Y$.  As $Y$ is a finite polytopal complex, volume doubling on $Y$
implies that at most a finite number of balls of radius $\delta$ will
overlap.  $X$ can be written by taking a copy of $Y$ for each element
of $G$, and so this cover can be expanded to a cover of $X$.  Note
that once we have a copy of the cover of $Y$ for each element of $G$,
balls of radius $\delta /2$ in this larger cover may overlap. For $X$,
the overlap is also finite; call the number of overlapping balls
$\cover$.

Call the centers of the balls covering $Y$ $\{ \gamma_i\}_{i=1}^N$ and
the balls covering $X$ $\{ g\gamma_i\}_{i=1..N; g \in G}$.  Note that
each $x\in X$ is within $\delta$ of at least one of the $g\gamma_i$.
As we are frequently switching between $X$ and $G$, we will use $B_X$
for balls in $X$ and $B_G$ for balls in $G$.

\begin{example}
\begin{figure}[h]
\centering
       \includegraphics[angle=0,width=1.5in]{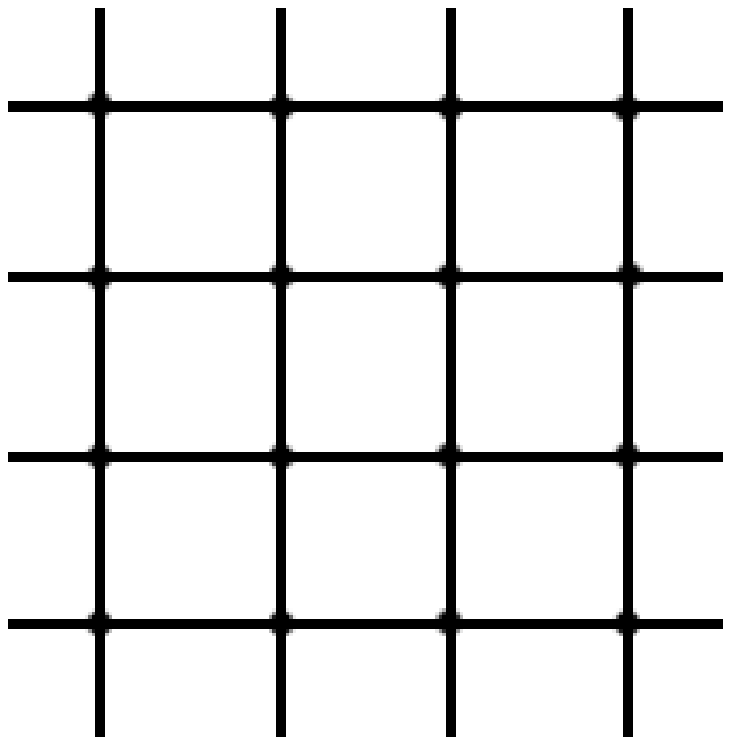}
\hspace{.5in}
       \includegraphics[angle=0,width=1.5in]{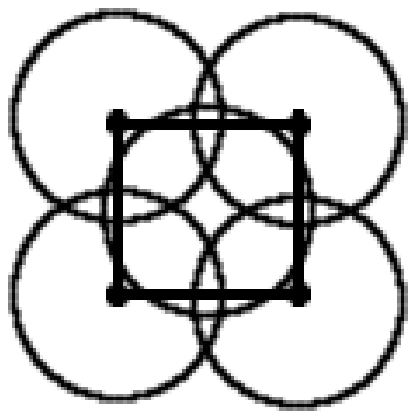}
\hspace{.5in}
       \includegraphics[angle=0,width=1.5in]{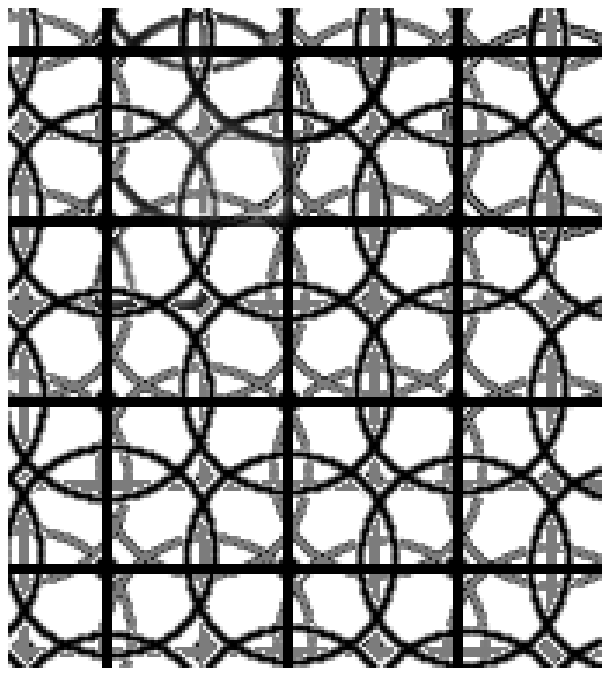}
\caption{X split into copies of Y(left);a copy of Y covered by 5 balls of radius .6 (center); a cover for Y shifted by five copies of G cover everything -- the black lines represent four copies that overlap exactly; the fifth copy is gray(right).}
\end{figure}
Let $X=R^2$, $G =Z^2$ and $Y =[0,1]^2$.  For $\delta =.6$, balls of
radius $\delta/2 =.3$ centered at the corners of $Y$ are disjoint, but
not all points in the plane are covered. We can introduce another copy
of $G$ that's shifted by $(.5,.5)$.  All points in $Y$ are covered by
some ball of radius delta, but the balls of radius $.3$ will not
overlap.  We can check this by comparing the distance along the
diagonal from $(0,0)$ to $(1,1)$ with the length covered by the radii
along the same diagonal.  We have $d((0,0),(1,1))= \sqrt{2} \approx
1.4$.  For the balls, there's the radius of the one centered at
$(0,0)$, the one centered at $(1,1)$ and the diameter of the one
centered at $(.5,.5)$.  This sums to $4\delta =1.2$.

The number of overlapping balls in $X$ is $\cover=10$. This happens at
$(0,.5)$ where there are four balls centered at $(0,0)$, one centered
at $(.5,.5)$, one centered at $(-.5,.5)$, and four centered at
$(0,1)$.  In this example, $\gamma_1=(0,0)$, $\gamma_2=(0,1)$,
$\gamma_3=(1,0)$, $\gamma_4=(1,1)$, and $\gamma_5=(.5,.5)$.
\end{example}

We define
\begin{eqnarray*}
\group{f(g,i)} = \frac{1}{\mu\left(B_X(g\gamma_i,\delta)\right)}
                 \int_{B_X(g\gamma_i,\delta)} f(x) dx 
               = \dashint_{B_X(g\gamma_i,\delta)} f(x) dx .
\end{eqnarray*}

We can view $\group{f}$ as a collection of $N$ functions defined on
$G$.  For any fixed $i$, we can treat $\group{f(\cdot,i)}$ as a
function on the group, and so the norm of $\group{f}$ can be found by
summing these over $i$.

It's important to note that the sets $\{g : g\in B_X(r) \cap G\}$ and
$B_G(r)$ are potentially different.  We can describe this by comparing
distances with some explicit constants.
\begin{lemma}\label{CompareDistance}
We can compare distances in $G$ and $X$ in the following manner.
There exist constants $\CompareXG$ and $\CXG$ so that for any $g,h\in
G$ we have:
\begin{eqnarray*}
\frac{1}{\CXG} d_X(g,h) \le d_G(g,h) \le \CompareXG d_X(g,h).
\end{eqnarray*}
This tells us that balls centered at points in $G$ compare as:
\begin{eqnarray*}
G \cap B_X(\frac{r}{\CompareXG}) \subset B_G(r) \subset G \cap B_X(\CXG r).
\end{eqnarray*}
Here, $\CXG = \max_{v_1,v_2 \in Y^{(0)}} d_{Y^{(1)}}(v_1,v_2)$ and \\
$\CompareXG = \frac{1}{\min_{v_1,v_2 \in Y^{(0)}} d_{X^{(1)}}(v_1,v_2)}
\prod_{i=1}^n
\max\left(\sqrt{\frac{2}{1-\cos(\alpha)}},\frac{\max_{y_1,y_2 \in Y^{(i-1)}}
d_{X^{(i-1)}}(y_1,y_2)}{\min_{v_1,v_2 \in Y^{(0)}}
d_{X^{(1)}}(v_1,v_2)} \right)$ where $\alpha$ is the smallest interior angle
in $X$.
\end{lemma}
\begin{proof}
We'll start with the easy direction.  Any path in $X^{(1)}$ is also a
path in $X$, and so $d_X(g,h) \le d_{X^{(1)}}(g,h)$.  To compare this
with distances in $G$, which count numbers of points in paths, we use
lengths of edges between them.
\begin{eqnarray*}
d_{X^{(1)}}(g,h) \le 
\max_{v_1,v_2 \in Y^{(0)}} d_{Y^{(1)}}(v_1,v_2) d_G(g,h).
\end{eqnarray*}

Set $\CXG = \max_{v_1,v_2 \in Y^{(0)}} d_{Y^{(1)}}(v_1,v_2)$.
Since whenever $ \CXG d_{G}(g,h)\le r$ we also have $d_X(g,h)\le r$,
we know that any point in $B_G(\frac{r}{\CXG})$ is also in $B_X(r)$.
This tells us that $B_G(r) \subset G \cap B_X(\CXG r)$.

We can use the fact that $X$ can be subdivided into copies of $Y$ to
relate distances in the other direction as well.  We will compare
distances in the different skeletons of $X$.  In order to simplify
notation, $d_k(x,y)$ will refer to the distance between $x$ and $y$ when
we restrict to paths in $X^{(k)}$.

Let $g,h \in G$ be given.  Then there is a shortest path in $X$
between them.  If there are multiple such paths, pick one. Label it
$\{x_0=g,x_1,x_2,..,x_k=h\}$ where $x_i \in$ $X^{(n-1)}$ for
$i=1..k-1$, and $x_i$, $x_{i+1}$ are both in the boundary of the same
maximal polytope, although they are on different faces.  In
particular, both are in $X^{(n-1)}$.  Then we know that the length of
this path is$\sum_{i=1}^k d_{n}(x_{i-1},x_i)$.

We will compare $d_{n}(x_{i-1},x_i)$ with $d_{n-1}(x_{i-1},x_i)$, and
use this to relate the distances between skeletons who differ by 1
dimension.  This will allow us to work our way from $n$ dimensions
down to $X^{(0)}$.  Then we can compare $X^{(0)}$ with $G$.

Look at $x_i$ and $x_{i+1}$ which are in $(n-1)$ dimensional faces $F_i$ and
$F_{i+1}$.  Either $F_i \cap F_{i+1}$ is nonempty and so they share a
lower dimensional point, or else it is empty and they do not.  

If they do not share a point, then:
\begin{eqnarray*}
\min_{v_1,v_2 \in Y^{(0)}} d_{1}(v_1,v_2)
\le d_{n}(x_i,x_{i+1}).
\end{eqnarray*}
We get this bound because the diameter of any subpolyhedra of $Y$ must
be bounded below by the length of the smallest edge of that polyhedra.

We can also bound the $(n-1)$ distance:
\begin{eqnarray*}
 d_{n-1}(x_i,x_{i+1}) 
   \le \max_{y_1,y_2 \in Y^{(n-1)}} d_{n-1}(y_1,y_2).
\end{eqnarray*}

Putting this together, we get:
\begin{eqnarray*}
d_{n-1}(x_i,x_{i+1}) \le \frac{\max_{y_1,y_2 \in Y^{(n-1)}}
d_{n-1}(y_1,y_2)}{\min_{v_1,v_2 \in Y^{(0)}}
d_{1}(v_1,v_2)} d_{n}(x_i,x_{i+1}).
\end{eqnarray*}

Otherwise, if $F_i$ and $F_{i+1}$ intersect in a lower dimensional
face, we will call $v$ the point on the intersection which minimizes
$d_{n-1}(v,x_i)+d_{n-1}(v,x_{i+1})$.  These three points form a
triangle with angle $x_i v x_{i+1} = \theta \ge \alpha$, where
$\alpha$ is the smallest interior angle in $Y$ as well as in $X$.
Note that this angle is bounded because $Y$ is made up of a finite
number of polytopes.  We would like to determine a relationship
between $d_{n}(x_i,x_{i+1})$ and $\inf_{v \in F_i \cap F_{i+1}}
d_{n-1}(v,x_i)+d_{n-1}(v,x_{i+1})$.

To find this, we will use a simple derivation. For positive numbers $a$
and $b$ we have
\begin{eqnarray*}
(a-b)^2 &\ge& - \cos(\alpha)(a-b)^2 \\
a^2 + b^2 -2ab\cos(\alpha) &\ge& 2ab -a^2\cos(\alpha) -b^2\cos(\alpha) \\
2a^2 + 2b^2 -4ab\cos(\alpha) &\ge& a^2 + b^2 + 2ab -2ab\cos(\alpha) 
-a^2\cos(\alpha) -b^2\cos(\alpha) \\
a^2 + b^2 -2ab\cos(\alpha) &\ge& (a+b)^2\left(\frac{1-\cos(\alpha)}{2}\right).
\end{eqnarray*}

This is helpful because when we apply the law of cosines to the
triangle we have:
\begin{eqnarray*}
d_{n}^2(x_i,x_{i+1}) = d_{n-1}^2(x_i,v) +d_{n-1}^2(v,x_{i+1})
-2 d_{n-1}(x_i,v)d_{n-1}(v,x_{i+1})\cos(\theta).
\end{eqnarray*}

We can form an inequality by replacing $\cos(\theta)$ with the larger
$\cos(\alpha)$:
\begin{eqnarray*}
d_{n}^2(x_i,x_{i+1}) \ge d_{n-1}^2(x_i,v)
+d_{n-1}^2(v,x_{i+1}) -2
d_{n-1}(x_i,v)d_{X^{(n-1)}}(v,x_{i+1})\cos(\alpha).
\end{eqnarray*}

Then we can apply our fact with $a=d_{n-1}(x_i,v)$ and
$b=d_{n-1}(v,x_{i+1})$.
\begin{eqnarray*}
d_{n}^2(x_i,x_{i+1}) \ge (d_{n-1}(x_i,v)+
d_{n-1}(v,x_{i+1}))^2\left(\frac{1-\cos(\alpha)}{2}\right).
\end{eqnarray*}

This leads us to the conclusion that 
\begin{eqnarray*}
d_{n}(x_i,x_{i+1}) &\ge& \left(d_{n-1}(x_i,v)+
d_{n-1}(v,x_{i+1})\right)\sqrt{\frac{1-\cos(\alpha)}{2}} \\
&\ge& d_{n-1}(x_i,x_{i+1})\sqrt{\frac{1-\cos(\alpha)}{2}}.
\end{eqnarray*}

When we combine the cases where faces intersect with the case where
they do not, we get the following inequality:
\begin{eqnarray*}
d_{n-1}(x_i,x_{i+1}) 
\le \max\left(\sqrt{\frac{2}{1-\cos(\alpha)}},\frac{\max_{y_1,y_2 \in Y^{(n-1)}}
d_{n-1}(y_1,y_2)}{\min_{v_1,v_2 \in Y^{(0)}}
d_{1}(v_1,v_2)} \right) d_{n}(x_i,x_{i+1}).
\end{eqnarray*}

We can sum and use the fact that we had a distance minimizing path in
$X^{(n)}$ to get
\begin{eqnarray*}
d_{n-1}(g,h) 
\le \max\left(\sqrt{\frac{2}{1-\cos(\alpha)}},\frac{\max_{y_1,y_2 \in Y^{(n-1)}}
d_{n-1}(y_1,y_2)}{\min_{v_1,v_2 \in Y^{(0)}}
d_{1}(v_1,v_2)} \right) d_{n}(g,h).
\end{eqnarray*}

We can repeat this argument for the lower dimensions (down to
dimension 1) to get:
\begin{eqnarray*}
d_{1}(g,h) 
\le \prod_{i=1}^n 
\max\left(\sqrt{\frac{2}{1-\cos(\alpha)}},\frac{\max_{y_1,y_2 \in Y^{(i-1)}}
d_{i-1}(y_1,y_2)}{\min_{v_1,v_2 \in Y^{(0)}}
d_{1}(v_1,v_2)} \right) d_{X^{(n)}}(g,h).
\end{eqnarray*}

To compare with the distance in $G$, we see that 
\begin{eqnarray*}
\min_{v_1,v_2 \in Y^{(0)}} d_{1}(v_1,v_2) d_G(g,h) \le
d_{1}(g,h). 
\end{eqnarray*}

We will define $\CompareXG$ to be 
\begin{eqnarray*}
\CompareXG = \frac{1}{\min_{v_1,v_2 \in Y^{(0)}} d_{1}(v_1,v_2)}
\prod_{i=1}^n
\max\left(\sqrt{\frac{2}{1-\cos(\alpha)}},\frac{\max_{y_1,y_2 \in Y^{(i-1)}}
      d_{i-1}(y_1,y_2)}{\min_{v_1,v_2 \in Y^{(0)}} d_{1}(v_1,v_2)} \right).
\end{eqnarray*}
  
This gives us the inequality:
\begin{eqnarray*}
d_G(g,h) \le \CompareXG d_{X}(g,h).
\end{eqnarray*}
The inequality implies the containment $G \cap
B_X(\frac{r}{\CompareXG}) \subset B_G(r)$ by the argument from the
start of this proof.
\end{proof}
\section{Comparing functions on $X$ with corresponding ones on $G$}
We can compare the norm of $f$ with the norm of $\group{f}$, as well
as the norm of $\grad f$ with that of its analogue.  Note that given a
radius, $R$, Corollary \ref{PoincareExtend1} tells us that we have a
uniform Poincar\'{e} inequality for $f$ on balls of radius at most
$R$.  This will be helpful for our comparison.  In particular, we will
use this where $\cpon$ is the constant associated to the Poincar\'{e}
inequality for balls of radius up to $3 \diam(Y)$.  Note that if we
took $\delta =\diam(Y)$, we could cover $Y$ with exactly one ball.
\begin{lemma}\label{fToGroupf}
Let $B_X(r) := B_X(g',r)$ be a ball in $X$ centered at $g' \in G$.  For any
$c \in R$, we can compare $f:X\rightarrow R$ with $\group{f}: G^N
\rightarrow R$ in the following manner:
\begin{eqnarray*}
||f-c||_{p,B_X(r)}^p \le C \left(\delta^p ||\grad f||_{p,B_X(r+2\delta)}^p
         + ||\group{f} -c||_{p,B_G(\CompareXG(r+\delta+\diam(Y)))}^p\right).
\end{eqnarray*}
When $r=\infty$, this says that:
\begin{eqnarray*}
||f-c||_{p,X}^p 
  \le C \left(\delta^p ||\grad f||_{p,X}^p +||\group{f} -c||_{p,G}^p \right).
\end{eqnarray*}
The constant $C$ depends on $X$, $p$, and $\delta$.
\end{lemma}
\begin{proof}
We begin by rewriting the norm using the fact that balls of radius $r
+ \delta$ centered at $g \gamma_i$ form a cover.
\begin{eqnarray*}
||f-c||_{p,B_X(r)}^p &=& \int_{B_X(r)} |f(x) -c|^p dx \\
&\le& \sum_i \sum_{g\gamma_i \in B_X(r+ \delta)} 
                \int_{B_X(g\gamma_i,\delta)} |f(x) -c|^p dx .
\end{eqnarray*}
We'll use the fact that 
$|f(x) -c|^p \le 2^p |f(x) -\group{f(g,i)}|^p + 2^p |\group{f(g,i)} -c|^p$
 to split this into two pieces.  In the first piece, we can simplify using
the local Poincar\'{e} inequality in $X$.
\begin{equation*}
\begin{split}
\sum_i \sum_{g\gamma_i \in B_X(r + \delta)} 
&      2^p \int_{B_X(g\gamma_i,\delta)} |f(x) -\group{f(g,i)}|^p dx \\
&\le \sum_i \sum_{g\gamma_i \in B_X(r + \delta)} 
      2^p \delta^p \cpon \int_{B_X(g\gamma_i,\delta)} |\grad f(x)|^p dx \\
&\le 2^p \delta^p \cpon \cover \int_{B_X(r + 2\delta)} |\grad f(x)|^p dx .
\end{split}
\end{equation*}
In the second, we first note that there is no $x$ dependence in the
integrand. We integrate to get the volume of the ball.  This will be
dominated by the largest such volume.
\begin{equation*}
\begin{split}
\sum_i &\sum_{g\gamma_i \in B_X(r + \delta)} 
     2^p \int_{B_X(g\gamma_i,\delta)}  |\group{f(g,i)} -c|^p dx \\
&= \sum_i \sum_{g\gamma_i \in B_X(r + \delta)} 
     2^p \mu(B_X(g\gamma_i,\delta)) |\group{f(g,i)} -c|^p  \\
&\le  2^p 
\left( \max_{g\gamma_i \in B_X(r+ \delta)} \mu(B_X(g\gamma_i,\delta)) \right)
     \sum_i \sum_{g\gamma_i \in B_X(r + \delta)} |\group{f(g,i)} -c|^p 
\end{split}
\end{equation*}
This gives us the definition of the $p$ norm in $X$.  
We switch to the norm in $G$ using the distance comparisons from Lemma
\ref{CompareDistance}.
\begin{equation*}
\begin{split}
... &=   2^p \max_{g\gamma_i \in B_X(r+ \delta)} \mu(B_X(g\gamma_i,\delta)) 
             ||\group{f} -c||_{p,B_X(r + \delta)}^p \\
&\le   2^p \max_{g\gamma_i \in B_X(r+ \delta)} \mu(B_X(g\gamma_i,\delta)) 
             ||\group{f} -c||_{p,B_G(\CompareXG(r+\delta+\diam(Y)))}^p .
\end{split}
\end{equation*}
When we put these together we have for some constant $C$:
\begin{eqnarray*}
||f-c||_{p,B_X(r)}^p
\le  C \left(\delta^p ||\grad f||_{p,B_X(r+2\delta)}^p 
         + ||\group{f} -c||_{p,B_G(\CompareXG(r+\delta+\diam(Y)))}^p\right).
\end{eqnarray*}
Note that the uniformity of $X$ tells us that
$\mu(B_X(g\gamma_i,\delta))$ can be bounded by a constant.  In
particular, we use the fact that
$\mu(B_X(g\gamma_i,\delta))=\mu(B_X(h\gamma_i,\delta))$ for any
$g,h \in G$.
\end{proof}
We can also bound the gradients of $f$ and $\group{f}$ in their
respective norms.
\begin{lemma}\label{gradGroupftogradf}
Let $f\in \Lip(X)$ and $B_G(r)$ be given.  For $1 \le p <\infty$, we have
\begin{eqnarray*}
||\grad \group{f} ||_{p,B_G(r)}^p
\le  C(\delta) ||\grad f||_{p,B_X(\CXG r+2 \diam(Y))}^p.
\end{eqnarray*}
When $r = \infty$, this is:
\begin{eqnarray*}
||\grad \group{f} ||_{p,G}^p \le  C(\delta) ||\grad f||_{p,X}^p.
\end{eqnarray*}
$C(\delta) = \max_{\gamma_i}
     \frac{\mu(B_X(g\gamma_i,\Size))}{\mu(B_X(g\gamma_i,\delta))^2} N
     \max_{x\in X} \# \{B_X(g,\Size) | x \in B_X(g,\Size)\} \cpon
     \Size^p$. \\  Note that this constant depends on $X$, $p$, $N$, and
     $\delta$.
\end{lemma}
\begin{proof}
Here, $\Size$ is a large enough radius so that for any $g$ and
$\gamma_i$ both $B_X(g\gamma_i,\delta)$ and $B_X(gs\gamma_i,\delta)$
are covered by $B_X(g\gamma_i,\Size)$.  Note that $\Size = \diam(Y)+\delta$
will work, but to remove the dependence on $\delta$, we can take
$\Size = 2\diam(Y)$.

We start by explicitly writing out the gradient and then moving the
$p/2$ into the integral via Jensen.
\begin{eqnarray*}
\lefteqn{||\grad \group{f} ||_{p,B_G(r)}^p}\\
& = &
  \sum_i \sum_{g\in B_G(r)} 
  \left( \frac{1}{|S|}\sum_{s \in S} |\group{f(g,i)} -\group{f(gs,i)}|^2 \right)^{p/2}\\
&\le& \sum_i \sum_{g\in B_G(r)} 
   \frac{1}{|S|} \sum_{s \in S} |\group{f(g,i)} -\group{f(gs,i)}|^{p}\\
&= &\sum_i \sum_{g \in B_G(r)} 
  \frac{1}{|S|} \sum_{s \in S}  
   | \dashint_{B_X(g\gamma_i,\delta)} f(x) dx 
    - \dashint_{B_X(gs\gamma_i,\delta)} f(y) dy|^p.
\end{eqnarray*}
We apply Jensen again; this time to the absolute value. 
\begin{eqnarray*}
...&\le& \sum_i \sum_{g \in B_G(r)} 
   \frac{1}{|S|}\sum_{s \in S} 
\dashint_{B_X(g\gamma_i,\delta)} \dashint_{B_X(gs\gamma_i,\delta)}
               |f(x) - f(y)|^p dx dy .
\end{eqnarray*}
The regularity of the space and the cover tell us 
$\mu(B_X(g\gamma_i,\delta)) =\mu(B_X(g s \gamma_i,\delta))$. 
\begin{eqnarray*}
...&=& \sum_i \sum_{g \in B_G(r)} 
   \frac{1}{|S|}\sum_{s \in S} 
               \int_{B_X(g\gamma_i,\delta)}  \int_{B_X(gs\gamma_i,\delta)}
               |f(x) - f(y)|^p \frac{dx dy}{\mu(B_X(g\gamma_i,\delta))^2}.
\end{eqnarray*}
We expand the sets we are integrating over to $B_X(g\gamma_i,\Size)$.
This larger set contains both $B_X(g\gamma_i,\delta)$ and
$B_X(gs\gamma_i,\delta)$ by construction.  We then rewrite the sum
over $S$, and change one integral to an average integral.
\begin{eqnarray*}
...&\le &\frac{1}{|S|} 
\sum_i \sum_{g \in B_G(r)} 
    \sum_{s \in S} \frac{1}{\mu(B_X(g\gamma_i,\delta))^2}
                 \int_{B_X(g\gamma_i,\Size)}  \int_{B_X(g\gamma_i,\Size)}
                 |f(x) - f(y)|^p dx dy \\
& = & \frac{1}{|S|} 
\sum_i \sum_{g \in B_G(r)} 
    |S| \frac{\mu(B_X(g\gamma_i,\delta))}{\mu(B_X(g\gamma_i,\delta))^2}
                 \int_{B_X(g\gamma_i,\Size)}  \dashint_{B_X(g\gamma_i,\Size)}
                 |f(x) - f(y)|^p dx dy .
\end{eqnarray*}
We now apply a local $p$ Poincar\'{e} inequality on $X$ to $f$.  The
constant for this is $\cpon$.
\begin{eqnarray*}
...&\le & \sum_i \sum_{g \in B_G(r)} 
    \cpon \Size^p 
    \frac{\mu(B_X(g\gamma_i,\Size))}{\mu(B_X(g\gamma_i,\delta))^2} 
         \int_{B_X(g\gamma_i,\Size)} |\grad f(x)|^p dx.
\end{eqnarray*}
We combine the sums and integral into a single integral.  All of the
$B_X(g\gamma_i,\Size)$ for $g\in B_G(r)$ are contained in $B_X(\CXG
r+\Size)$ by our distance comparison between $G$ and $X$.  We multiply
this integral by the number of overlapping balls in our sum. $C_M$ is
$N$ times the maximum number of balls $B_X(g\gamma_i,\Size)$ which
overlap at a point in $X$.
\begin{eqnarray*}
...&\le& \left(\max_{\gamma_i} 
     \frac{\mu(B_X(g\gamma_i,\Size))}{\mu(B_X(g\gamma_i,\delta))^2}\right)
     C_{M} \cpon \Size^p \int_{B_X(\CXG r+\Size)} 
     |\grad f(x)|^p dx \\
&=& C(\delta) ||\grad f||_{p,B_X(\CXG r+\Size)}^p.
\end{eqnarray*}
\end{proof}

\section{Poincar\'{e} inequalities on X with underlying group structure}
The bounds in the previous section can be used to transfer
inequalities between $X$ and $G$.  We can combine them with the weak
Poincar\'{e} inequality on $G$ to get an inequality on $X$.
\begin{theorem}\label{WeakGtoX}
Let $X$, a volume doubling Euclidean complex and $G$, a finitely
generated group with $X/G=Y$, a finite admissible polytopal complex be
given.  $X$ admits a Poincar\'{e} inequality with uniform
constant at all scales.  Let $f\in \Lip(X)$.  For $1 \le p < \infty$,
we have:
\begin{eqnarray*}
\inf_c||f-c||_{p,B_X(r)} \le C r ||\grad f||_{p,B_X(r)}.
\end{eqnarray*}
Note that this implies:
\begin{eqnarray*}
||f-f_{B_X(r)}||_{p,B_X(r)} \le 2C r ||\grad f||_{p,B_X(r)}.
\end{eqnarray*}
Here the balls can be centered at any point in $X$.
\end{theorem}
\begin{proof} 
Note that we chose $\cpon$ so that the Poincar\'{e} inequality holds
for balls of radius up to $3\diam(Y)$.  We need to show that it also
holds for balls of radius greater than $3\diam(Y)$.  Let $r\ge
3\diam(Y)$ be given.  To start, we will assume that the center of
$B_X(r)$ is in $G$.  Pick $\delta=\diam(Y)$; this will force $N=1$.
This will allow us to split things up in such a way that we can use
the weak Poincar\'{e} inequality on $G$.  If we had multiple copies of
$G$, we wouldn't necessarily have the same average on each of them.
By choosing a value of $c$, we obtain something at least as large as
the infimum:
\begin{eqnarray*}
\inf_c ||f-c||_{1,B_X(r)} 
  \le ||f-\group{f}_{B_G(2\CompareXG r)}||_{1,B_X(r)}.
\end{eqnarray*}
Then we can use our first bound to get 
\begin{eqnarray*}
\lefteqn{||f-\group{f}_{B_G(2\CompareXG r)}||_{1,B_X(r)}} \\
&\le&  C \left(\delta ||\grad f||_{1,B_X(r+\delta)} 
        + ||\group{f} -\group{f}_{B_G(2\CompareXG r)}
            ||_{1,B_G(\CompareXG(r+\delta+\diam(Y)))}\right)\\
&\le &C \left(\delta ||\grad f||_{1,B_X(1.5r)} 
     + ||\group{f} -\group{f}_{B_G(2\CompareXG r)}||_{1,B_G(2\CompareXG r)}\right).
\end{eqnarray*}
Happily, we can apply the weak Poincar\'{e} inequality on groups
(Lemma \ref{groupVDpoincare}) to the second term:
\begin{eqnarray*}
||\group{f} -\group{f}_{B_G(2\CompareXG r)}||_{1,B_G(2\CompareXG r)}
\le 3 r \sqrt{|S|} ||\grad \group{f}||_{1,B_G(6\CompareXG r)}.
\end{eqnarray*}
Then, we can use the bound we have on the gradients to get an
inequality on $\grad f$.  Setting $\Size=2\diam(Y) < r$ tells us that
$B_X((6r+\Size)\CXG \CompareXG) \subset B_X(7\CXG \CompareXG r)$.
\begin{eqnarray*}
||\grad \group{f} ||_{1,B_G(6\CompareXG r)}
\le C(\delta) ||\grad f||_{1,B_X(7\CXG \CompareXG r)}.
\end{eqnarray*}
 Combining these, we have:
\begin{eqnarray*}
\inf_c ||f-c||_{1,B_X(r)} &\le& 
C \left(\delta ||\grad f||_{1,B_X(1.5r)} +3 r \sqrt{|S|} 
              C(\delta) ||\grad f||_{1,B_X(7\CXG \CompareXG r)}\right) \\
&\le& C_0 r ||\grad f||_{1,B_X(7\CXG \CompareXG r)}.
\end{eqnarray*}
As in lemma \ref{PoincareP}, we have:
\begin{eqnarray*}
||f-f_{B_X(r)}||_{1,B_X(r)} \le 2 \inf_c ||f-c||_{1,B_X(r)}. 
\end{eqnarray*}
If the center, $x$, were not in $G$, there is some $g' \in G$ such
that the center is within $\diam(Y)$ of $g'$.  That is, $d_X(x,g') \le
\diam(Y)$.  By inclusions of balls, we know that:
\begin{eqnarray*}
\inf_c ||f-c||_{1,B_X(x,r)} \le \inf_c ||f-c||_{1,B_X(g',r+\diam(Y))}.
\end{eqnarray*}
As $r+\diam(Y) \le 1.5r$ and $B_X(g',7 \CXG \CompareXG 1.5 r) \subset
B_X(x,12 \CXG \CompareXG r)$, we can switch centers by increasing the radius:
\begin{eqnarray*}
||\grad f||_{1,B_X(g',7 \CXG \CompareXG 1.5 r)}
\le ||\grad f||_{1,B_X(x, 12\CXG \CompareXG r)}.
\end{eqnarray*}
This tells us that any complex $X$ with the underlying group structure
admits a weak $p=1$ Poincar\'{e} inequality.  

$X$ is volume doubling, and so this weak inequality can be turned into
a strong p inequality via repeated application of a Whitney cover,
using Corollary \ref{PoincareExtend1}.
\end{proof}

In \cite{Varo} Varopolous showed that groups with polynomial growth of
degree $d$ have on diagonal behavior $p_{2n}(e,e) \approx n^{-d/2}$.
We show that a similar result holds for volume doubling
complexes with underlying group structure.

\begin{theorem} \label{XuniformPoincarefromG}
Assume $X$ is a volume doubling Euclidean complex and $G$ is a
finitely generated group with $X/G=Y$, where $Y$ is a finite
admissible polytopal complex.

Then $X$ satisfies the on diagonal heat kernel estimates:
\begin{eqnarray*}
\frac{1}{C \mu(B(x,\sqrt{t}))}
\le h_t(x,x) \le \frac{C}{\mu(B(x,\sqrt{t}))} 
\end{eqnarray*}
$X$ also satisfies the off diagonal heat kernel lower bound:
\begin{eqnarray*}
\frac{1}{C \mu(B(x,\sqrt{t}))} \exp\left(-C \frac{d_X^2(x,y)}{t}\right)
\le h_t(x,y),
\end{eqnarray*}
as well as the upper bound:
\begin{eqnarray*}
h_t(x,y) \le \frac{C}{\sqrt{\mu(B(x,\sqrt{t}))\mu(B(y,\sqrt{t}))}}
    \exp\left(-\frac{d_X^2(x,y)}{4t}\right) \left(1 +
    \frac{d_X^2(x,y)}{t}\right)^{n/2}.
\end{eqnarray*}
\end{theorem}
\begin{proof}
To get the heat kernel bounds, apply Sturm \cite{Sturm} theorems
\ref{Sturm1} and \ref{Sturm2}, noting that we've satisfied both volume
doubling and a Poincar\'{e} inequality at all scales uniformly.
\end{proof}
\section{Mapping functions on $G$ to $X$}
Now we will look at how to take functions on $G$ to smooth versions on
$X$.  Let $f$ be a function mapping $G$ to the reals.  We'll look at a
partition of unity on the complex, $X$, which is created by
translating a smooth function $\chi$ by $g\in G$.  Then $\sum_{g \in
G} \chi_g(x) =1$.  We require the following:
\begin{itemize}
\item $\chi_g(x) = 1$ if $d_X(x,g) \le \frac{1}{4}$
\item $\chi_g(x) = 0$ if $d_X(x,g) \ge \csup$
\item $|\grad \chi_g(x)| \le \Const$
\end{itemize}

We know that $|\{g \in G : d_X(g,e) \le \csup \}|$ is finite; when $Y$
is nice and $\csup=1$ this will be $|S|$.  We also know that for any
$x\in X$, $|\{g \in G : \chi_g(x) \ne 0\}|$ is finite.  In particular,
there is a uniform bound, $\coverlap$.

This allows us to define a nice smooth function, $\comp{f(x)}$,
mapping $X$ to the reals:
\begin{eqnarray*}
\comp{f(x)} = \sum_{g \in G} f(g) \chi_g(x).
\end{eqnarray*}
Its $L^p$ norm is comparable to that of $f$.
\begin{theorem}\label{Compcompare}
Let $f: G\rightarrow R$ be given. If we limit ourselves to a ball,
$B_G(r)$, with radius at least 1, we can compare $L^p$ norms in the
following way:
\begin{eqnarray*}
C_1 ||f-c||_{p,B_G(\frac{r-.25}{\CXG})}
  \le ||\comp{f}-c||_{p,B_X(r)} 
\le C_2 ||f-c||_{p,B_G(\CompareXG(r+\csup))}
\end{eqnarray*}
This holds for any $c\in R$.
Note that when $r=\infty$ and $c=0$ we have a nice bound on the norms:
\begin{eqnarray*}
C_1 ||f||_{p,G} \le ||\comp{f}||_{p,X} \le C_2 ||f||_{p,G}.
\end{eqnarray*}
For both of these inequalities,
$C_1=\mu(B_X(e,\frac{1}{4}))^{\frac{1}{p}}$ and
$C_2=\coverlap^{(p-1)/p} ||\chi_e||_{p,X}$.
\end{theorem}
\begin{proof}
If we limit ourselves to a ball, $B_G(r)$, with radius at least 1, we
have a comparison.
We first write out the definition of the norm, and then we use the fact that for every $x$, $\sum_{g \in G} \chi_g(x) =1$.
\begin{eqnarray*}
||\comp{f} -c||_{p,B_X(r)} 
&=& \left(\int_{B_X(r)} |\sum_{g \in G} f(g) \chi_g(x)-c|^p dx\right)^{\frac{1}{p}} \\
&=& \left(\int_{B_X(r)} |\sum_{g \in G} (f(g)-c) \chi_g(x)|^p dx\right)^{\frac{1}{p}}. 
\end{eqnarray*}
For each $x$, at most $\coverlap$ of the $\chi_g(x)$ are nonzero.
This allows us to apply a discrete version of Jensen to move the
exponent into the sum.
\begin{eqnarray*}
...&\le& \left(\int_{B_X(r)}\coverlap^{p-1} 
    \sum_{g \in G} |f(g)-c|^p \chi_g(x)^p dx\right)^{\frac{1}{p}}.
\end{eqnarray*}
The only $g$ with a nonzero $\chi_g(x)$ will be those within $X$
distance $\csup$ of a point in $B_X(r)$.  We can integrate over $g \in
G\cap B_X(r+\csup)$, and switch the finite integral and sum.
\begin{eqnarray*}
...&\le& \coverlap^{(p-1)/p}
  \left(\sum_{g \in G\cap B_X(r+\csup)}|f(g)-c|^p \int_{B_X(r)} \chi_g(x)^p
    dx\right)^{\frac{1}{p}}.
\end{eqnarray*}
The quantity $\int_{B_X(r)} \chi_g(x)^p$ will be bounded above by
$\int_{X} \chi_g(x)^p = \int_{X} \chi_e(x)^p$.
\begin{eqnarray*}
...&\le& \coverlap^{(p-1)/p} ||\chi_e||_{p,X}
  \left(\sum_{g \in G\cap B_X(r+\csup)}|f(g)-c|^p \right)^{\frac{1}{p}}.
\end{eqnarray*}
We then use the distance comparisons from lemma \ref{CompareDistance}
to get a norm with respect to distance in $G$.
\begin{eqnarray*}
...&\le& 
\coverlap^{(p-1)/p} ||\chi_e||_{p,X} ||f-c||_{p,B_G(\CompareXG(r+\csup))}.
\end{eqnarray*}
Now we will show the other inequality.  By definition, we can write
the norm in $G$ as:
\begin{eqnarray*}
||f-c||_{p,B_G(r)} 
&=& \left(\sum_{g \in B_G(r)} |f(g)-c|^p\right)^{\frac{1}{p}} .
\end{eqnarray*}
We introduce $\chi_g$ by noting $\chi_g(x) = 1$ for $x$ in
$B_X(g,\frac{1}{4})$, and integrating over this set.  Due to the
regularity of $X$, $\mu(B_X(g,\frac{1}{4}))$ does not depend on $g$,
and so we write it as $\mu(B_X(e,\frac{1}{4}))$.
\begin{eqnarray*}
...&=& \left(\sum_{g \in B_G(r)} \frac{1}{\mu(B_X(e,\frac{1}{4}))} 
    \left(\int_{B_X(g,\frac{1}{4})}
     |f(g)-c|^p \chi_g(x) dx \right)\right)^{\frac{1}{p}}.
\end{eqnarray*}
We now will switch the integral and the sum.  We are integrating only
over $x$ in balls centered at points in $B_G(r)$ of radius $1/4$.  This set
can be written $\cup_{h \in B_G(r)}B_X(h,\frac{1}{4})$. 
\begin{eqnarray*}
...&\le& \left(\frac{1}{\mu(B_X(e,\frac{1}{4}))} 
   \int_{\cup_{h \in B_G(r)}B_X(h,\frac{1}{4}) }  
\sum_{g \in G}
 |f(g)-c|^p \chi_g(x) dx \right)^{\frac{1}{p}}.
\end{eqnarray*}
For $x$ in $\cup_{h \in B_G(r)}B_X(h,\frac{1}{4})$, $\chi_g(x)=1$ for
exactly one $g \in G$, and it is zero otherwise.  This tells us
$\sum_{g \in G} |f(g)-c|^p \chi_g(x) = |\sum_{g \in G}(f(g)-c)
\chi_g(x)|^p$.  We then can write the above as
\begin{eqnarray*}
...&=&  \left(\frac{1}{\mu(B_X(e,\frac{1}{4}))} 
   \int_{\cup_{h \in B_G(r)}B_X(h,\frac{1}{4}) }  
   |\sum_{g \in G}  (f(g)-c) \chi_g(x)|^p dx \right)^{\frac{1}{p}}.
\end{eqnarray*}
We use the distance comparisons from Lemma \ref{CompareDistance}
to see that $\cup_{h \in B_G(r)}B_X(h,\frac{1}{4}) \subset B_X(\CXG
r+.25)$.
\begin{eqnarray*}
...&\le& \left(\frac{1}{\mu(B_X(e,\frac{1}{4}))} 
   \int_{B_X(\CXG r+.25) }  
   |\sum_{g \in G}  (f(g)-c) \chi_g(x)|^p dx \right)^{\frac{1}{p}}.
\end{eqnarray*}
Now we rewrite this using the fact that $\sum_{g \in G}\chi_g(x) =1$.
\begin{eqnarray*}
...&=& \left(\frac{1}{\mu(B_X(e,\frac{1}{4}))} 
   \int_{B_X(\CXG r+.25) }  
   |\sum_{g \in G}  f(g) \chi_g(x)-c|^p dx \right)^{\frac{1}{p}} \\
&=& 
 \left(\frac{1}{\mu(B_X(e,\frac{1}{4}))}\right)^{\frac{1}{p}}
||\comp{f}-c||_{p,B_X(\CXG r+.25)}.
\end{eqnarray*}
\end{proof}
We'd also like to compare the norms of the gradients.  To do this, we
want to write the gradient in such a way that we can compare it with
the one on $G$.   We first note that:
\begin{eqnarray*}
\grad\left(\sum_{g \in G} \chi_g(x)\right) = \grad 1 =0.
\end{eqnarray*}

This allows us to write the gradient of $\comp{f(x)}$ as 
\begin{eqnarray*} 
\grad \comp{f(x)} = \sum_{g \in G} f(g) \grad \chi_g(x) = \sum_{g \in G}
(f(g) -f(h))\grad \chi_g(x).
\end{eqnarray*}
\begin{lemma}\label{CompGradcompare}
Let $f: G\rightarrow R$ be given. Then for any $r$ we have:  
\begin{eqnarray*}
||\grad \comp{f(x)}||_{p,B_X(r)}^p 
\le C ||\grad f||^p_{p,B_G(\CompareXG(r+3\csup))}
\end{eqnarray*}
where $C = \Const^p \mu(B_X(e,\csup)) \Vol_G(G \cap B_X(e,2\csup))^p |S|^p$.

If we have $B(e,2\csup) = S$, the generating set, then this is:
\begin{eqnarray*}
||\grad \comp{f(x)}||_{p,B_X(r)}^p 
\le C ||\grad f||^p_{p,B_G(\CompareXG(r+\csup))}
\end{eqnarray*}
where $C = \Const^p \mu(B_X(e,\csup)) |S|^p $.

When $r=\infty$, this is:
\begin{eqnarray*}
||\grad \comp{f(x)}||_{p,X}^p \le C ||\grad f||^p_{p,G}.
\end{eqnarray*}
\end{lemma}
\begin{proof}
We can cover $X$ with balls of radius $\csup$.  This lets us rewrite the
norm as follows:
\begin{eqnarray*}
||\grad \comp{f(x)}||_{p,B_X(r)}^p 
& = &\int_{B_X(r)} |\grad \comp{f(x)}|^p dx \\
& \le& \sum_{h \in G\cap B_X(r+\csup)} \int_{B_X(h,\csup)} 
    |\grad \comp{f(x)}|^p dx \\
& =& \sum_{h \in G\cap B_X(r+\csup)} \int_{B_X(h,\csup)} 
    \abs{\sum_{g \in G} (f(g) -f(h))\grad \chi_g(x)}^p dx 
\end{eqnarray*}
 From its definition, we know that $\grad \chi_g(x)$ will be nonzero
 only when $d_X(x,g) < \csup$.  As we're integrating over $x$ with
 $d_X(x,h) \le \csup$, we can restrict our possible $g$ to those with
 $d_X(g,h) < 2\csup$.  Then we use the fact that $|\grad \chi_g(x)|
 \le \Const$.
\begin{equation*}
\begin{split}
\sum_{h \in G\cap B_X(r+\csup)}& \int_{B_X(h,\csup)} 
    \abs{\sum_{g \in B_X(h,2\csup)} (f(g) -f(h))\grad \chi_g(x)}^p dx \\
&\le \Const^p \sum_{h \in G\cap B_X(r+\csup)}
   \abs{ \sum_{g \in G \cap B_X(h,2\csup)} (f(g) -f(h)) }^p
   \mu(B_X(h,\csup)) .
\end{split}
\end{equation*}
Note that by invariance, $\mu(B_X(h,\csup)) =\mu(B_X(e,\csup))$.  At
this point, if we had $G \cap B_X(e,2\csup) = S$, the generating set,
we could proceed as follows.  Otherwise, we'll need to expand things a
little bit more.
\begin{equation*}
\begin{split}
 \Const^p &\mu(B_X(e,\csup)) 
   |S|^p \sum_{h \in G\cap B_X(r+\csup)}
     \abs{ \sum_{s\in S} \frac{1}{|S|}(f(hs) -f(h)) }^p \\
&\le \Const^p \mu(B_X(e,\csup)) |S|^p 
   \sum_{h \in G\cap B_X(r+\csup)}
      \left( \abs{\sum_{s\in S} \frac{1}{|S|}(f(hs) -f(h) )^2}\right)^{p/2} \\
&= \Const^p \mu(B_X(e,\csup)) |S|^p ||\grad f||^p_{p,G \cap B_X(r+\csup)} \\
&\le \Const^p \mu(B_X(e,\csup)) |S|^p 
   ||\grad f||^p_{p,B_G(\CompareXG(r+\csup))}.
\end{split}
\end{equation*}
If $G \cap B(e,2\csup) \ne S$, we could modify this by noting that:
\begin{eqnarray*}
&&\lefteqn{\sum_{h \in G\cap B_X(r+\csup)}
 \abs{ \sum_{g \in G\cap B_X(h,2\csup)} (f(g) -f(h)) }^p}
\\&& = \sum_{h \in G\cap B_X(r+\csup)}
   \abs{ \sum_{g \in G\cap B_X(h,2\csup)} 
     \sum_{i=0 : s_0..s_k=h^{-1}g}^{k-1} (f(hs_0..s_i) -f(hs_0..s_{i+1})) }^p 
\\
&&\le \sum_{h \in G\cap B_X(r+3\csup)}
   \Vol_G(G \cap B_X(e,2\csup))^p 
     \abs{\sum_{s \in S} (f(hs) -f(h)) }^p.
\end{eqnarray*}
This will yield the inequality:
\begin{eqnarray*}
||\grad \comp{f(x)}||_{p,B_X(r)}^p
\le C ||\grad f||^p_{p,B_G(\CompareXG(r+3\csup))}
\end{eqnarray*}
for $C =\Const^p \mu(B_X(e,\csup)) \Vol_G(G \cap B_X(e,2\csup))^p |S|^p $.

Note that in these, $\Const$ is the bound on the gradient of $\chi_g$.
\end{proof}
\section{Poincar\'{e} inequality for volume doubling finitely generated groups}
We can use these estimates along with our knowledge of complexes in
order to show that volume doubling finitely generated groups admit a
strong Poincar\'{e} Inequality.  This is not a new fact, but it is a
cute proof.
\begin{theorem}
Let $G$ be a finitely generated volume doubling group.
Let $f: G\rightarrow R$ and $B_G(r) \subset G$ be given.  Then  
\begin{eqnarray*}
||f-f_{B_G(r)}||_{1,B_G(r)} \le C r ||\grad f||_{1,B_G(r)}.
\end{eqnarray*}
Here $C=4 C_P \Const |S|\csup $ where $C_P$ is the constant in the
global Poincar\'{e} inequality for $X$.
\end{theorem}
\begin{proof}
Take any such group, and let $X$ be its Cayley graph.  Theorem
\ref{XuniformPoincarefromG} showed that strong Poincar\'{e}
inequalities hold on $X$.  We happily note that both $\CXG$ and
$\CompareXG$ are $1$ on a Cayley graph, and so we can omit them from
our calculation.  We form a chain of inequalities as follows.  From
Theorem \ref{Compcompare}, we can set $c=(\comp{f})_{B_X(r+.25)}$ to get:
\begin{equation*}
\begin{split}
||f-(\comp{f})_{B_X(r+.25)}&||_{1,B_G(r)} \\
& \le \frac{1}{\mu(B_X(g,\frac{1}{4}))}
||\comp{f}-(\comp{f})_{B_X(r+.25)}||_{1,B_X(r+.25)}.
\end{split}
\end{equation*}
Note that for every $g \in G$, $\mu(B_X(g,\frac{1}{4}))= 1/4 |S|$.
 From Theorem \ref{XuniformPoincarefromG}, we know that:
\begin{eqnarray*}
 ||\comp{f}-(\comp{f})_{B_X(r+.25)}||_{1,B_X(r+.25)} 
\le C_P r ||\grad \comp{f}||_{1,B_X(r+.25)}.
\end{eqnarray*}
Then we transfer back, using the fact that $G \cap B_X(e,2\csup) =S$.
\begin{eqnarray*}
||\grad \comp{f(x)}||_{1,B_X(r+.25)} 
= \Const \mu(B_X(\csup)) |S| ||\grad f||_{1,B_G(r+.25+\csup)}.
\end{eqnarray*}
We can evaluate this as $X$ is a Cayley graph: $\mu(B_X(e,\csup)) =
|S|\csup$.

Since $X$ is a graph whose edges have unit length, $\csup <1$.  In
particular, we can pick $\csup=.74$.  Since our original ball, $B_G(r)$
is on the group, without loss of generality we know that $r$ is an
integer.  Then $B_G(r + \csup +.25) = B_G(r+.99) = B_G(r)$ on the
group.  Combining this, we have:
\begin{eqnarray*}
||f-\comp{f}_{B_X(r+.25)}||_{1,B_G(r)} \le 
  \frac{1}{.25|S|}
   C_P r \Const |S|\csup |S| ||\grad f||_{1,B_G(r)}.
\end{eqnarray*}
We can get the desired left hand side from
\begin{eqnarray*}
||f-f_{B_G(r)}||_{1,B_G(r)} \le
||f-\comp{f}_{B_X(r+.25)}||_{1,B_G(r)}.
\end{eqnarray*}
  We can use the graph structure to reduce this to:
\begin{eqnarray*}
||f-f_{B_G(r)}||_{1,B_G(r)} \le 
  4 C_P \Const |S|\csup r ||\grad f||_{1,B_G(r)}.
\end{eqnarray*}
\end{proof}

\chapter{Comparing Heat Kernels on X and G}
The main goal of this chapter is to show that for large times, the
heat kernel on the group is comparable to the heat kernel on the
complex.  The comparison was shown for groups and manifolds by
Saloff-Coste and Pittet \cite{LSCP}.
\begin{notation}
To simplify notation, we use $p_t$ for the heat kernel on the
group, and $h_t$ when it is on the complex.
\end{notation}
On a finitely generated group, the heat kernel can be used to describe
a symmetric random walk.  This is a walk where from a point $g\in G$,
the probability of moving to $gs$ in one step is $\frac{1}{|S|}$ for
each generator $s \in S$.  The value of the heat kernel on the
diagonal, $p_{2n}(e,e)$ gives us the probability of returning to the
same point after $2n$ steps.  We are interested in this for even
numbers of steps because this avoids parity issues.  The set-up for
these walks can be found in \cite{Kesten}.
\begin{definition}
We say $f(t) \approx g(t)$ if there exist positive finite constants
$C_1,C_2,C_3$, and $C_4$ so that
\begin{eqnarray*}
C_1 f(C_2 t) \le g(t) \le C_3 f(C_4 t).
\end{eqnarray*}
\end{definition}
We will show that the following holds when $t \ge 1$:
\begin{eqnarray*}
p_{2\floor{t}}(e,e) \approx \sup_{x \in X} h_t(x,x). 
\end{eqnarray*}
Note that it doesn't make sense to compare them for small times, since
$p_t$ is only defined for integer values of $t$.  

An important notion in this proof is that of amenability.
\begin{definition}
A {\bf F\o lner sequence} is a sequence of finite subsets, $F(i)$,
with the following properties: 
\\ (1) For any $g \in G$ there exists $i$ such that $g \in F(i)$,
\\ (2) $F(i) \subset F(i+1)$, and 
\\ (3) For any finite subset $Q \subset G$, $\lim_{i \rightarrow \infty}
\frac{\#(QF(i))}{\# F(i)} =1$. 
\\ Here, $QF(i)$ refers to the set $\{g : g = qf \text{ with }q \in Q, 
f \in F(i)\} $.
\end{definition}
\begin{definition}
$G$ is {\bf amenable} if and only if $G$ admits a F\o lner sequence.
\end{definition}
\begin{example}
The group of integers, $Z$, is amenable.  Here, the sets $[-i,i]$ form
a F\o lner sequence.
\end{example}

In order to show this, we will split it into two cases.  In the first,
we look at when $G$ is nonamenable.  Here, $p_{2\floor{t}}(e,e)
\approx e^{-t}$.  Then we will look at when $G$ is amenable.  
We will first show $p_t$ is approximately less than or equal to $h_t$,
and then we will show the reverse.

\section{Heat kernels in the nonamenable case}
We now look at the behavior of the heat kernel on $X$ and $G$ when $G$
is nonamenable.

We call $H_t$ is the semigroup form for the heat kernel on
$X$.  It is related to $h_t(x,y)$ by $H_t f(x) = \int_X f(y) h_t(x,y)
dy$.  It is also written as $H_t = e^{-t\Delta}$.  Alternatively,
$h_t$ is called the transition function for $H_t$.  Estimates on norms
of functions and their derivatives can give us estimates on
$||H_t||_{2\rightarrow 2}$.
\begin{lemma}\label{fnormvsHt}
\begin{eqnarray*}
||f||_2 \le C ||\grad f||_2
\end{eqnarray*}
will be true for all $f\in \Dom(\Delta)$ if and only if for all $t>0$,
\begin{eqnarray*}
||H_t||_{2\rightarrow 2} \le e^{-t/C}.
\end{eqnarray*}
\end{lemma}
\begin{proof}
We will sketch the proof.  We can show the forward implication by
using integration by parts:
\begin{eqnarray*}
||\grad f||^2_2 = \int |\grad f|^2 = \int f \Delta f = \int
  \sqrt{\Delta}f \sqrt{\Delta} f = ||\sqrt{\Delta} f||^2_2.
\end{eqnarray*}
This tells us that for any non-zero $f\in \Dom(\Delta)$, we have:
\begin{eqnarray*}
\frac{||\sqrt{\Delta} f||^2_2}{||f||^2_2} \ge \frac{1}{C}.
\end{eqnarray*}
We can take a square root and then an infimum to get:
\begin{eqnarray*}
\inf_{f \ne 0} \frac{||\sqrt{\Delta} f||_2}{||f||_2} \ge \frac{1}{\sqrt{C}}.
\end{eqnarray*}
This tells us that $\frac{1}{\sqrt{C}}$ is a lower bound on
eigenvalues of $\sqrt{\Delta}$.  Spectral theory tells us that
$\frac{1}{C}$ is a lower bound on eigenvalues of $\Delta$, and
$e^{-\frac{t}{C}}$ is an upper bound on eigenvalues of $H_t =
e^{-t\Delta}$.  This yields
\begin{eqnarray*}
||H_t||_{2\rightarrow 2} \le e^{-t/C}.
\end{eqnarray*}
For the reverse, consider the fact that
\begin{eqnarray*}
E(f,f) = \lim_{t \rightarrow 0} \frac{\langle(H_t - I)f,f \rangle }{t} 
= -\langle \grad f ,\grad f \rangle.
\end{eqnarray*}
We can use our bound to get:
\begin{eqnarray*}
\lim_{t \rightarrow 0} \frac{\langle(H_t - I)f,f\rangle}{t} \le 
\lim_{t \rightarrow 0} \frac{\langle(e^{-t/C} - 1)f,f\rangle}{t}
= \langle(-1/C)f,f\rangle.
\end{eqnarray*}
Then since 
\begin{eqnarray*}
E(f,f) = -\langle \grad f ,\grad f\rangle,
\end{eqnarray*}
we have 
\begin{eqnarray*}
 -\langle\grad f ,\grad f\rangle \le \langle(-1/C)f,f\rangle
\end{eqnarray*}
which gives us
\begin{eqnarray*}
||f||_2 \le C||\grad f||_2.
\end{eqnarray*}
\end{proof}
  We can transfer between estimates on $||H_t||_{2\rightarrow 2}$ and
$h_t(x,y)$.  Since the bound on the norm of $f$ will hold for
nonamenable groups, we will combine lemmas \ref{fnormvsHt} and
\ref{Httopt} to get our heat kernel estimates.
\begin{lemma}
If $||H_t||_{2 \rightarrow 2}^2 \le e^{-2t/C}$, then for all $z\in X$
and $t \ge t'$:
\begin{eqnarray*}
h_t(z,z) \le h_{t'}(z,z)e^{-t/C}.
\end{eqnarray*}
\label{Httopt}
\end{lemma}
\begin{proof}
Apply $H_t$ to $f(y) = \frac{h_s(y,z)}{||h_s(\cdot,z)||_2}$.  
This gives you 
\begin{eqnarray*}
H_t\frac{h_s(x,z)}{||h_s(\cdot,z)||_2} = \int_X
\frac{h_s(y,z)}{||h_s(\cdot,z)||_2} h_t(x,y) dy =
\frac{h_{t+s}(x,z)}{||h_s(\cdot,z)||_2}.  
\end{eqnarray*}
Our estimate then tells us:
\begin{eqnarray*} 
||H_t||_{2 \rightarrow 2}^2 
\ge \int_X \frac{h_{t+s}(x,z)^2}{||h_s(\cdot,z)||_2^2} dx 
= \int_X \frac{h_{t+s}(x,z)h_{t+s}(z,x)}{||h_s(\cdot,z)||_2^2} dx 
=\frac{h_{2t+2s}(z,z)}{||h_s(\cdot,z)||_2^2}.
\end{eqnarray*}
Note that 
\begin{eqnarray*}
||h_s(\cdot,z)||_2^2 = \int_X h_s(y,z)^2 dy = h_{2s}(z,z).
\end{eqnarray*}
When we combine this with the inequality for $H_t$, we have:
\begin{eqnarray*}
\frac{h_{2t+2s}(z,z)}{h_{2s}(z,z)} \le e^{-2t/C}.
\end{eqnarray*}
Fix $z$ and let $u(t)= h_{t}(z,z)$.  This can be written as:
\begin{eqnarray*}
u(t+s) \le u(s)e^{-t/C}.
\end{eqnarray*}
This is equivalent to:
\begin{eqnarray*}
\frac{u(t+s)-u(s)}{t} \le u(s)\frac{e^{-t/C}-1}{t}.
\end{eqnarray*}
Taking the limit as $t \rightarrow 0^{+}$ gives us:
\begin{eqnarray*}
u'(s) \le (-1/C)u(s).
\end{eqnarray*}
This gives us the estimate that $u(t) \le u(t_0) e^{-t/C}$ for any $t\ge t_0$.
Rewriting this, we have the long time decay for all $z\in X$ and $t \ge t'$:
\begin{eqnarray*}
h_t(z,z) \le h_{t'}(z,z)e^{-t/C}.
\end{eqnarray*}
\end{proof}
Note that the converse is essentially true as well.  If $h_t(z,z) \le
h_{t'}(z,z)e^{-t/C}$ for $t \ge t'$, then we can construct an upper
bound for $||H_t||_{2 \rightarrow 2}^2$ whenever $t \ge t'$.
\begin{eqnarray*}
||H_t||_{2 \rightarrow 2}^2  
&=& \sup_{||f||_2=1} \int_X \left(\int_X f(y) h_t(x,y) dy \right)^2 dx.
\end{eqnarray*}
Note that $\int_X f(y)p_t(x,y) \le ||f||_2 ||p_t(x,\cdot)||_2$ holds
by H\"older.
\begin{eqnarray*}
...&\le& \sup_{||f||_2=1} \int_X ||f||_2^2 ||p_t(x,\cdot)||_2^2 dx \\
&=& \int_X ||p_t(x,\cdot)||_2^2 dx \\
&=& \int_X \int_X h_t(x,y)h_t(y,x) dy dx \\
&=& \int_X h_{2t}(x,x) dx \\
&\le& \left(\int_X h_{t'}(x,x) dx \right) e^{-2t/C}.
\end{eqnarray*}
This gives us $||H_t||_{2 \rightarrow 2} \le C' e^{-t/C}$ for $t \ge
t'$ where $C' = \int_X h_{t'}(x,x) dx$ depends only on $X$ and $t'$.

In the case where $G$ is not amenable, it is well known that the heat
kernel decays exponentially.  This result was shown by Kesten \cite{Kesten}.
In particular, for any $f \in \Dom(E)$, we know that:
\begin{eqnarray*}
||f||_{2,G} \le C_G || \grad f||_{2,G}.
\end{eqnarray*}
We can use averaging to show that this will hold on $X$ as well.
\begin{lemma} \label{NonAmenableht}
If $G$ is not amenable and $X/G=Y$, then for any $t'>0$ there exist constants
$C_0 = \sup_{y \in Y}h_{t'}(y,y)$ and $C_1 =\sqrt{C (\delta^2 + C_G^2
C(\delta))}$ so that for all $x,y \in X$
\begin{eqnarray*}
h_t(x,y) \le C_0 e^{-t/C_1}
\end{eqnarray*}
holds for all $t \ge t'$.  Note that $C, C(\delta)$ are as in Lemmas
\ref{fToGroupf} and \ref{gradGroupftogradf}.
\end{lemma}
\begin{proof}
Applying Lemma \ref{fToGroupf} with $p=2$, $r=\infty$ gives us:
\begin{eqnarray*}
||f||_{2,X}^2 
\le C \left(\delta^2 ||\grad f||_{2,X}^2 + ||\group{f}||_{2,G}^2\right).
\end{eqnarray*}
The inequality for groups then tells us this is less than
\begin{eqnarray*}
||f||_{2,X}^2 
\le C \left(\delta^2 ||\grad f||_{2,X}^2 
              + C_G^2 ||\grad(\group{f})||_{2,G}^2\right).
\end{eqnarray*}
We can then bound the gradient in $G$ by the gradient in $X$ using
Lemma \ref{gradGroupftogradf} with $p=2$, $r=\infty$:
\begin{eqnarray*}
||f||_{2,X}^2 
\le C \left(\delta^2 ||\grad f||_{2,X}^2 
                + C_G^2 C(\delta) ||\grad f||_{2,X}^2\right).
\end{eqnarray*}
Putting this together, we have:
\begin{eqnarray*}
||f||_{2,X} \le \sqrt{C (\delta^2 + C_G^2 C(\delta))} ||\grad f||_{2,X}.
\end{eqnarray*}
We can apply this with $C_1 =\sqrt{C (\delta^2 + C_G^2
C(\delta))}$ to the first argument to get $||H_t||_{2\rightarrow 2}
\le e^{-t/C_1}$ on our complex, $X$.  Then, apply Lemma \ref{Httopt}
to get the on-diagonal heat kernel bound for any fixed $z$.  
\begin{eqnarray*}
h_t(z,z) \le h_{t'}(z,z)e^{-t/C_1}.
\end{eqnarray*}
Because $X/G=Y$, we can shift $z$ by elements of $G$, and it won't
affect our heat kernel.  Specifically, this means
$h_t(z,z)=h_t(z+g,z+g)$ for any $g \in G$.  This allows us to consider
only values of $h_t(y,y)$ for points $y \in Y$. This tells us that the
supremum in $Y$ dominates: $\sup_{y \in Y} h_{t}(y,y) \ge h_t(z,z).$
Set $C_0 = \sup_{y \in Y} h_{t'}(y,y)$.  Because $Y$ is compact and
$h_{t'}(y,y)$ is continuous in $y$, for fixed $t'>0$ we will have
$C_0<\infty$.  As $\sup_{x,y} h_{t}(x,y) = \sup_y h_{t}(y,y)$, this
will give us our overall bound.
\end{proof}
\begin{corollary}\label{nonam}
If $G$ is not amenable and $X/G=Y$, then for $t\ge 1$
\begin{eqnarray*}
\sup_{x \in X} h_t(x,x) \approx p_{2\ceil{t}}(e,e).
\end{eqnarray*}
\end{corollary}
\begin{proof}
Kesten \cite{Kesten} showed that nonamenable groups have heat kernel
behavior $p_{2\ceil{t}}(e,e) \approx e^{-t/C}$.  By lemma
\ref{NonAmenableht}, we have $h_{t}(x,x) \le  c e^{-t/C}$.  Since 
\\ $\sup_{x \in X} h_{t}(x,x) \ge  c' e^{-t/C'}$, we have the equivalence.
\end{proof}
  \section{Heat Kernels in the Amenable Case}
This is a modified version of the argument in LSC-Pittet paper
\cite{LSCP} which shows that the on diagonal heat kernel on a group is
bounded above (in some sense) by the one on a manifold.  The basic
argument involves comparing eigenvalues and traces of the heat
equation restricted to a finite set.  We iterate through these sets
using F\o lner sequences, and then we compare the heat kernels
themselves.  
\subsection {Bounding those on G above by those on X}
\begin{theorem}\label{plessh}
Let $G$ be an amenable group and $X$ the associated complex.  For times
$t > 1$, we have constants $C, C_0$ so that
\begin{eqnarray*}
p_{\ceil{Ct}}(e,e) \le C_0 \sup_{x \in X} h_t(x,x).
\end{eqnarray*}
Here $C=2C_1C_2$ where $C_1$ and $C_2$ are the constants in
\ref{Compcompare} and\ref{CompGradcompare} and \\
$C_0= |S|^{\CompareXG \Rz / \min_{g\ne h} d_G(g,h)}$.
\end{theorem}
\begin{proof}
Let $A$ be a finite subset of $G$, and let $A_0$ be the set of points
in $X$ which surround it.  That is, $A_0 := \{x \in X | d(x,A) <\Rz
\}$.

Because $\Rz \ge \csup$, we will have functions $f: G\rightarrow R$
which are supported in $A$ map to functions $\comp f: X \rightarrow R$
which are supported in $A_0$.  Using lemmas \ref{Compcompare} and
\ref{CompGradcompare}, we know that $||f||_2^2 \le C_1
||\comp{f}||_2^2$ and $|| \grad \comp{f}||_2^2
\le C_2 E(f,f)$.  Combining these, we get:
\begin{eqnarray*}
\frac{|| \grad \comp{f}||_2^2}{||\comp{f}||_2^2} 
  &\le& C_1C_2 \frac{ E(f,f)}{||f||_2^2} \\ 
   &=& C_1C_2 \frac{ \langle I-K_A f,f \rangle}{||f||_2^2} \\
   &=&  C_1C_2 \left(1 - \frac{ ||K^{1/2}_A f||_2^2}{||f||_2^2}\right).
\end{eqnarray*}
Here, we used the fact that $K_A$ is self-adjoint.  We can apply the
min-max principle in order to compare eigenvalues. Let
$\lambda_{A_0}(i)$ be the ith eigenvalue for $H_t$ on $A_0
\subset X$ (denoted $H_t^{A_0}$) and $\beta_{A}(i)$ ith 
eigenvalue for $K$ on $A \subset G$ (denoted $K_A$).  For eigenvalues
$1..|A|$, we have:
\begin{eqnarray*}
\lambda_{A_0}(i) \le C_1C_2 (1 -\beta_{A}(i))
\end{eqnarray*}
We can rewrite this as:
\begin{eqnarray*}
\beta_{A}(i) \le 1 - \frac{1}{C_1C_2} \lambda_{A_0}(i).
\end{eqnarray*}
As $\lambda_{A_0}(i)$ will be bounded below by 0, we can use $1-x
\le e^{-x}$ to get:
\begin{eqnarray*}
\beta_{A}(i) \le e^{- \frac{1}{C_1C_2} \lambda_{A_0}(i)}.
\end{eqnarray*}
We can use this to compare the traces.  Recall
\begin{eqnarray*} 
\trace(H_t^{A_0}) &=& \sum_i e^{-t\lambda_{A_0}(i)} \\
\trace(K_A^n) &=& \sum_i \beta_{A}^n(i).
\end{eqnarray*}
When $\beta_{A}(i)\ge 0$, we have:
\begin{eqnarray*}
\beta_{A}^{2n}(i) \le e^{- \frac{1}{C_1C_2} \lambda_{A_0}(i)2n}.
\end{eqnarray*}
We will compare the negative $\beta_{A}(i)$ terms with the positive
ones.  We know that $0 \le \trace(K_A^{2n+1})$.  This means we can
split the sum into two pieces and subtract the part with negative
eigenvalues from both sides:
\begin{eqnarray*}
\sum_{\beta_{A}(i) <0} \abs{\beta_{A}^{2n+1}(i)} 
\le
\sum_{\beta_{A}(i) >0} \beta_{A}^{2n+1}(i). 
\end{eqnarray*}
Since all of the eigenvalues are between -1 and 1, we have:
\begin{eqnarray*}
\sum_{\beta_{A}(i) <0} \abs{\beta_{A}^{2n+2}(i)}
\le \sum_{\beta_{A}(i) <0} \abs{\beta_{A}^{2n+1}(i)}
\le \sum_{\beta_{A}(i) >0} \beta_{A}^{2n+1}(i) 
\le \sum_{\beta_{A}(i) >0} \beta_{A}^{2n}(i) .
\end{eqnarray*}
This tells us that 
\begin{eqnarray*}
\trace(K_A^{2n+2}) \le \sum_{\beta_{A}(i)} \abs{\beta_{A}^{2n+2}(i)}
 \le 2 \sum_{\beta_{A}(i) >0} \beta_{A}^{2n}(i)
 \le 2 \sum_i \beta_{A}^{2n}(i).
\end{eqnarray*}

We can compare the first $|A|$ terms in the two sums, and the extra
terms in $\trace(H_t^{A_0})$ will only help us: 
\begin{eqnarray*}
\trace(K_A^{2n+2}) \le 2\trace(H_{\frac{2n}{C_1C_2}}^{A_0}).  
\end{eqnarray*}

We are now in a good spot.  We will compare the heat kernels with the
respective traces.  Fix $n$.  Let $F(i)$ be a F\o lner sequence in
$G$, and recall $S^n$ is the set of words in $G$ of length at most
$n$.  For each $i$ we will have a set \\ $A = S^nF(i)= \{ g : g = f u, f
\in F(i), u \in S^n \}$.  In Lsc-Pittet \cite{LSCP}, they showed that
for an amenable group $G$ we have the comparison:
\begin{eqnarray*}
p_{2n+2}(e,e) \le \frac{1}{|F(i)|} \trace(K_A^{2n+2}).
\end{eqnarray*}
By the definition of the trace, we know that on the complex we have:
\begin{eqnarray*} 
\trace(H_t^{A_0}) &=& \sum_i e^{-t\lambda_{A_0}(i)} \\
                 &=& \int_{A_0} h^{A_0}_t(x,x) dx \\
                 &\le& \mu(A_0) \sup_{x\in A_0} h^{A_0}_t(x,x) \\
                 &\le& \mu(A_0) \sup_{x\in X} h_t(x,x).
\end{eqnarray*}
When we combine these, we find that:
\begin{eqnarray*}
p_{C_1C_2(2n+2)}(e,e) \le 
  \frac{\mu(A_0)}{|F(i)|}\sup_{x\in X} h_{2n}(x,x).
\end{eqnarray*}
We can compare $\mu(A_0)$ with $\Vol_G(A)$.  Since $A_0 := \{x \in X
| d_X(x,A) <\Rz \}$, each element in $A$ can expand to at most
$|S|^{\CompareXG \Rz / \min_{g\ne h} d_G(g,h)}$ new elements in $A_0$.
This tells us:
\begin{eqnarray*}
p_{C_1C_2(2n+2)}(e,e) \le 
 |S|^{\CompareXG \Rz / \min_{g\ne h} d_G(g,h)}
   \frac{|A|}{|F(i)|}\sup_{x\in X} h_{2n}(x,x).
\end{eqnarray*}
We can now let $i$ go to infinity; since we have a F\o lner sequence,
$\frac{|A|}{|F(i)|} =\frac{|S^n F(i)|}{|F(i)|} $ will become $1$.  This
leaves us with:
\begin{eqnarray*}
p_{C_1C_2(2n+2)}(e,e) \le 
 |S|^{\CompareXG \Rz / \min_{g\ne h} d_G(g,h)}
   \sup_{x\in X} h_{2n}(x,x).
\end{eqnarray*}
\end{proof}
  \subsection{Bounding those on X above by those on G}
We'd like to show the reverse inequality.  We will do this using a
chain of comparisons.  First, we will compare $h_t(x,x)$ with
$h_t^W(x,x)$, where $h_t^W$ is the diffusions in an open subset $W
\subset X$.  Then we will compare eigenvalues of $h_t^W(x,x)$ and
$p_t^{W'}(e,e)$, where $p_t^{W'}(e,e)$ represents probability of a
random walk restricted to a set $W' \subset G$ returning to the
identity, using our bounds on norms and minimax inequalities.  Lastly,
we use a comparison for $p_t^{W'}(e,e)$ and $p_t(e,e)$.  At this
point, we will remove some of the dependence on $W$, and limit away
other factors to get the final result.

We would like to look at what happens to diffusions in an open subset
$W \subset X$.  Let $\tau$ be the exit time for this set: $\tau = \inf
\{t : t\ge 0, X_t \in W \}$.  Then by the strong Markov property we
have a restricted heat kernel:
\begin{eqnarray*}
h_t^W(x,y) =h_t(x,y) -E^x(h_{t-\tau}(X_{\tau},y)1_{\tau \le t}).
\end{eqnarray*}
  Here, $X_{t}$ is a random variable which at time $t=\tau$ will be
the point on $\bdry W$ where $X_t$ exits $W$.  The term
$h_{t-\tau}(X_{\tau},y)$ represents going from the point on the
boundary to $y$ in the time $t-\tau$ which is left after exiting $W$.
We take the expected value of this where $X_0 = x$.  We can bound the
expected value above by the maximum value. Since
$E^x(h_{t-\tau}(X_{\tau},y)1_{\tau \le t}) \le
\sup_{0<s<t}\sup_{z \in \bdry W} h_s(z,y)$, we have
\begin{eqnarray*}
h_t^W(x,y) \ge h_t(x,y) - \sup_{0<s<t}\sup_{z \in \bdry W} h_s(z,y).
\end{eqnarray*}

We can use this to bound $h_t^W(x,x)$ below for $x$ sufficiently far
from $\bdry W$.
\begin{lemma}\label{htWbysupht}
There exists a constant $C_H$ so that for all $\varepsilon_1 > 0$ there
exists $a>0$ so that for all open subsets $W \subset X$ and for all $t \ge
6r_1^2$ we know that
\begin{eqnarray*}
h_t^W(x,x) \ge C_H^{-1} \sup_{y \in X} h_{t-3r_1^2}(y,y) -\varepsilon_1.
\end{eqnarray*}
for all $x \in \{x \in W : d(x,\bdry W)> at^{1/2} \}$.  Here, $r_1=\diam(Y)$.
\end{lemma}
\begin{proof}
By Corollary \ref{offdiagonalHeat} we know that there are constants $C_1$ and
$C_2$ so that
\begin{eqnarray*}
h_t(x,y) \le \frac{C_1}{\min(t,1)^{d/2}} e^{-C_2\frac{d^2(x,y)}{t}}.
\end{eqnarray*}

This estimate allows us to bound $\sup_{0<s<t}\sup_{z \in \bdry W}
h_s(z,y)$ whenever $y$ is at a distance at least $at^{1/2}$ away from
the boundary of $W$.  If $s \le 1$, then \\ $h_s(z,y) \le
\frac{C_1}{s^{d/2}} e^{-C_2\frac{a^2t}{s}}$. This has a maximum at
$s=\frac{2a^2 C_2 t}{d}$ which tells us:
\begin{eqnarray*} 
h_s(z,y) \le C_1 \left(\frac{d}{2a^2 C_2 t}\right)^{d/2} e^{-d/2}.  
\end{eqnarray*}
For $t \ge 6r_1^2$, this is maximized at $t=6r_1^2$.  We have $C_1
(\frac{d}{12r_1^2 C_2 e})^{d/2} a^{-d} < \varepsilon_1$ when $a >
C_1^{1/d} (\frac{d}{12r_1^2 C_2 e})^{1/2} \varepsilon_1^{-1/d}$.

If $t>s>1$, then the maximum occurs when $s=t$:
\begin{eqnarray*}
h_s(x,y) \le C_1 e^{-C_2\frac{a^2 t}{s}} \le C_1
e^{-C_2 a^2}.  
\end{eqnarray*}
We know that $C_1 e^{-C_2 a^2} < \varepsilon_1$ whenever $a >
\sqrt{\frac{1}{C_2} \ln(\frac{C_1}{\varepsilon_1})}$.

Thus, whenever $a > \max\left(C_1^{1/d} (\frac{d}{12r_1^2 C_2e})^{1/2}
\varepsilon_1^{-1/d},
\sqrt{\frac{1}{C_2} \ln(\frac{C_1}{\varepsilon_1})}\right)$
we have 
\begin{eqnarray*}
\sup_{0<s<t}\sup_{z \in \bdry W} h_s(z,y) < \varepsilon_1.
\end{eqnarray*}

We can bound $h_t(x,x)$ below using a parabolic Harnack inequality.
Theorem 3.5 in Sturm \cite{Sturm} uses techniques in Moser
\cite{Moser} to show that Poincar\'{e} and volume doubling locally
imply a parabolic Harnack inequality.  In our situation, we have
uniformly bounded constants for both local Poincar\'{e} and volume
doubling, and so the constant $C_H$ in the Harnack inequality will
also be uniform.

In the language of Sturm: \\
For all $K\ge 1$ and all $\alpha, \beta,
\gamma, \delta$ with $0 < \alpha < \beta < \gamma < \delta$ and $0 <
\varepsilon < 2$ there exists a constant $C_H = C_H(Y_1)$ such that for
balls $B_{2r}(x) \subset Y_1$ and all $T$,
\begin{eqnarray*}
\sup_{(s,y) \in Q^-} u(s,y) \le C_H \inf_{(s,y) \in Q^+} u(s,y)
\end{eqnarray*}
whenever $L_T$ is a uniformly parabolic operator whose associated
Dirichlet form is comparable by a factor of $K$ with the original
Dirichlet form, and $u$ is a nonnegative local solution of the parabolic
equation $(L_T -
\frac{\del}{\del T}) u =0$ on $Q = (T-\delta r^2,T) \times B_{2r}(x)$.
Here $Q^- = (T-\gamma r^2,T-\beta r^2) \times B_{\varepsilon r}(x)$ and
$Q^+ = (T-\alpha r^2,T)
\times B_{\varepsilon r}(x)$.

We can translate this language to our situation.  For us, $L_T =
\Delta$, and so there is no $T$ dependence in the operator.  This 
means the Dirichlet form condition will be trivially satisfied when
$K=1$.  We also will take $\varepsilon =1$.  We will set $Y_1 =
B_{3r_1}$.  This is a ball which is large enough so that every
equivalence class of $x \in Y$ has a representative in $Y_1$, as well as
an associated copy of $Y$ in $Y_1$.
When $t > 6r_1^2$, we can set $T = t + r^2$, $\alpha =1 $, $\beta =2$,
$\gamma = 4$, and $\delta =5$.  Then $Q^+ = (T-r^2, T+r^2)$, $Q^- =
(T-4r^2,T -2r^2)$, and $Q = (T-5r^2, T+r^2)$.  Applying Sturm here gives us
\begin{eqnarray*}
\sup_{y \in B_r(x)} h_{t-3r^2}(y,y) 
  \le C_H \inf_{y \in B_r(x)} h_t(y,y) 
  \le C_H h_t(x,x).
\end{eqnarray*}
Due to the symmetry of the space $X$, $h_{s}(y,y)$ is the same as
$h_{s}$ when $y$ is translated by an element of $G$.  For $r =
\diam(Y)$, we have a copy of \\ 
$Y \subset B_r(x) \subset B_{2r}(x) \subset Y_1$ 
for every $x \in Y$.  This tells us that 
\begin{eqnarray*}
\sup_{y \in B_r(x)} h_{s}(y,y) =\sup_{y \in X} h_{s}(y,y).
\end{eqnarray*}
\end{proof}
We can bound the integral of $h_t^W(x,x)$ above by an analogue of
Lemma 5.3 in LSC-Pittet \cite{LSCP}.
\begin{lemma}\label{htWbylambdale1}
For subsets $W \subset X$, $B >0$, and $t\ge 1$,
\begin{eqnarray*}
\int_W h_t^W(x,x) dx \le 
 \sum_{\lambda_W(i) \le 1/B} e^{-t \lambda_W(i)} + C_1 2^{d/2}
 \mu(W)e^{-t/(2B)}.
\end{eqnarray*}  
Here, $C_1$ and $d$ are defined as in Corollary \ref{offdiagonalHeat},
and $\lambda_W$ are the eigenvalues of $h_t^W$.
\end{lemma}
\begin{proof}
When $a,b \ge 1$ we have the inequality $ab \ge a/2 +b/2$.  Let $a=t$
and $b=B/\lambda_W(i)$.  Then for $t \ge 1$ and $\lambda_W(i) \ge
1/B$ we have 
\begin{eqnarray*}
tB/\lambda_W(i) \ge t/2 + B/(2\lambda_W(i)).
\end{eqnarray*}
  If we multiply through by $-1/B$ and exponentiate we find
\begin{eqnarray*}
e^{-t \lambda_W(i)} \le e^{-t/(2B) - \lambda_W(i)/2}.
\end{eqnarray*}
  This allows us to bound the sum over the larger eigenvalues:
\begin{eqnarray*}
\sum_{\lambda_W(i) \ge 1/B}  e^{-t \lambda_W(i)}
&\le& \sum_{\lambda_W(i) \ge 1/B}  e^{-t/(2B) - \lambda_W(i)/2} \\
&\le& e^{-t/(2B)} \sum_{\lambda_W(i)}  e^{- \lambda_W(i)/2} \\
&=& e^{-t/(2B)}\int_W h_{1/2}^W(x,x) dx \\
&\le&  e^{-t/(2B)}  C_1 2^{d/2} \mu(W).
\end{eqnarray*}

In the last step, we used the bound in \ref{offdiagonalHeat} which
tells us $h_{1/2}^W(x,x) \le C_1 2^{d/2}$.

Using the eigenvalue expansion, we can compare the integral of the
heat kernel at times greater than one with the sum over small
eigenvalues plus our bound on the sum over larger eigenvalues:
\begin{eqnarray*}
\int_W h_t^W(x,x) dx &=& \sum_{\lambda_W(i)}  e^{-t \lambda_W(i)} \\
&\le& \sum_{\lambda_W(i) \le 1/B} e^{-t \lambda_W(i)} 
    + C_1 2^{d/2} \mu(W)e^{-t/(2B)}.
\end{eqnarray*}
\end{proof}
Let's consider what it means to have a Laplacian, $\Delta^{\Omega}$,
defined for functions restricted to a set, $\Omega$ with a polygonal
boundary.  Let the domain of $\Delta^{\Omega}$ be the closure of the
intersection of $\Dom(\Delta)$ and the continuous functions which are
compactly supported on $\Omega$; that is, $\Dom(\Delta^{\Omega}) =
\overline{\Dom(\Delta) \cap C^C_0(\Omega)}$.  Note that since
$\Dom(\Delta) \cap C_0(\Omega)
\subset \Dom(\Delta)$ and $\Dom(\Delta)$ is closed, we know that
\\ $\Dom(\Delta^{\Omega}) \subset \Dom(\Delta)$.

For functions $f \in \Dom(\Delta^{\Omega})$, we set $\Delta^{\Omega} f
= \Delta f$.  $\Delta^{\Omega}$ inherits many properties from
$\Delta$.  It is self-adjoint with a discrete spectrum, and as we will
see in the following lemma, for the $\Omega$ that we are interested in
there will be only finitely many eigenvalues which are close to $0$.

 We can show this by comparing operators restricted to subsets of $X$
to operators restricted to subsets of $G$.  Let $A \subset G$ be
given.  Let $\Omega = U(A)$ be a subset of $X$ with polygonal boundary so
that any function $f$ whose support is in $U(A)$ has an associated
function $\group{f}$ whose support is in $(A, \{1..N\})$.  In
particular, we would like $U(A)$ to be close in size to $A$.  Since
$\group f(g,i) = \dashint_{B_X(g\gamma_i,\delta)} f(x) dx $ averages
over neighborhoods of points in $X$, we can guarantee a set with
volume estimate:
\begin{eqnarray*}
\min_{y \in Y} \mu(B(y, \delta)) N \# A \le \mu(U(A)) \le \mu(Y) N \#A.
\end{eqnarray*}
The following lemma will give us a comparison for small eigenvalues on
$h_t^{U(A)}$.
\begin{lemma}\label{htvsKAN}
Let $A \subset G$ and $U(A) \subset X$ be given as above.  Eigenvalues
of $h_t^{U(A)}(x,x)$ and $p_n(e,e)$ are comparable in the
following manner:
\begin{eqnarray*}
\sum_{i :0 \le \lambda_{U(A)}(i) \le 1/B} e^{-2 n \lambda_{U(A)}(i)} 
\le \#A N p_{2\floor{n/(B(2-\sqrt{2}))}}(e,e).
\end{eqnarray*}
Here $B=4 C(\sqrt{\frac{1}{2C_{grad}}})/(2-\sqrt{2})$ where $C_1$ is
the constant in Lemma \ref{fToGroupf} and $C(\cdot)$ is the constant
from Lemma \ref{gradGroupftogradf}. $N$ depends on $\delta =
\sqrt{\frac{1}{2 C_{grad}}}$.
\end{lemma}
\begin{proof}
Suppose $u$ is a solution to $\Delta^{\Omega} u = \lambda u$ on a set $\Omega
\subset X$ with polygonal boundary, and $u=0$ on $\bdry \Omega$.  Set $u=0$ 
outside of $\Omega$.  For $u\in \Dom(\Delta^{\Omega})$ and $\lambda
\ne 0$, a formal argument using integration by parts tells us:
\begin{eqnarray*}
\langle u,u \rangle &=& \frac{1}{\lambda} \langle \lambda u, u \rangle \\
 &=& \frac{1}{\lambda} \langle \Delta^{\Omega} u, u \rangle \\
 &=& \frac{1}{\lambda} \langle \grad u, \grad u \rangle  
    + \langle \grad u, u \rangle |_{\bdry \Omega} \\
 &=& \frac{1}{\lambda} \langle \grad u, \grad u \rangle .
\end{eqnarray*}

This gives us $||u||_2 = |\frac{1}{\lambda}| ||\grad u||_2$.  We know
that such eigenfunctions exist because $\Delta^{\Omega}$ is self-adjoint.

We will combine this with the inequality in Lemma
\ref{fToGroupf} for eigenfunctions $f$ on the set $U(A)$:
\begin{eqnarray*}
||f||_{2,X}^2 &\le& C_{grad}(\delta^2 ||\grad f||_{2,X}^2 +||\group{f}||_{2,G}^2)
 \\
& =&  C_{grad} (\delta^2 |\lambda| ||f||_{2,X}^2 +||\group{f}||_{2,G}^2).
\end{eqnarray*}
This tells us:
\begin{eqnarray*}
(1- C_{grad} \delta^2 |\lambda|)||f||_{2,X}^2 \le ||\group{f}||_{2,G}^2.
\end{eqnarray*}
If $\delta$ is less than $\sqrt{\frac{1}{\lambda C_{grad}}}$, we have
a nice bound for that $\lambda$.  In particular, $\delta =
\sqrt{\frac{1}{2C_{grad}}}$ gives us a simple bound for all 
$\lambda \le 1$ because $1- (1/2) |\lambda| > 1/2$.
\begin{eqnarray*}
||f||_{2,X}^2 \le 2 ||\group{f}||_{2,G}^2.
\end{eqnarray*}
Lemma \ref{gradGroupftogradf} tells us
\begin{eqnarray*}
||\grad \group{f} ||_{2,G}^2 \le  C(\delta) ||\grad f||_{2,X}^2.
\end{eqnarray*}
We have that for $C' = 2 C(\sqrt{\frac{1}{2C_{grad}}})$:
\begin{eqnarray*}
\frac{||\grad \group{f} ||_{2,G}^2}{||\group{f}||_{2,G}^2} 
  \le C' \frac{||\grad f||_{2,X}^2}{||f||_{2,X}^2}.
\end{eqnarray*}
We can rewrite $\grad \group{f}$ in terms of $K_{A,N}^{1/2}$.
\begin{eqnarray*}
1 - \frac{||K_{A,N}^{1/2} \group{f}||_{2,G}^2}{||\group{f}||_{2,G}^2} 
  \le C' \frac{||\grad f||_{2,X}^2}{||f||_{2,X}^2}.
\end{eqnarray*}
This will allow us to compare the first $k$ eigenvalues of
$h^{U(A)}_t$ with the absolute values of those for $K_{A,N}$, where $k
= \min(\#NA, \#\{\lambda_{U(A)}(i) \in [0,1]\})$.  The min-max definition
will give us these eigenvalue comparisons.  For simplicity, we will
use $\lambda_{U(A)}(i)$ to refer to the ith smallest eigenvalue of
$h^{U(A)}_t$, and $|\beta_A(i)|$ to refer to the ith largest absolute
value of the eigenvalue of $K_{A,N}$.  We have
\begin{eqnarray*}
1- |\beta_{A}(i)| &\le& C' \lambda_{U(A)}(i) \text{ which can be written as}\\
1- C' \lambda_{U(A)}(i) &\le& |\beta_{A}(i)|.
\end{eqnarray*}
When $1/2 \le x \le 1$, we know $x \ge e^{-2(1-x)}$.  Applying that to $x=
1-C'\lambda_{U(A)}(i)$, we have
\begin{eqnarray*}
e^{-2C'\lambda_{U(A)}(i)} \le |\beta_{A}(i)|
\end{eqnarray*}
 for $i \le k$ with  $0 \le \lambda_{U(A)}(i) \le 1/(2C')$.  
We can exponentiate to get:
\begin{eqnarray*}
e^{-2 n \lambda_{U(A)}(i)} \le |\beta_{A}(i)|^{n/C'}.
\end{eqnarray*}
We will have this bound for all of the $\lambda_{U(A)} \in
[0,(2-\sqrt{2})/(2C')]$ provided we can show that we have an $i$ with
$C'\lambda_{U(A)}(i) > (2-\sqrt{2})/2$.  If we knew that
$(2-\sqrt{2})/2 \le 1- |\beta_{A}(i)|$ for some $i$, then this would
be shown.  This means we want to have $|\beta_{A}(i)|^2 \le 1/2$ for
some $i$.  We know that $K_{A,N}$ is an $\#AN$ by $\#AN$ matrix whose
entries are either $1/|S|$ or $0$ and that there are $|S|$ nonzero
entries per row.  When we look at its square, we have another $\#AN$
by $\#AN$ matrix whose entries are at most $|S|/|S|^2 = 1/|S|$ and at
least $0$.  $K_{A,N}^2$ has eigenvalues $|\beta_A(i)|^2$.  This means
that the largest $\trace(K_{A,N}^2)$ could possibly be is $\#AN/|S|$,
and so $\sum_{i=1}^{\#AN} |\beta_A(i)|^2 \le \#AN/|S|$.  The average
value of an eigenvalue $|\beta_{A}|^2$ is $1/|S|$.  Since
$|\beta_A(i)|^2 \in [0,1]$, we must have at least one $|\beta_A|^2$
which is smaller than $1/|S|$ in order to have that as the average.
This tells us that there is some $i$ with $|\beta_A(i)| \le
1/\sqrt{|S|} \le 1/\sqrt{2}$.

In this way, we have guaranteed the bound for all $\lambda_{U(A)} \in
[0,(2-\sqrt{2})/(2C')]$.  Note that this also shows that there are at
most $\#AN$ such eigenvalues.

Summing over $\lambda_{U(A)}(i) \in [0,(2-\sqrt{2})/(2C')]$ gives us:
\begin{eqnarray*}
\sum_{i :0 \le \lambda_{U(A)}(i) \le (2-\sqrt{2})/(2C')} 
      e^{-2 n \lambda_{U(A)}(i)} 
\le \sum_{i :0 \le \lambda_{U(A)}(i) \le (2-\sqrt{2})/(2C')} 
      (\beta_{A}(i))^{n/C'}.
\end{eqnarray*}
Note that the $\beta_{A}(i)$ in this sum are positive.  We can
compare these to positive eigenvalues in the trace by using
$K^{2\floor{n/(2C')}}_{A,N}$.  
\begin{eqnarray*}
\sum_{i : 0 \le \lambda_{U(A)}(i) \le (2-\sqrt{2})/(2C')} (\beta_{A}(i))^{n/C'}
\le \trace (K_{A,N}^{2\floor{n/(2C')}}).
\end{eqnarray*}
Combining these yields:
\begin{eqnarray*}
\sum_{i :0 \le \lambda_{U(A)}(i) \le (2-\sqrt{2})/(2C')} 
      e^{-2 n \lambda_{U(A)}(i)} 
\le \trace (K_{A,N}^{2\floor{n/(2C')}}). 
\end{eqnarray*}
We know that by its definition
\begin{eqnarray*}
\trace (K_{A,N}^{2\floor{n/(2C')}}) 
&=& \sum_{g \in A, j = 1..n} p_{2\floor{n/(2C')}}(g,g) \\
&\le& \#A N p_{2\floor{n/(2C')}}(e,e).
\end{eqnarray*}
This gives us the result:
\begin{eqnarray*}
\sum_{i :0 \le \lambda_{U(A)}(i) \le (2-\sqrt{2})/(2C')} 
     e^{-2 n \lambda_{U(A)}(i)} 
\le \#A N p_{2\floor{n/(2C')}}(e,e).
\end{eqnarray*}
If we want to simplify the notation on the left, we may set $B=
2C'/(2-\sqrt{2})$.  This means $n/(2C') = n/(B(2-\sqrt{2}))$.  Hence:
\begin{eqnarray*}
\sum_{i :0 \le \lambda_{U(A)}(i) \le 1/B} 
     e^{-2 n \lambda_{U(A)}(i)} 
\le \#A N p_{2\floor{n/(B(2-\sqrt{2}))}}(e,e).
\end{eqnarray*}
\end{proof}
\begin{theorem}\label{hlessp}
For $t> 6r_1^2$  we get:
\begin{eqnarray*}
\sup_{y \in X} h_{t-3r_1^2}(y,y) 
\le  C p_{2\floor{\frac{t}{B \log |S|}}}(e,e)
\end{eqnarray*}
where $C = C_H \left(\frac{1}{\min_{y \in Y} \mu(B(y, \delta))} +
\frac{\mu(Y)}{\mu(B(y, \delta))} C_1 2^{d/2}\right)$.
\end{theorem}
\begin{proof}
We'll use these lemmas and F\o lner sequences to build this
inequality.  

Recall Lemma \ref{htWbysupht} told us:
\begin{eqnarray*}
h_t^W(x,x) 
\ge C_H^{-1} \sup_{y \in X} h_{t-3r_1^2}(y,y) -\varepsilon_1.
\end{eqnarray*}
for all $x \in \{x \in W : d(x,\bdry W)> at^{1/2} \}$ when $t >
6r_1^2$.  
\\ Set $T = \{g \in G : d_X(e,g) \le \sqrt{t} a + 10 R_0 \}$. 
\\ Then $AT = \{g \in G : g=th$ for $t \in T, h \in A \}$.  We'll apply
this to $W=U(AT)$.  
\\ Note that $U(A) \subset \{x \in U(AT) : d(x,\bdry U(AT))> at^{1/2} \}$.

When we take the average over $U(A)$ we have:
\begin{eqnarray*}
\sup_{y \in X} h_{t-3r_1^2}(y,y) 
 \le C_H \left(\dashint_{U(A)} h_t^{U(AT)}(x,x)dx + \varepsilon_1 \right)
\end{eqnarray*}

From Lemma \ref{htWbylambdale1} we know how to bound the integral in
terms of $\lambda_W \le 1/B$:
\begin{eqnarray*}
\int_{U(A)} h_t^{U(AT)}(x,x) dx 
&\le& \int_{U(AT)} h_t^{U(AT)}(x,x) dx \\ 
&\le& \sum_{\lambda_{U(AT)}(i) \le 1/B} e^{-t \lambda_{U(AT)}(i)} 
    + C_1 2^{d/2} \mu({U(AT)})e^{-t/(2B)}.
\end{eqnarray*}

Putting them together gives us:
\begin{equation*}
\begin{split}
\sup_{y \in X} & \text{ }h_{t-3r_1^2}(y,y) \\
&\le C_H \left(\frac{1}{\mu({U(A)})}
         \sum_{\lambda_{U(AT)}(i) \le 1/B} e^{-t \lambda_{U(AT)}(i)}
       + \frac{\mu(U(AT))}{\mu(U(A))} C_1 2^{d/2} e^{-t/(2B)} + \varepsilon_1
        \right).
\end{split}
\end{equation*}
By Lemma \ref{htvsKAN} we have:
\begin{eqnarray*}
\sum_{i :0 \le \lambda_{U(AT)}(i) \le 1/B} 
     e^{-2 n \lambda_{U(AT)}(i)} 
\le \#(AT) N p_{2\floor{\frac{n}{B(2-\sqrt{2})}}}(e,e).
\end{eqnarray*}
When we set $n=t/2$, this gives us:
\begin{equation*}
\begin{split}
\sup_{y \in X} &  \text{ } h_{t-3r_1^2}(y,y) \\
& \le C_H \left(\frac{\# (AT) N}{\mu(U(A))}
         p_{2\floor{\frac{t}{2B(2-\sqrt{2})}}}(e,e)
      + \frac{\mu(U(AT))}{\mu(U(A))} C_1 2^{d/2} e^{-t/(2B)} + \varepsilon_1
        \right).
\end{split}
\end{equation*}
On $G$, we can bound below the probability of returning to the start
by noting that because $S=S^{-1}$, after moving $n$ steps, we have a
$\frac{1}{|S|^n}$ chance of exactly retracing our path.  
\begin{eqnarray*}
p_{2n}(e,e) \ge \frac{1}{|S|^n} = e^{-n\log |S|}.
\end{eqnarray*}
  A more convenient time gives us 
\begin{eqnarray*}
p_{2\floor{\frac{n}{2B \log |S|}}}(e,e) \ge e^{-n /(2B)}.
\end{eqnarray*}
When we place this into the inequality, we have:
\begin{equation*}
\begin{split}
&\sup_{y \in X} h_{t-3r_1^2}(y,y) \\
&\le C_H \left(\frac{\#(AT) N}{\mu({U(A)})} 
      p_{2\floor{\frac{t}{2B(2-\sqrt{2})}}}(e,e) 
   + \frac{\mu(U(AT))}{\mu(U(A))}C_1 2^{d/2} p_{2\floor{\frac{t}{2B \log|S|}}}(e,e) 
   + \varepsilon_1 \right) .
\end{split}
\end{equation*}
We can use the fact that $p_t(e,e) \le p_s(e,e)$ whenever $t>s$ noting
that both $2\floor{\frac{t}{2B(2-\sqrt{2})}}$ and $2\floor{\frac{t}{2 B \log|S|}}$
are larger than $2\floor{\frac{t}{B \log |S|}}$.
\begin{eqnarray*}
...\le  C_H \left( \left(\frac{\#(AT) N}{\mu({U(A)})}
            + \frac{\mu(U(AT))}{\mu(U(A))} C_1 2^{d/2}\right) 
    p_{2\floor{\frac{t}{B \log |S|}}}(e,e)) + \varepsilon_1\right) .
\end{eqnarray*}

We take a F\o lner sequence for $G$, and set $A=F(i)$.  We can
use our volume estimates to find:
\begin{eqnarray*}
\frac{\#(AT) N}{\mu({U(A)})} \le 
\frac{\#(AT) N}{\min_{y \in Y} \mu(B(y, \delta)) N \# A}
= \frac{\#(AT)}{\# A \min_{y \in Y} \mu(B(y, \delta))}
\end{eqnarray*}
and 
\begin{eqnarray*}
\frac{\mu(U(AT))}{\mu(U(A))} \le 
\frac{\mu(Y) N \#(AT)}{\min_{y \in Y} \mu(B(y, \delta)) N \# A}
= \frac{\#(AT) \mu(Y)}{\# A \min_{y \in Y} \mu(B(y, \delta))}.
\end{eqnarray*}
When we take the limit of $\frac{\#(AT)}{\# A} =
\frac{\#(F(i)T)}{\#F(i)}$ as $i \rightarrow \infty$, we find it is $1$.

This gives us:
\begin{eqnarray*}
\sup_{y \in X} h_{t-3r_1^2}(y,y)
\le  C_H 
\frac{1 + \mu(Y)}{\min_{y \in Y} \mu(B(y, \delta))} 
C_1 2^{d/2}
 p_{2\floor{\frac{t}{B \log |S|}}}(e,e) 
+ C_H \varepsilon_1.
\end{eqnarray*}
Now let $\varepsilon_1$ go to zero.  This yields the comparison.
\end{proof}

We can combine these three results into a single theorem.
\begin{theorem}
Let $G$ be a finitely generated group and $X$ the associated complex.
For times $t > 1$, we have the comparison 
\begin{eqnarray*}
p_{2\ceil{t}}(e,e) \approx \sup_{x \in X} h_t(x,x).
\end{eqnarray*}
Note that by transitivity, this holds for the heat kernels on the
skeletons as well.
\end{theorem}
\begin{proof}
If $G$ is amenable, apply theorems \ref{plessh} and \ref{hlessp}.  If $G$ is
nonamenable, apply corollary \ref{nonam}.
\end{proof}
This theorem gives a comparison of heat kernel behavior at large
times.  It does not; however, tell you what that behavior is for a
given group.  Even though the proof tells you the asymptotic for
nonamenable groups, it is not easy to determine amenability.  For
example, it is unknown whether Thompson's group $F$ is amenable or
not. (See Belk \cite{Belk}.)

\bibliography{biblio_thesis}
\bibliographystyle{plain}

\end{document}